\documentclass{amsart}
\usepackage{amssymb,latexsym,graphics,enumerate,color}
\usepackage[normalem]{ulem}
\usepackage[mathscr]{eucal}
\usepackage{fullpage}

\newtheorem{thm}{Theorem}[section]
\newtheorem{cor}[thm]{Corollary}
\newtheorem{lem}[thm]{Lemma}
\newtheorem{prop}[thm]{Proposition}

\theoremstyle{remark}
\newtheorem{rem}[thm]{Remark}
\newtheorem{defn}[thm]{Definition}
\newtheorem{assum}[thm]{Assumption}

\newcommand{\lp}[2]{\Vert \, #1 \, \Vert_{#2}}
\newcommand{\slp}[2]{\Vert  #1  \Vert_{#2}}

\newcommand{\LLp}[2]{|\!|\!| \, #1\,  |\!|\!|_{#2} }
\newcommand{\sLLp}[2]{|\!|\!|  #1  |\!|\!|_{#2} }

\newcommand{\td}{\widetilde}

\newcommand{\snabla}{ {\, {\slash\!\!\! \nabla  }} }

\newcommand{\bL}{  {\underline{L}} }
\newcommand{\CE}{ {E} }

\newcommand{\WLE}{ {W\!L\!E} }
\newcommand{\dint}{ \int_0^T\!\!\!\!\!\int_{\mathbb{R}^3\setminus\mathcal{K}} }
\newcommand{\bu}{ {\underline{u}} }
\newcommand{\la}{\langle}
\newcommand{\ra}{\rangle}
\newcommand{\NLE}{ {N\!L\!E} }
\newcommand{\LE}{ {L\!E} }

\newcommand{\Part}[2]{\noindent\textbf{Part #1:}(\emph{#2})}
\newcommand{\case}[2]{\noindent\textbf{Case #1:}(\emph{#2})}
\newcommand{\step}[2]{\noindent\textbf{Step #1:}(\emph{#2})}

\begin{document}

\title[Vector Fields]
{A Vector Field Method for Radiating 
Black Hole Spacetimes}

\author{Jes\'us Oliver}
\address{Department of Mathematics, California State University, East Bay (CSUEB) 25800 Carlos Bee Blvd, Hayward, 
CA 94542 USA}
\email{jesus.oliver@csueastbay.edu}
\author{Jacob Sterbenz}
\address{Department of Mathematics,
University of California, San Diego (UCSD)
9500 Gilman Drive, Dept 0112
La Jolla, CA 92093-0112 USA}
\email{ jsterben@math.ucsd.edu}

\thanks{}
\date{}
\dedicatory{}
\commby{}


\begin{abstract}
We develop a  commuting vector field method for a  general class of   radiating spacetimes.
The   metrics we consider  include  those constructed from global stability
problems in general relativity, and  our method provides sharp peeling estimates for solutions to both linear and (null form)
nonlinear  scalar fields.
\end{abstract}

\maketitle
\tableofcontents


\section{Introduction}

In this paper we develop a sharp variant of Klainerman's vector field method for solutions to scalar wave equations
on a generic class of asymptotically flat
spacetimes. These are taken to be certain
long range perturbations of Minkowski space
which enjoy a standard local energy decay assumption.

The first main  consideration here
is to place the minimal conditions on  our metrics at null infinity which are compatible 
with gravitational  radiation, but at the same time  are also strong enough 
to provide full peeling estimates for   scalar waves. 
Our conditions   turn out to be natural even if one 
is only interested in stationary long range perturbations of Minkowski space, because they highlight  
certain peeling properties of  Lorentzian metrics at null infinity which appear to be necessary in order to produce
estimates on the order of the classical Morawetz conformal energy.

The  second main consideration here  is to produce a collection of norms which are natural for
studying nonlinear stability problems, at least when the quadratic part of nonlinearity enjoys 
a certain generalized null condition. In fact, we will produce a  range of   norms with weights that are also
capable of handling systems satisfying weaker ``null conditions'', so  
our setup should also be useful for a class of
quasilinear stability problems where one assumes certain bounds on
Ricci curvature;  but we leave this to a subsequent work.

First we discuss the basic assumptions and results which are contained in this paper. 
Additional remarks and references  with then follow.

\subsection{Basic Notation and Metric Assumptions}

When stating inequalities we will use $A\lesssim B$ to mean $A\leq CB$ for some fixed $C>0$
independent of $A,B$. We also employ the index notation $A\lesssim_k B$ when $C=C(k)$ depends
on some auxiliary parameter $k$ (although we will not always use such notation when a dependency exists).

The setting for the paper is the following: 
We fix some compact  $\mathcal{K}\subset \mathbb{R}^3$ with smooth boundary such that
$\mathbb{R}^3\setminus \mathcal{K}$ is connected. On
the manifold $\mathcal{M}=[0,\infty)\times (\mathbb{R}^3\setminus \mathcal{K})$
we suppose there is given a smooth Lorentzian metric $g_{\alpha\beta}$. We write $(t,x)$ for rectangular coordinates
$0\leq t<\infty$ and $x\in \mathbb{R}^3\setminus \mathcal{K}$. Note that we may 
assume $\mathcal{K}=\emptyset$. In the sequel we will also assume that the level sets $t=const.$ are uniformly
spacelike in the sense that $-C<g^{00}<-c$.

With respect to the coordinates $(t,x)$ and Euclidean measure $dtdx$ we have $L^p$ norms restricted to time slabs 
in $\mathcal{M}$ of the form $[0,T]\times (\mathbb{R}^3\setminus \mathcal{K})$. These will be denoted by
$L^p[0,T]$. On the time slice $t=const$ we  denote by $L^p_x$ the corresponding restricted 
norm, and we denote by $L^p_tL^q_x[0,T]$ the mixed norms. 
It will also be useful to employ various inhomogeneous Besov versions of these spaces. For example
we denote   $\lp{\phi}{\ell^p_r L^q[0,T]}^p=\sum_{j\geq 0}\lp{\chi_j\phi}{ L^q[0,T]}^p$ 
where $\chi_j$ is a spatial cutoff on scale $\la x\ra\approx 2^j$. Similar notation is used for dyadic summations with respect
to the time variable $t$, and other auxiliary variables as well (such as the optical functions described below).
We   define fixed time and space time Sobolev spaces which will be denoted by $H^s_x$ and $H^s[0,T]$,
with the convention that in both cases we consider \emph{all}  $(t,x)$ derivatives $\partial^I$ for multiindex $|I|\leq s$.

In this work we make two basic assumptions about our spacetimes $(\mathcal{M},g)$. The first is an asymptotic condition on the metric:

\begin{defn}[Asymptotically flat radiating spacetimes]\label{rad_metrics}
Suppose $(\mathcal{M},g)$ is given as above. Then we say $g_{\alpha\beta}$
is ``outgoing radiating'' if the following hold:
\begin{enumerate}[I)]
		\item \label{opt_cond} (Weak ``Optical Function'') On $\mathcal{M}$ there is defined
		a smooth  $u=u(t,x)$ such that
		$(u,x)$ forms a uniform set of coordinates in the sense that $C^{-1}<u_t <C$, and 
		$u$ has the symbol bounds:
		\begin{equation}
				\lp{(\tau_-\partial_t )^i(\tau_x  \tau_0 \partial_x)^J (\partial_t u - 1, 
				\partial_i u + \omega^i) }{\ell^1_r L^\infty[0,\infty)} \ < \ \infty 
				\ , \qquad
				\hbox{for all\ \ \ }  (i,J)\in\mathbb{N}\times \mathbb{N}^4
				\ , \label{u_sym_bnd}
		\end{equation}
		where $\tau_x=\la x \ra$, $\tau_-=\la u\ra $,    $\tau_+=\la(t,x)\ra$, 
		$\tau_0=\tau_-\tau_+^{-1}$, and where $\omega^i=x^i\tau_x^{-1}$.  
		Henceforth we let $\partial=(\partial_t,\partial_x)$ to denote the $(t,x)$ coordinate derivatives,
		and likewise $\partial^b=({\partial_u^b},\partial_x^b)$ will denote the $(u,x)$ coordinate derivatives.
		\item\label{sym_bnds} (Outgoing Radiation Condition) First define symbol classes
		$\mathcal{Z}^k$ in terms of the the seminorms (restricted to $t\in [0,\infty)$):
		\begin{equation}
				\LLp{q}{k,N} \ := \ \sum_{i+|J|\leq N}
				\lp{\tau_0^{-k}(\tau_-{\partial_u^b} )^i (\tau_x \partial_x^b)^J q }{\ell^1_rL^\infty} 
				+\sum_{i+|J|\leq N}
				\lp{  \tau_0^{-k}(\tau_-{\partial_u^b} )^i (\tau_x \partial_x^b)^J
				q }{\ell^1_u\ell^1_r L^\infty(\frac{1}{2}t<r<2t)} 
				\ . \label{Zk_defn}
		\end{equation}
		Then for the inverse metric $g^{\alpha\beta}$ in $(u,x)$ coordinates one has the symbol bounds:
		\begin{equation}
			  g^{\alpha\beta}-h^{\alpha\beta} \ \in \
			 \mathcal{Z}^0 \ ,  \qquad	
			  \sqrt{|g|}g^{iu} +  \omega^i \ \in \
			 \mathcal{Z}^\frac{1}{2} \ , \qquad
			  g^{ui}- \omega^i\omega_j g^{uj}   \ \in \
			 \mathcal{Z}^1 \ ,  \qquad
			   g^{uu}  \ \in \
			 \mathcal{Z}^2  \ , \label{mod_coords} 
		\end{equation}
		where:
		\begin{equation}
			h^{uu}\ = \ 0 \ , \qquad
			h^{ui} \ = \ -\omega^i \ , \qquad
			h^{ij} \ = \ \delta^{ij} \ . \label{h_tensor}
		\end{equation}
\end{enumerate}
\end{defn}

A number of further remarks are in order concerning the previous definition. 
Proofs are provided in Appendix \ref{mod_coords_sect}.

\begin{rem}
An immediate consequence of  \eqref{u_sym_bnd} is that one has
a uniformly bounded change of frame between $(\partial_t,\partial_x)$
and $({\partial_u^b},\partial_x^b)$. Specifically:
\begin{equation}
		\partial^b_\alpha \ = \ e^\beta_\alpha \partial_\beta \ , \ \ 
		\sum_{|I|\leq k} |\partial^I e^\beta_\alpha| \ \lesssim_k \ 1 \ , \qquad\qquad
		\partial_\alpha \ = \ f^\beta_\alpha \partial^b_\beta \ , \ \ 
		\sum_{|I|\leq k} |(\partial^b)^I f^\beta_\alpha| \ \lesssim_k \ 1 \ . \label{bondi_to_rec}
\end{equation}
In particular one may use either $(\partial_t,\partial_x)$
or $({\partial_u^b},\partial_x^b)$ to define the Sobolev spaces $H^s[0,T]$ and $H^s_x$, and for this purpose
we will use these frames
interchangeably in the sequel. Note however that the frames $(\partial_t,\partial_x)$
and $({\partial_u^b},\partial_x^b)$ are not interchangeable in all contexts. See for instance
Remark \ref{NLE_frame_rem}  below.
\end{rem}

\begin{rem}\label{optical_remark}
The condition \ref{opt_cond}) implies that $\tau_+^{-1}(u+\tau_x-t)=o_r(1)$, and together with \ref{sym_bnds}) 
we have:
 \begin{equation}
		\lp{(\tau_-\partial_t )^i(\tau_x  \tau_0 \partial_x)^J (g-\eta) }{\ell^1_r L^\infty[0,\infty)} \ < \ \infty \ , \qquad
		\hbox{for all\ \ \ }  (i,J)\in\mathbb{N}\times \mathbb{N}^4 \ , 
		\notag 
\end{equation}
where $\eta=diag(-1,1,1,1)$ is the Minkowski metric in $(t,x)$ coordinates. In particular
$g$ is weakly asymptotically flat in the sense that $\partial^J(g-\eta)=o_r(1)$ for all $J\in\mathbb{N}^4$,
and away from the region $\la t-r \ra\ll \la t+r \ra $
this can be strengthened to $(t\partial_t)^i (\tau_x \partial)^J(g-\eta)=o_r (1)$ 
for all $(i,J)\in\mathbb{N}\times \mathbb{N}^4$.
\end{rem}

\begin{rem}\label{coord_rem}
Definition  \ref{rad_metrics} allows us  to replace  $u$ by
$\td{u}= \chi u + (1-\chi)(t-\tau_x)$, where $\chi$ is a cutoff 
supported  where $\la t-r \ra\ll \la t+r \ra $ with bounds $|(\tau_+\partial)^J \chi|\lesssim 1$. Thus, in the sequel
we shall always assume that $u=t-\tau_x $ away from the ``wave zone'' $\la t-r \ra\ll \la t+r \ra $.
\end{rem}

\begin{rem}
For metrics which are quasi-stationary and quasi-spherical in the sense that $g=g_0+g_1$ where
$g_0$ is spherical in $(t,x)$ coordinates and:
\begin{equation}
	\slp{ \ln^2(1+\tau_x)  (t\partial_t)^i (\tau_x \partial)^J(g_0-\eta)}{\ell^1_rL^\infty[0,\infty)} 
	+\slp{ \tau_x  (t\partial_t)^i (\tau_x \partial)^J g_1}{\ell^1_rL^\infty[0,\infty)}  \ < \ \infty
	\ , \quad
	\hbox{for all\ \ \ }  (i,J)\in\mathbb{N}\times \mathbb{N}^4 \ ,
	\notag
\end{equation}
it can be shown the the condition of Definition \ref{rad_metrics} hold.  In particular this guarantees that 
our conditions hold for certain stationary long range perturbations of the Kerr family of spacetimes. In general
we leave as an open question to find  natural  (e.g.~geometric) conditions on 
metrics which  satisfy only: 
$$\lp{ \tau_x^a  (t\partial_t)^i (\tau_x \partial)^J(g-\eta)}{\ell^1_rL^\infty[0,\infty)}<\infty$$
for some $0\leq a<1$,  which would guarantee the existence of an optical function such as in 
Definition \ref{rad_metrics}. 

\end{rem}

\subsection{The Local Energy Decay   Assumption}

The second main assumption of this paper concerns the behavior of solutions to the inhomogeneous
scalar wave equation on $(\mathcal{M},g)$.
Following Tataru etal  (\cite{MMTT}, \cite{TT},\cite{T_price}) we define the local energy decay norms.

\begin{defn}[``Classical'' Local Energy Decay Norms]\label{class_LE_defn}
First set:
\begin{equation}
		\lp{\phi}{\LE[0,T]} \ = \ \lp{\tau_x^{-\frac{1}{2}} \phi}{\ell^\infty_r L^2[0,T]} \ , \qquad
		\lp{F}{\LE^{*}[0,T]} \ = \ \lp{\tau_x^{\frac{1}{2}} F}{\ell^1_r L^2[0,T]} \ . \notag
\end{equation}
Next, for a fixed $R_0\geq 1$ sufficiently large ($r<R_0$ may
be assumed to contain $\mathcal{K}$), and integer $s\geq 0$, we define:
\begin{subequations}\label{class_LE_norms}
\begin{align}
		\lp{\phi}{\LE^{s }_{class}[0,T]} \ &= \
		\sum_{|J| \leq s}\big( 
		\lp{\tau_x^{-1} \partial^J \phi}{\LE[0,T]} + 
		\lp{ \partial \partial^{J} \phi}{\LE[0,T]} \big) \ , \label{class_LE_norm1}\\
		\slp{\phi}{\WLE^s_{class}[0,T]} \ &= \
		\sum_{|J| \leq s} \big( 
		\slp{\tau_x^{-1} \partial^J \phi}{\LE[0,T]} + 
		\slp{ \partial \partial^{J} \phi}{\LE(r>R_0)[0,T]} 
		\big) \ , \label{class_LE_norm2}\\
		\lp{F}{\LE^{*,s} [0,T]} \  &= \  \sum_{|J| \leq s} 
		\lp{ \partial^J F }{\LE^*[0,T]} \ , \label{class_LE_norm3}\\
		\lp{F}{\WLE^{*,s}[0,T]} \  &= \  \sum_{|J| \leq s} \big(
		\lp{ \partial^J F }{\LE^*[0,T]} 
		+ \lp{\partial  \partial^{J}F}{L^2(r<R_0)[0,T]} 
		\big)\ . \label{class_LE_norm4}
\end{align}
\end{subequations}
\end{defn}

In terms of these spaces the second main assumption of the paper is:

\begin{assum}[Weak and Stationary Energy Boundedness/Decay Estimates]
For  $R_0$ as in Definition \ref{class_LE_defn} above and any
$s\geq 0$ there holds the estimates:
\begin{subequations}\label{class_LE_assm}
\begin{align}
		\sup_{0\leq t\leq T}\
		\slp{\partial\phi(t)}{H^{s}_x} + \slp{\phi}{\WLE^s_{class}[0,T]}  \ &\lesssim_s\  \slp{\partial \phi(0)}{H^{s}_x} 
		+ \slp{\Box_g \phi }{(\WLE^{*,s}+L^1_t H^s_x)[0,T]} \ , \label{basic_LE}\\
		\sup_{0\leq t\leq T}
		\slp{\partial\phi(t)}{H^{s}_x}\!+\! \slp{\phi}{\LE^s_{class}[0,T]}   \ &\lesssim_s\  \slp{\partial \phi(0)}{H^{s}_x} 
		\!+\! \slp{ (\phi,\partial_t \phi)}{L^2\times H^s(r<R_0)[0,T]}
		\!+\! \slp{\Box_g \phi }{\LE^{*,s} [0,T]} \ . \label{basic_stat_LE}
\end{align}
\end{subequations}
\end{assum}

In practice it is useful to have a bound which does not include the low frequency error on the 
RHS\eqref{basic_stat_LE}. This is had by concatenating \eqref{basic_LE}--\eqref{basic_stat_LE} which gives:
\begin{equation}
		 \sup_{0\leq t\leq T}
		\lp{\partial\phi(t)}{H^{s}_x}+ \lp{\phi}{\LE^s_{class}[0,T]}  \ \lesssim_s  \ \lp{\partial \phi(0)}{H^{s}_x} 
		+ \lp{\partial_t\phi}{H^{s}(r<R_0)[0,T]}
		+ \lp{\Box_g \phi }{\WLE^{*,s}[0,T]} \ . \label{concat_LE}
\end{equation}

Before continuing we make a number of remarks about these assumptions.

\begin{rem}
Estimates of the form \eqref{basic_LE} have a long history  for both asymptotically flat
nontrapping spacetimes, and for the Kerr family of metrics. For the case of black holes we refer the reader 
to \cite{BS}, \cite{DR}, \cite{MMTT},  \cite{TT}, and \cite{DR_LE} and the references therein for more detailed discussions. 
In the sharp form needed for the present paper
the estimate \eqref{basic_LE} was proved by Marzuola-Metcalfe-Tataru-Tohaneanu \cite{MMTT} for Schwarzschild space and
Tataru-Tohaneanu \cite{TT} for Kerr with $|a|\ll M$. More recently a 
slightly less precise version of 
\eqref{basic_LE} was established for the full subextremal range of Kerr by Dafermos-Rodnianski and Shlapentokh-Rothman
\cite{DR_LE}. It is expected that \eqref{basic_LE}
holds for a wide class of asymptotically flat   spacetimes which satisfy certain natural structural conditions
similar to those of the Kerr metric, as well as certain natural spectral assumptions. We refer the reader 
to \cite{MST} for a definitive account in the case of nontrapping spacetimes.
\end{rem}

\begin{rem}
The second  estimate   \eqref{basic_stat_LE} is essentially
the ``stationary local energy decay'' estimate of Tataru etal \cite{MTT}.
It can be shown  that \eqref{basic_stat_LE} follows from estimates similar to \eqref{basic_LE} when 
$\partial_t$ is timelike on a set $\mathcal{T}$ where one looses regularity in the 
local energy decay norms \eqref{class_LE_norms}.
See Appendix 2 for  a proof. 
We remark that such  timelike  conditions hold for the Kerr family of metrics
when the angular momentum satisfies $0\leq |a| <a_0$, for some $0<a_0<M$. On the other hand
one should still expect \eqref{basic_stat_LE} to hold for the
full subextremal range $0\leq |a| <M$ of Kerr  
 thanks to the (microlocal) nonvanishing of the symbol of $\partial_t$ on the trapped set. 
\end{rem}


 

\subsection{Norms and vector fields}

We now introduce the weighted local energy decay norms that will play a leading role in the remainder of the paper.

\begin{defn}[Null Energy Decay Norms]
First  we define a null energy seminorm with the same scaling as $\LE$:
\begin{equation}
		\lp{\phi}{\NLE[0,T]} \ = \ 
		\lp{\tau_-^{-\frac{1}{2}}  \phi}{\ell^\infty_uL^2(\frac{1}{2}t<r<2t)[0,T]} \ , \qquad
		\lp{F}{\NLE^*[0,T]} \ = \  
		\lp{ \tau_-^\frac{1}{2}   F}{\ell^1_u L^2(\frac{1}{2}t<r<2t)[0,T]} \ . \notag
\end{equation}
Next, we  have the following  generalization of   $\LE^s_{class}$  
and $\WLE^s_{class}$ which include these semi-norms:
\begin{subequations}\label{LE_norms}
\begin{align}
		\lp{\phi}{\LE^{s }[0,T]} \ &= \ \lp{\phi}{\LE^{s }_{class}[0,T]}+
		\sum_{|J| \leq s} 
		\lp{( \partial_x^b (\partial^b)^J\phi,\tau_x^{-1}\partial^J\phi )}{\NLE(r>R_0)[0,T]}  \ , \label{LE_norm1}\\
		\lp{\phi}{\WLE^s[0,T]} \ &= \ \lp{\phi}{\WLE^s_{class}[0,T]}+
		\sum_{|J| \leq s} 
		\lp{  (\partial_x^b (\partial^b)^J \phi,\tau_x^{-1}\partial^J\phi)}{\NLE(r>R_0)[0,T]} \ . \label{LE_norm2}
\end{align}
\end{subequations}
\end{defn}

\begin{rem}\label{NLE_frame_rem}
Notice  we have specifically used  Bondi coordinate derivatives $({\partial_u^b},\partial_x^b)$ inside the
gradient portion of RHS\eqref{LE_norms}. This appears to be necessary because condition \eqref{u_sym_bnd}
does not guarantee good peeling estimate for $(\partial^b_x)^I \partial u$. In particular we don't assume improved control
of the commutators $[\partial_x^b,\partial_\alpha]$ beyond the bounds 
\eqref{bondi_to_rec}.
\end{rem}

Next, we define   $\LE$ type norms with uniform weights in time. These  will play a central
role in establishing peeling estimates for solutions to the wave equation.

\begin{defn}[Weighted LE and Conformal Energy Norms]
For functions $\phi$ which solve the wave equation and $0\leq a\leq 1$ we set:
\begin{subequations}\label{weighted_LE}
\begin{align}
		\lp{\phi(t)}{\CE^a} \ &= \ \lp{\tau_+^a\tau_0^{\max\{a,\frac{1}{2}\}} \partial \phi(t)}{L^2_x } +
		 \lp{\tau_+ ^a (\partial^b_x \phi(t),\tau_x^{-1}\phi(t))}{L^2_x }
		  \ , \label{CE_def}\\
		\lp{\phi}{S^a[0,T]} \ &= \  
		 \lp{\tau_+^a \tau_0^{\max \{a,\frac{1}{2}\}}
		\partial\phi }{\LE[0,T]}
		+ \lp{\tau_+ ^a (\partial^b_x \phi,\tau_x^{-1}\phi) }{\LE[0,T]} +
		\lp{ \tau_+^a  \tau_x^{-1}\partial_r^b(\tau_x  \phi)}{\NLE[0,T]}
		 \ ,  \label{Sa_def}\\
		\lp{\phi}{S^{1,\infty}[0,T]} \ &= \  \lp{\phi}{\ell^\infty_t S^{1}[0,T]}+
		\lp{ \tau_+ \tau_x^{-1}\partial_r^b(\tau_x  \phi)}{\NLE[0,T]} \ , \label{S100_def}	
\end{align}
\end{subequations}
For source terms we fix  a parameter $ R_0$ as in Definition \ref{class_LE_defn} and set:
\begin{subequations}\label{N_norms}
\begin{align}
		\lp{F}{N^a[0,T]}  \ &=  
		\lp{\tau_+^a\tau_0^\frac{1}{2} F}{\LE^* [0,T]}  +
		\lp{\tau_+^a \partial F}{L^2(r<R_0) [0,T]}  +
		\lp{\tau_+^a  F}{  \NLE^*[0,T]}  
		\ . \label{Na_def}\\
		\lp{F}{N^{1,1}[0,T]}  \ &=  
		\lp{\tau_+ \tau_0^\frac{1}{2} F}{\ell^1_t \LE^* [0,T]}  +
		\lp{\tau_+ \partial F}{\ell^1_t L^2(r<R_0) [0,T]} +
		\lp{\tau_+   F}{  \NLE^*[0,T]}  \ . \label{N11_def}
\end{align}
\end{subequations}
\end{defn}
  
Finally, we  construct higher order versions of all the above norms.

\begin{defn}[Modified Vector Fields]\label{mod_vect_defn}
 First  define the approximate Lie algebras:
\begin{equation}
		\mathbb{L}_0 \ = \ \{ {\partial_u^b} , \partial_i^b-\omega^i {\partial_u^b} \} \ , \quad
		\mathbb{L} \ = \ \{S, \Omega_{ij} \}\cup \mathbb{L}_0 \ , \qquad
		\hbox{where\ \ } S = u{\partial_u^b} + r\partial_r^b \ , \ \ 
		\Omega_{ij} \ = \ x^i \partial_j^b - x^j\partial_i^b \ . \label{VF}
\end{equation}
Note that all members of $\mathbb{L}$ commute modulo $\mathbb{L}$ with the exception 
of:
\begin{equation}
		[\partial_i^b-\omega^i {\partial_u^b},S ] \ = \ \partial_i^b-\omega^i {\partial_u^b} + \tau_x^{-2}\omega^i {\partial_u^b} \ 
		. \label{comm_fail}
\end{equation}
For a function $\phi$ we write:
\begin{equation}
		\phi^{(k)} \ = \ (\phi, \Gamma^{I_1}\phi, \Gamma^{I_2}\phi, \ldots) \ , \notag
\end{equation}
where the RHS is an array of all products $\Gamma^I$ of vector fields in $\mathbb{L}$
up to length $|I|\leq k$. If $\lp{\cdot }{}$ is any norm we write:
\begin{equation}
		\lp{\phi^{(k)}}{} \ = \ \sum_{|I|\leq k}\lp{\Gamma^I\phi}{} \ , \qquad
		\Gamma\in \mathbb{L} \ . \notag
\end{equation}
We use a similar notation for pointwise identities, for example
$|\partial \phi^{(k)}|=\sum_{|I|\leq k}|\partial \Gamma^I\phi|$ etc. 
In the case of norms we use a subscript notation to denote higher order derivatives by vector fields:
\begin{equation}
		\lp{\phi (t)}{\CE^a_k} \ = \ \lp{\phi^{(k)} (t)}{\CE^a} \ , \qquad
		\lp{\phi}{S^a_k[0,T]} \ = \ \lp{\phi^{(k)}}{S^a[0,T]} \ , \qquad
		\lp{F}{N^a_{k}[0,T]} \ = \ \lp{F^{(k)}}{N^a[0,T]} \ , \notag
\end{equation}
and similarly for $\LE$ type norms. In addition we set $\lp{\phi}{H^{s}_k}=\lp{\phi^{(k)}}{H^s}$
and  $\lp{\phi(t)}{H^{s}_{x,k}}=\lp{\phi^{(k)}(t)}{H^s_x}$.
\end{defn}


\subsection{Main Results I: Linear Estimates}
 
The main result of the paper can now be stated as follows.

\begin{thm}[Weighted Local Energy Decay Estimates]\label{main_thm}
Assuming estimates \eqref{basic_LE} and \eqref{basic_stat_LE}, then for $0\leq a\leq 1$
and fixed $s,k \geq 0$ there exists parameters $B_a=B_a(s,k)$ such that:

\begin{enumerate}[I)]
\item In the case $a=0$ one has:
\begin{multline}
		\sup_{0\leq t\leq T} \lp{\partial\phi(t)}{H^{s}_{x,k} } 
		+ \lp{\phi }{ \WLE^s_{k}[0,T]} 
		 \lesssim_{s,k} \  \sum_{|I|\leq k} \sum_{|J|\leq s}\lp{(\tau_x \partial_x)^I \partial_x^J \partial  \phi(0) }{L^2_{x}}
		+ \lp{\phi}{H^{s+3}_{k-1}(r<B_0)[0,T]}\\
		 + \lp{\Box_g \phi }{(\WLE^{*,s}_{k}+L^1_t H^s_{x,k})[0,T]}
		\ , \label{LE_null_est_k}
\end{multline}
where in the case $k=0$ we define  $\slp{\phi}{H^{s+3}_{-1}(r<B_0)[0,T]}=0$.  
 
\item For  $0<a<1$   one has:
\begin{equation}
		\sup_{0\leq t\leq T}\!\! \slp{\phi(t)}{\CE^a_k}+
		s\lp{\phi}{S^a_k[0,T]} 
		 \lesssim_{s,k,a}   \sum_{|J|\leq k}\slp{\tau_x^a (\tau_x \partial_x)^J\partial  \phi(0) }{ L^2_x}
		+ \slp{\tau_+^{a-1} \phi }{ H^{1}_{k+1}(r<B_a)[0,T]} + 
		\slp{\Box_g \phi}{N^a_{k}[0,T]}
		\ . \label{Sa_est_k}
\end{equation}

\item Corresponding to $a=1$   one has  the endpoint bound:
\begin{equation}
		\sup_{0\leq t\leq T}\slp{\phi(t)}{\CE^1_k}
		+ \slp{ \phi }{  S^{1,\infty}_k[0,T]}\lesssim_{s,k} 
		 \sum_{|J|\leq k}\slp{\tau_x  (\tau_x \partial_x)^J\partial  \phi(0) }{ L^2_x}
		+ \slp{\phi }{ \ell^1_t H^{1}_{k+1}(r<B_1)[0,T]}  
		  + \slp{ \Box_g\phi}{ N^{1,1}_{k}[0,T]} \ . \label{CE_est_k}
\end{equation}
\end{enumerate}
\end{thm}

\begin{rem}\label{R0_rem}
In the proof of Theorem \ref{main_thm}  we find that for $a\neq 0,1$ one has
$B_a\to \infty$ as $a\to 0,1$. 
\end{rem}

In addition to the above estimates we shall also prove the following analog of Klainerman's  
estimate:

\begin{thm}[Global Sobolev Estimate]\label{glob_sob_thm}

For $k\geq 1$ one has the estimate:
\begin{multline}
		\sum_{i+|J|\leq k}
		\lp{\tau_+^\frac{3}{2} \tau_0^\frac{1}{2}(\tau_-{\partial_u^b})^i
		(\tau_x \partial_x^b)^J\phi}{L^\infty[0,T]} \ \lesssim_k \
		 \sup_{0\leq t\leq T}\lp{\phi(t)}{E^1_{k+1}}
		 + \lp{\phi}{S^{1,\infty}_{k+2}[0,T]}\\
		 +\sum_{|J|\leq k}\lp{  \tau_x^2
		(\tau_x \partial)^J\Box_g\phi (0)}{  L^2_x}
		+\sum_{i+|J|\leq k+1}\lp{ (\tau_-{\partial_u^b})^i
		(\tau_x \partial_x^b)^J\Box_g\phi}{  N^{1,1}[0,T]} \ . \label{main_L00_est}
\end{multline}

\end{thm}

One can combine the above two  results in a straightforward way, which produces
the first  main conclusion of our paper:

\begin{thm}[Peeling Estimates for the  Inhomogeneous Wave Equation]\label{main_lin_thm}
Given any $k\geq 1$ there exists  an integer $N=N(k) $ depending (linearly) on $k$
such that:
\begin{multline}
		\sum_{i+|J|\leq k}
		\lp{\tau_+^\frac{3}{2} \tau_0^\frac{1}{2}(\tau_-{\partial_u^b})^i
		(\tau_x \partial_x^b)^J\phi}{L^\infty[0,T]} \lesssim_k \
		 \sum_{|J|\leq N}\lp{\tau_x (\tau_x\partial_x)^J\partial  \phi(0) }{ L^2_x}
		 +\sum_{|J|\leq N}\lp{  \tau_x^2
		(\tau_x \partial)^J\Box_g\phi (0)}{  L^2_x}\\
		+ \sum_{i+|J|\leq N}
		\lp{ (\tau_-{\partial_u^b})^i
		(\tau_x \partial_x^b)^J\Box_g\phi}{  N^{1,1}[0,T]} \ . \label{lin_peeling}
\end{multline}
\end{thm}
 

\subsection{Main Results II:  Nonlinear estimates}

The estimates of Theorems \ref{main_thm} and \ref{glob_sob_thm}
naturally lend themselves bounding solutions to semilinear wave equations of the form $\Box_g \phi = F(t,x,\phi,\partial\phi)$.
Rather than develop a comprehensive theory we concentrate on the
equations     $\Box_g \phi = \mathcal{N}^{\alpha\beta}(t,x,\phi)
\partial_\alpha\phi\partial_\beta\phi$
where the quadratic form $\mathcal{N}^{\alpha\beta}$ is sufficiently tame.

\begin{defn}[Generalized null forms]\label{NF_defn}
A two-tensor $\mathcal{N}^{\alpha\beta}$ is called a ``generalized null form'' with respect to
a (weak) optical function $u$ if 
its Bondi coordinate components satisfy:
\begin{equation}
		|(\tau_-{\partial_u^b})^i(\tau_x \partial_x^b)^J\partial_\phi^k \mathcal{N}^{\alpha\beta}| 
		 \ \lesssim_{i,J} \ c_k(|\phi|) \ , \qquad
		 |(\tau_-{\partial_u^b})^i(\tau_x \partial_x^b)^J\partial_\phi^k 
		 \mathcal{N}^{uu}| \ \lesssim_{i,J}\ c_k(|\phi|)\tau_0
		 \ . \label{NF_cond}
\end{equation}
\end{defn}

\begin{rem}
Natural examples of $\mathcal{N}$ satisfying Definition \ref{NF_defn} include multiples of the
inverse metric $g^{\alpha\beta}$ by  factors $\mathcal{N}(t,x,\phi)$ which satisfy derivative estimates
consistent with \eqref{NF_cond}. This would include the case of wave-maps from $\phi:(\mathcal{M},g)\to 
(\mathcal{M}',g')$
into Riemannian or Lorentzian targets targets where $\phi$ is close to a constant map, in either an
intrinsic or extrinsic formulation.

Another example of such $\mathcal{N}$ would be 
skew symmetric forms $\mathcal{N}^{\alpha\beta}=-\mathcal{N}^{\beta\alpha}$
which obey the first condition on line \eqref{NF_cond}.
\end{rem}

In order to prove a-priori estimates we define the following norms:
\begin{align}
		\LLp{\phi}{S_k[0,T]} \ &= \ 
		 \sum_{j=0}^{k+4} \big( \sup_{0\leq t\leq T} \slp{\partial\phi(t)}{H^{13 +3(k-j)}_{x,j} } 
		+ \slp{\phi }{ \WLE^{13 +3(k-j)}_j[0,T]}\big)  
		+ \lp{\phi}{S^\frac{1}{2}_{k+3}[0,T]} \label{big_S_norm} \\ 
		&\hspace{.2in} + \sup_{0\leq t\leq T}\lp{\phi(t)}{\CE^1_{k+2}}
		+ \lp{ \phi }{  S^{1,\infty}_{k+2}[0,T]} + 
		\sum_{i+|J|\leq k}
		\lp{\tau_+^\frac{3}{2} \tau_0^\frac{1}{2}(\tau_-{\partial_u^b})^i
		(\tau_x \partial_x^b)^J\phi}{L^\infty[0,T]}  \ , \notag\\ 
		 \LLp{F}{N_k[0,T]} \ &= \ \sum_{j=0}^{k+4} \slp{F }{(\WLE^{*,13 +3(k-j)}_j+L^1_t H^{13 +3(k-j)}_{x,j})[0,T]}  
		+\lp{ F}{ N^{\frac{1}{2}}_{k+3}[0,T]}  \label{big_N_norm}\\
		&\hspace{.2in} +\sum_{|J|\leq k}\lp{  \tau_x^2
		(\tau_x \partial)^J F(0)}{  L^2_x}
		+ \sum_{i+|J|\leq k+2}
		\lp{ (\tau_-{\partial_u^b})^i
		(\tau_x \partial_x^b)^J F}{  N^{1,1}[0,T]} \ . \notag
\end{align}

With this notation the main nonlinear theorem of our paper is the following:

\begin{thm}\label{main_nonlin_thm_est}
Let $\phi$ be a smooth vector valued function and let 
$\mathcal{N}(\phi,\partial\phi)$ denote a quadratic form obeying \eqref{NF_cond}, where
we are using the notation
$\mathcal{N}(\phi,\partial\phi)=\mathcal{N}^{\alpha\beta}(t,x,\phi)
\partial_\alpha\phi\partial_\beta\phi$.
Then one has the bounds:
\begin{align}
		\LLp{\phi}{S_k[0,T]} \ &\lesssim_k \ 
		\sum_{|J|\leq 3k+13}\lp{\tau_x  (\tau_x \partial_x)^J\partial  \phi(0) }{ L^2_x}
		+\LLp{\Box_g\phi}{N_k[0,T]} \ , \label{NL_NS_bnd}\\
		\LLp{\mathcal{N}(\phi,\partial\phi)}{N_k[0,T]} \ &\lesssim \ C_k(\LLp{\phi}{S_k[0,T]})
		\LLp{\phi}{S_k[0,T]}(\LLp{\phi}{S_k[0,T]}
		+ \LLp{\Box_g\phi}{N_k[0,T]})   \ , \qquad
		k\geq 18
		 \ .  \label{NL_SN_bnd}
\end{align}
The functions $C_k(\cdot)$ are locally bounded functions  determined by $k$
as well as the functions $c_k$ on RHS \eqref{NF_cond}.
\end{thm}

From this and a standard continuity argument one has the nonlinear analog of Theorem \ref{main_lin_thm}:

\begin{thm}[Peeling Estimates for Solutions to Null Form Systems]\label{main_nonlin_thm}
Suppose quadratic  forms $\mathcal{N}$ are given which
satisfy \eqref{NF_cond}. Let $\phi$ be a sufficiently smooth and well localized  
solution to the system of semilinear equations:
\begin{equation}
		\Box_g \phi \ = \ \mathcal{N}^{\alpha\beta}(t,x,\phi)
		\partial_\alpha\phi\partial_\beta\phi \ , \label{box_NF}
\end{equation}
which is assumed to hold   on a time interval $[0,T]$.
Then given  any $k\geq 18$  there exists  an $\epsilon_0=\epsilon_0(k)>0$  
such that one has the a-priori estimate:
\begin{equation}
		 \sum_{|J|\leq 3k+13}\lp{\tau_x (\tau_x\partial_x)^J\partial  \phi(0) }{ L^2_x} \ = \ \epsilon \ \leq \ \epsilon_0
		 \ \ \ \Longrightarrow \ \ \ 
		\LLp{\phi}{S_k[0,T]} \ \lesssim_k \ \epsilon \ . \notag
\end{equation}
In particular for sufficiently smooth, small, and well localized  
initial data the solution to \eqref{box_NF} exists globally and enjoys the peeling estimates:
\begin{equation}
		\sum_{i+|J|\leq k}
		\lp{\tau_+^\frac{3}{2} \tau_0^\frac{1}{2}(\tau_-{\partial_u^b})^i
		(\tau_x \partial_x^b)^J\phi}{L^\infty[0,T]} \ \lesssim_k  \ \epsilon \ , \notag
\end{equation}
for the same $k$ as above.  
\end{thm}


\subsection{Some remarks and references}

The vector field method for the wave equation on asymptotically flat 
spacetimes has a long and well developed history.
In the case of small perturbations of Minkowski space there is Klainerman's original 
works \cite{K_vect}, \cite{K_null}, followed by  the Christodoulou-Klainerman proof \cite{CK} 
of the nonlinear stability  of Minkowski space itself.
In the latter work a vector field method is developed for radiating spacetimes  where 
 control is 
ultimately provided through certain peeling estimates for the curvature tensor of $g$. 
In a related vein there is work of Bieri \cite{B} concerning nonlinear stability under the much less restrictive
decay assumptions on the initial data. In this case the peeling properties of the metric end up being closer the the thresholds
\eqref{mod_coords}.

Going in a  different direction there
is  the   proof of stability of Minkowski space and its asymptotics in wave coordinates
by Lindblad-Rodnianski and Lindblad \cite{LR}, \cite{L_asym}. Here one proceeds
more directly via Minkowksi and Schwarzschild vector fields.   
This produces  estimates with a small
loss due to the  divergence  of radiating null hypersurfaces compared to their  
stationary counterparts. 
On the other hand strong
peeling estimates such \eqref{lin_peeling} cannot hold for the (Minkowski difference of the) metric 
in wave coordinates due to obstructions at the level of semilinear terms. More specifically,
condition \eqref{NF_cond} seems to fail when writing the Einstein equations in any reasonable way
as a system of second order equations for the metric.

Next, we turn to vector field methods on spacetimes which are locally large perturbations of Minkowski space.
There are mainly two innovations  here, and we make a heavy use of both
  in the sequel. The first innovation,  due to 
Klainerman-Sideris \cite{KS}, allows one to replace Lorentz boosts  with
 certain weighted  identities involving
the wave operator $\Box_g$ (specifically see Lemma \ref{KS_lem} below). The second innovation,
due initially to Keel-Smith-Sogge \cite{KSS} and used by many authors,
concerns the use of local energy decay estimates 
such as \eqref{class_LE_assm} in order to control  localized errors generated by
commutations with vector fields.  For further background and developments concerning the combination of
local energy decay and vector fields on various large perturbations of Minkowski space
we refer the reader to \cite{MS}, \cite{MS_null}, \cite{BH}, \cite{W_null}, and 
\cite{Y} and the extensive references contained therein.

Concerning the application of vector fields to the class of black hole spacetimes there has recently been a great deal of progress.
The works most closely related to the present paper
are \cite{BS}, \cite{DR_LE}, \cite{Luk_lin} and \cite{Luk_nonlin} which proceed by way of conformal energy  
estimates and time dependent weights.  We also mention the recent work of 
Lindblad-Tohaneanu \cite{LT} concerning  certain quasilinear 
wave equations (in this case the vector field approach is similar to \cite{LR}).
We  remark that all of these works are written specifically to cover  the case of Schwarzschild 
or Kerr with small angular momentum. 

For more general ``black hole''  backgrounds there is the recent work of Moschidis \cite{M_VF}
building on ideas of Dafermos-Rodnianski \cite{DR_phys} (see also Proposition 6.1 of  \cite{S_WM} where weighted estimates
based on outgoing null vector fields were   introduced).
Here one produces a vector field method with time 
independent weights. This  shares a number of similarities with the present paper, in particular the production
of a  family of weighted spacetime energy estimates depending on a parameter $0\leq a\leq 1$ which interpolates
between the standard local energy decay estimates and the conformal energy. On the other hand our
 approach and the one of Dafermos-Rodnianski and Moschidis
 appear to diverge  in many  ways, especially  regarding   the issue of time 
dependent weights. Our method also appears to be more directly applicable to studying  nonlinear problems.  Indeed, with the
machinery of estimates \eqref{LE_null_est_k}--\eqref{CE_est_k} our proof of Theorem \ref{main_nonlin_thm_est} 
 occupies only a few additional pages.


Finally, one should also mention the works on Price's law 
for scalar waves \cite{T_price} and  \cite{MTT}  which uses vector fields as a launching point 
for much sharper estimates. An interesting open question in this regard is to find an appropriate collection of norms
capable of producing interior decay rates better than $t^{-\frac{3}{2}}$, 
but which are still  compatible with  semilinear problems  such as in Theorem \ref{main_nonlin_thm}.

We close with a few additional comments  on the specific methods we employ in this work. These are a natural 
outgrowth of the papers \cite{BS}, \cite{S_MKG}, \cite{S_WM},  \cite{O_thesis}, and  \cite{DR_phys}. In \cite{S_MKG} 
we developed a conformal multiplier
technique that works well for perturbations of the wave equation, and then used it to produce a collection of conformal
energies with weights depending on a parameter. These ideas will again play a central role in the present work. In \cite{S_WM}
we introduced a Morawetz type estimate 
based on the multiplier $f(r) (\partial_t+\partial_r)$
and used it to control null tangential derivatives for solutions to a certain wave equation
with a potential (Proposition 6.1 of  that paper). This idea was expanded in \cite{DR_phys}
to prove global decay estimates. We will use similar multipliers in this paper
to produce a key portion of our estimates.\footnote{For a complete the list of multipliers
we use here, see Lemma \ref{core_mult_lem} and the proofs immediately following.} 
 In \cite{BS} we produced a weighted in time local energy decay bound, and then used
this to control the multiplier error from the conformal vector field. This technique is used again here for a wider range
of weights, albeit assuming the local energy decay bound as a black box.
 Finally, the first authors thesis \cite{O_thesis} and \cite{O_VF} developed  vector fields
on radiating nontrapping spacetimes with even weaker local bounds on $\partial_t g$. We will
use much of the setup from  \cite{O_VF} in the present work.


\subsection{Outline of the paper}

In Section \ref{conf_mult_sect}
we record a number of algebraic identities for  energy momentum 
and  deformation tensors. These form the basis for all  the
multiplier and commutator identities needed in the sequel.

In Section \ref{bondi_alg_sect} we specialize the formulas of Section \ref{conf_mult_sect}
to the case of Bondi coordinates satisfying Definition \ref{rad_metrics} and  
vector fields from Definition \ref{mod_vect_defn}.

Section \ref{asym_bondi_sect} is the  technical heart of the paper. We begin
by producing symbol bounds for the Lie derivatives of  two-tensors which satisfy certain natural
asymptotic estimates.
Building on this we construct a generic multiplier estimate for vector fields satisfying certain
structural properties. This  estimate covers all of the multiplier bounds needed in the sequel, and may
be useful for other applications as well which is one of our motivations for introducing
an  axiomatic setup.
Following this we move on to estimates for commutators. This is done  in a way that 
allows us to perform some delicate integration by parts  later in Section \ref{comm_sect}, and is also 
convenient form proving pointwise bounds.
We end by proving  several generalized Klainerman-Sideris type identities, and then use these
to conclude the discussion of pointwise bounds for commutators.

In Section \ref{L2_est_sect} we apply the abstract multiplier bound from Section  
\ref{asym_bondi_sect} to produce the three  main estimates of Theorem \ref{main_thm}
at  level $k=0$. For the bounds \eqref{LE_null_est_k} and \eqref{CE_est_k} this is done in such a way 
that the source error term can be integrated by parts; this form is needed later to establish
bounds for commutators.

In Section \ref{comm_sect} we apply the formulas of the previous two sections
to prove the estimates of Theorem \ref{main_thm} for an arbitrary number of vector fields.  For 
the bound \eqref{Sa_est_k} this is more or less straight froward. However, for the bounds
\eqref{LE_null_est_k} and \eqref{CE_est_k} the argument is more involved because one cannot
proceed directly via H\"older's inequality, and  several  additional integrations  seem
 necessary to close the argument.
 
Section \ref{L00_sect} we prove some sharp $L^\infty$ decay estimates for functions
in terms of the conformal energy norms \eqref{CE_def}, \eqref{S100_def}, and \eqref{N11_def}.
This established Theorem \ref{glob_sob_thm}.

In Section \ref{nonlin_sect} we prove Theorem \ref{main_nonlin_thm_est} which concludes
the main nonlinear application of the paper.

In Appendix 1 we provide proofs of a number of the remarks following Definition \ref{rad_metrics}.
In particular for  large class of stationary metrics which are also spherically symmetric to highest order,
we construct optical functions satisfying bounds \eqref{mod_coords}. 
Such metrics include radial  long rage perturbations 
of the Kerr family of metrics.

In Appendix 2 we show that estimate \eqref{basic_LE} implies estimate \eqref{basic_stat_LE} at least when the metric satisfies 
structural assumptions similar to the Kerr family  with angular momentum in a certain range.

Finally, in Appendix 3 we record a number of elementary Hardy and trace type estimates. 
These will be used throughout the body of the paper.


\section{Formulas for Commutators and Multipliers}\label{conf_mult_sect}

In this section we recall  some basic formulas which underlie the energy 
method for the  wave equation on curved backgrounds.
We do this first  with respect to the metric original $g$.  We then generalize
such formulas cover the case of metric conformal to $g$. The latter will form the basis
for most of the multiplier estimates in the sequel.


\subsection{Identities involving $g$}

We begin with a basic definition:

\begin{defn}[Normalized deformation tensor]
Let $X$ be a vector field, and set ${}^{(X)}\!\pi=\mathcal{L}_X g$. Then we define 
${}^{(X)} \widehat{\pi} = {}^{(X)}\!\pi - \frac{1}{2}g \cdot \hbox{trace}({}^{(X)}\!\pi)$
to be the ``normalized deformation tensor'' of $X$.
\end{defn}

The quantity ${}^{(X)} \widehat{\pi}$ underlies all of our formulas for multipliers
and commutators as the following result shows:

\begin{lem}[Formulas involving ${}^{(X)} \widehat{\pi}$]\label{basic_iden_lemma}
Let $\phi$ be a scalar field, and $X$ a vector field. As usual set 
$T_{\alpha\beta}=\partial_\alpha\phi\partial_\beta\phi - \frac{1}{2}g_{\alpha\beta} g^{\alpha'\beta'}
\partial_{\alpha'}\phi\partial_{\beta'}\phi$ to be the energy momentum tensor of $\phi$. Then 
one has the identities:
\begin{subequations}
\begin{align}
		{}^{(X)} \widehat{\pi}^{\alpha\beta} \ &= \ -\frac{1}{\sqrt{|g|}} X(\sqrt{|g|} g^{\alpha\beta})
		- g^{\alpha\beta }\partial_\gamma X^\gamma + g^{\alpha\gamma}\partial_\gamma( X^\beta)
		+ g^{\beta\gamma}\partial_\gamma( X^\alpha) \ , \label{pi_hat_formula}\\
		[\Box_g,X] \ &= \ \nabla_\alpha {}^{(X)} \widehat{\pi}^{\alpha\beta}\nabla_\beta
		- \frac{1}{2}\hbox{trace}({}^{(X)}\! \widehat{\pi})\Box_g
		\ , \label{X_comm_formula}\\
		\nabla^\alpha( T_{\alpha\beta}X^\beta) \ &= \ \frac{1}{2}{}^{(X)} \widehat{\pi}^{\alpha\beta}
		\partial_\alpha\phi \partial_\beta\phi + \Box_g\phi X\phi \ . \label{divergence_formula}
\end{align}
\end{subequations}
Finally, if $q$ is a smooth function then one has the commutator formula:
\begin{equation}
		{}^{(qX)} \widehat{\pi}^{\alpha\beta} - q\, {}^{(X)} \widehat{\pi}^{\alpha\beta} \ = \ 
		X^\alpha \nabla^\beta q + X^\beta \nabla^\alpha q - g^{\alpha\beta} Xq
		\ . \label{comm_pi_form}
\end{equation}
\end{lem}

We'll prove each of these formulas separately.

\begin{proof}[Proof of  \eqref{pi_hat_formula} and \eqref{comm_pi_form}]
We have ${}^{(X)} \pi^{\alpha\beta} = 
g^{\alpha\alpha'}g^{\beta\beta'}(\mathcal{L}_x g)_{\alpha'\beta'} =
- X  g^{\alpha\beta} 
+ g^{\alpha\gamma}\partial_\gamma( X^\beta)
+ g^{\beta\gamma}\partial_\gamma( X^\alpha)$. On the other hand 
$\frac{1}{2}\hbox{trace}({}^{(X)}\!\pi)=\nabla_\alpha X^\alpha = \frac{1}{\sqrt{|g|}}\partial_\alpha (\sqrt{|g|}
X^\alpha)=\frac{1}{\sqrt{|g|}}X (\sqrt{|g|}) + \partial_\alpha X^\alpha $. Subtracting the last two
identities gives \eqref{pi_hat_formula}.
A direct application of \eqref{pi_hat_formula} shows \eqref{comm_pi_form}.
\end{proof}

\begin{proof}[Proof of  \eqref{X_comm_formula}]
We begin with a formula that will also be useful in the sequel. Let 
$\mathcal{R}^{\alpha\beta}$ be any contravariant two tensor, then we claim:
\begin{equation}
		[X,\nabla_\alpha \mathcal{R}^{\alpha\beta}\nabla_\beta] 
		\ = \ \nabla_\alpha \td {\mathcal{R}}^{\alpha\beta}\nabla_\beta + \td{\mathcal{S}}^\beta
		\nabla_\beta \ , 
		\qquad \hbox{where\ \ }  \td {\mathcal{R}} \ = \ \mathcal{L}_X\mathcal{R} \ , \ \ 
		\hbox{and\ \ } \td{\mathcal{S}}^\beta \ =  \ \mathcal{R}^{\alpha\beta}
		\nabla_\alpha(\nabla_\gamma X^\gamma) \ . \label{R_comm_form}
\end{equation}
 To prove it, first note
that a straightforward calculation using the coordinate based formula for
$\nabla_\alpha X^\alpha$ above
shows $X(\nabla_\alpha Y^\alpha)- Y(\nabla_\alpha X^\alpha)=\nabla_\alpha [X,Y]^\alpha$
for any pair of vector fields $X$ and $Y$.
Applying this last formula 
to the vector field $Y^\alpha=\mathcal{R}^{\alpha\beta}\nabla_\beta\phi$ and using the
Leibniz rule for Lie derivatives followed by $[\mathcal{L}_X,d]=0$ for the exterior
derivative $d$ gives \eqref{R_comm_form}.

Now apply \eqref{R_comm_form} to  $\mathcal{R}=g^{-1}$. Using 
$(\mathcal{L}_X\mathcal{R})^{\alpha\beta}=-{}^{(X)} {\pi}^{\alpha\beta}$
and $\nabla_\alpha X^\alpha=-\frac{1}{2}\hbox{trace}({}^{(X)}\!\widehat\pi)$ gives \eqref{X_comm_formula}.
\end{proof}

\begin{proof}[Proof of \eqref{divergence_formula}]
A standard calculation shows $\nabla^\alpha( T_{\alpha\beta}X^\beta) = 
\frac{1}{2}{}^{(X)}\pi^{\alpha\beta} T_{\alpha\beta} + \Box_g\phi X\phi$ and
\eqref{divergence_formula} follows easily.
\end{proof}


\subsection{Conformal Changes for Scalar Fields}

In this section we recall a standard formula from geometry.  
Let $ g_{\alpha\beta} $ be a Lorentzian
metric on an $(3+1)$ dimensional spacetime. We consider 
a conformally equivalent metric 
$ \td{g}_{\alpha\beta} $ where
$\Omega^2 \td{g} =g $ 
for some weight function $\Omega>0$. Let $\td\nabla$ denote the Levi-Civita connection of $\td{g}$
and $\Box_{\td{g}}=\td\nabla^\alpha\td\nabla_\alpha$
the corresponding wave operator. Then we have: 

\begin{lem}[Identity for the conformal wave operator]
Let $\Box_g\phi =F$. Then 
one has the formula:
\begin{equation}
		\Box_{\td{g}}\psi + V\psi \ = \ \Omega^3 F \ , \qquad
		\hbox{where\ \ } \psi \ = \ \Omega \phi \ , \quad
		\hbox{and\ \ } V \ = \ \Omega^3\Box_g \Omega^{-1} \ . \label{conf_wave_iden}
\end{equation}
\end{lem}

\begin{proof}
Start with:
\begin{equation}
		\Box_{\td{g}} \ = \ \Omega^{4}\frac{1}{\sqrt{|g|}}\partial_\alpha ( \Omega^{-2}
		\sqrt{|g|} g^{\alpha\beta}
		\partial_\beta) \ = \ \Omega^{2}( \Box_g - 2 g^{\alpha\beta}
		\partial_\alpha \ln(\Omega)\partial_\beta ) \ . \notag
\end{equation}
To eliminate the second term on the RHS we rescale $\phi$ via 
$\psi= \Omega \phi$ which gives us:
\begin{equation}
		\Omega^{-2}\Box_{\td{g}}\psi \ = \ \Box_g\psi - 2 g^{\alpha\beta}
		\partial_\alpha \ln(\Omega)\partial_\beta\psi \ =\  \Omega 
		\Box_g \phi - W\phi \ , \notag
\end{equation}
where:
\begin{equation}
		W \ = \ -\Box_g( \Omega )
		+ 2\Omega 
		\partial_\alpha\ln(\Omega)\partial^\alpha\ln(\Omega)
		 \ = \ \Omega^{2}\Box_g  \Omega^{-1}  \ . \notag
\end{equation}
A straight forward manipulation of the last two lines gives \eqref{conf_wave_iden}.

\end{proof}
 

\subsection{Conformal  Multipliers}

We now combine the identities of the last two section to produce conjugated
weighted $L^2$ identities. 
Our main result here is:

\begin{lem}[The conformal divergence identity]\label{conf_dividen_lem}
Let $\Box_g\phi=F$. Let $X$ be a vector field supported in the exterior region
$\mathbb{R}^3\setminus\mathcal{K}$. Let $\chi(t,x)$ be a smooth function and $\Omega>0$
a smooth weight. Then  in $(t,x)$ coordinates one has the divergence identity:
\begin{equation}
		\dint Q(X,\chi,\Omega,\phi)
		\sqrt{|g|}dxdt
		\ = \ \int_{\mathbb{R}^3\setminus\mathcal{K}}  
		P(X,\chi,\Omega,\phi)
		\sqrt{|g|}dx \big|_{t=0}^{t=T}   \ , \label{dividen1}
\end{equation}
where:
\begin{equation}
		P(X,\chi,\Omega,\phi)
		\ = \ \Omega^{-2}  g^{0\alpha}\partial_\alpha (\Omega\phi ) X (\Omega\phi)  - 
		\frac{1}{2}\Omega^{-2} X^0 ( g^{\alpha  \beta } 
		\partial_{\alpha } (\Omega\phi )
		 \partial_{\beta } (\Omega\phi ) - \chi   V\phi^2)   \ ,    \label{P_chi_iden}
\end{equation}
with $V  =  \Omega^3\Box_g \Omega^{-1}$, and where: 
\begin{equation}
		Q(X,\chi,\Omega,\phi) \ = \ 
		F \cdot \Omega^{-1}X(\Omega  \phi) + \Omega^{-2}A^{\alpha\beta}\partial_\alpha(\Omega\phi)
		 \partial_\beta(\Omega\phi) + 
		 B^\chi\phi^2 + C^\chi \phi\, \Omega^{-1}X(\Omega \phi) \ , \label{conf_div_chi}
\end{equation}
with:
\begin{equation}
		A \ = \ \frac{1}{2}  \big({}^{(X)} \widehat{\pi}
		+2X\ln (\Omega) g^{-1}\big)  , \quad
		B^\chi  \ =\     \frac{1}{2}\Omega^{-2}\big(  X(\chi V)- \hbox{trace}(A)\chi V
		\big)  \ ,    \quad
		C^\chi \ = \  \Omega^{-2}(\chi - 1)V \ . 
		\label{AB_formulas}
\end{equation}
On the RHS of  the last two lines above all contractions are computed with respect to $g$.
\end{lem}

\begin{proof}
First define the tensors:
\begin{equation}
		\td{T}_{\alpha\beta}^\chi \ = \ \partial_\alpha\psi\partial_\beta\psi - \frac{1}{2}\td{g}_{\alpha\beta}
		(\td{g}^{\gamma\delta}\partial_\gamma\psi\partial_\delta\psi - \chi V\psi^2) \ , \qquad
		 {}^{(X)}\!\td{P}_\alpha^\chi \ = \ \td{T}_{\alpha\beta}^\chi X^\beta \ , \qquad
		 \hbox{where\ \ } \psi\ = \ \Omega\phi
		 \ . \notag
\end{equation}
Then by Stokes theorem and the support property of $X$ we have:
\begin{equation}
		\dint 
		\td{\nabla}^\alpha {}^{(X)}\!\td{P}_\alpha^\chi \,  \sqrt{|\td{g}|} dxdt
		\ = \ \int_{\mathbb{R}^3\setminus\mathcal{K}} \td{g}^{\alpha 0} \td{T}_{\alpha\beta}^\chi X^\beta
		 \sqrt{|\td{g}|}dx \big|_{t=0}^{t=T}   \ , \notag
\end{equation}
and RHS line \eqref{dividen1} follows by substituting $\Omega^{-2}g=\td{g}$ into the volume forms
and the RHS contractions. 

It remains to compute the $\td{g}$ contraction 
$\td{\nabla}^\alpha {}^{(X)}\!\td{P}_\alpha^\chi $
and show this produces the terms on LHS line \eqref{dividen1}.
To this end suppose $(\Box_{\td{g}}+V)\psi=G$. 
Then  one has: 
\begin{equation}
		\td{\nabla}^\alpha\td{T}_{\alpha\beta}^\chi \ = \ \big( (\chi-1)V\psi
		+G\big)\partial_\beta\psi + \frac{1}{2}\partial_\beta (\chi V)
		\psi^2 \ . \notag
\end{equation}
Using formula \eqref{divergence_formula}
we have:
\begin{equation}
		\td{\nabla}^\alpha {}^{(X)}\!\td{P}^\chi_\alpha \ = \ 
		\frac{1}{2} ( \widehat {\mathcal{L}_X \td{g}})^{\alpha\beta} \partial_\alpha\psi\partial_\beta\psi
		- \frac{1}{4}\hbox{trace} ( \widehat {\mathcal{L}_X \td{g}})\chi V\psi^2
		+\big( (\chi-1)V\psi
		+G\big)X\psi + \frac{1}{2}X (\chi V)
		\psi^2 \ , \label{div_P_iden}
\end{equation}
where all the contractions   are computed with respect to $\td{g}$.
To compute the first two RHS terms we use:
\begin{equation}
		\mathcal{L}_{X}\td{g}  
		\ = \ \Omega^{-2}\big(\mathcal{L}_X g - 2X\ln(\Omega)g \big) \ ,
		\qquad \hbox{and so\ \ }
		 \widehat {\mathcal{L}_X \td{g}}  \ =\  
		\Omega^{-2}\big(\widehat{\mathcal{L}_X g} 
		+ 2 X\ln (\Omega) g\big) \ . \notag
\end{equation}
Substituting the last  line into RHS \eqref{div_P_iden} and using
$G=\Omega^3 F$ we have \eqref{conf_div_chi} and \eqref{AB_formulas}.
\end{proof}


\section{Algebraic Formulas Involving Bondi Coordinates}\label{bondi_alg_sect}

In this section we compute the key quantities from Lemmas \ref{basic_iden_lemma} 
and \ref{conf_dividen_lem} in Bondi coordinate  $(u,x)$.

\begin{lem}[Formulas for deformation tensors]
Let $d=|g|$ be the determinant $g_{\alpha\beta}$ 
in rectangular bondi coordinates $(u,x^i)$, and set $\Omega= \tau_x$. Then 
if $X=X^\alpha \partial_\alpha^b$  is any  vector field  
in Bondi coordinates  we have the following formula for contravariant tensors:
\begin{equation}
		 {}^{(X)} \widehat{\pi}
		+ 2 X\ln (\Omega) g^{-1}   \ = \ 
		- d^{-\frac{1}{2}}\big( \mathcal{L}_X h +
		({\partial_u^b} X^u + \partial_r^b X^r + \partial_i^b \overline{X}^i ) h
		\big) + 
		\mathcal{R}  \ , \label{polar_bondi_defiden1}
\end{equation}
where $ \overline{X}^i=X^i-r^{-2}x^i x_j X^j$ denotes the angular portion of $X$
and $X^r=r^{-1}x_i X^i$ the radial portion, and 
where $h$ is given on line \eqref{h_tensor}.
The remainder tensor  $\mathcal{R}$ is given by the covariant formula:
\begin{equation}
		\mathcal{R}  \ =\ - d^{-\frac{1}{2}} \mathcal{L}_X(d^\frac{1}{2}g^{-1}-h) 
		- d^{-\frac{1}{2}}({\partial_u^b} X^u + \partial_r^b X^r + \partial_i^b\overline{X}^i )
		 (d^\frac{1}{2} g^{-1}-h )  - 2r^{-1}\tau_x^{-2}X^r g^{-1}
		\ . \label{polar_bondi_defiden2}
\end{equation}
\end{lem}

\begin{proof}[Proof of formula \eqref{polar_bondi_defiden1}]
Starting with formula \eqref{pi_hat_formula} in Bondi coordinates 
and then using the identity:
\[ \partial_\gamma^b X^\gamma = 2X\ln(\tau_x) + {\partial_u^b} X^u + \partial_r^b X^r
+ \partial_i^b\overline{X}^i+2r^{-1}\tau_x^{-2}X^r
\] 
and setting $\Omega=\tau_x$, we have the following formula for raised indices:
\begin{equation}
		{}^{(X)} \widehat{\pi}
		+ 2 X\ln (\Omega) g^{-1}    \ = \ 
		- d^{-\frac{1}{2}}\mathcal{L}_X (d^\frac{1}{2} g^{-1})
		  - ({\partial_u^b} X^u + \partial_r^b X^r
		+ \partial_i^b\overline{X}^i +2r^{-1}\tau_x^{-2} X^r) g^{-1}   \ . \notag
\end{equation}
Writing $g^{-1}=d^{-\frac{1}{2}}(d^\frac{1}{2}g^{-1}-h) + d^{-\frac{1}{2}}h$
and inserting into this last line gives \eqref{polar_bondi_defiden1} and \eqref{polar_bondi_defiden2}.
\end{proof}

\begin{lem}[Formulas for commutators]
The following commutator formulas hold
where $d=|g|$ is computed in Bondi coordinates:
\begin{enumerate}[I)]
		\item For $X\in \{{\partial_u^b} , \partial_i^b - \omega^i{\partial_u^b} , \Omega_{ij}\}$ one has the formula:
		\begin{equation}
			[\Box_g,X] \  = \  \nabla_\alpha \mathcal{R}^{\alpha\beta}\nabla_\beta
			+ \frac{1}{2} ( X\ln (d) )  \Box_g \ ,  
			\qquad \hbox{where\ \ }
			\mathcal{R}  \ =\ - d^{-\frac{1}{2}} \mathcal{L}_X(d^\frac{1}{2}g^{-1}-h) +\mathcal{R}_1\ ,
			\label{L_comm_form}
		\end{equation}
		where $\mathcal{R}_1=0$ for 
		$X\in \{{\partial_u^b} ,  \Omega_{ij}\}$, and $\mathcal{R}_1^{\alpha\beta}=
		2d^{-\frac{1}{2}}\omega^i\tau_x^{-3} \delta_u^\alpha\delta_u^\beta$
		when $X= \partial_i^b - \omega^i{\partial_u^b}  $. Again
		$h$ is as defined on line \eqref{h_tensor}.
		\item For $X=S$ one has:
		\begin{equation}
			[\Box_g,S]   =   \nabla_\alpha \mathcal{R}^{\alpha\beta}\nabla_\beta
			+ \frac{1}{2} ( 4+ S \ln (d))  \Box_g \ , \qquad
			\hbox{where\ \ } 
			\mathcal{R}     = - d^{-\frac{1}{2}} \mathcal{L}_S(d^\frac{1}{2}g^{-1}-h) 
			-2d^{-\frac{1}{2}}  (d^\frac{1}{2}g^{-1}-h) +\mathcal{R}_1\ ,
			\label{S_comm_form}
		\end{equation}
		and where $\mathcal{R}^{\alpha\beta}_1=d^{-\frac{1}{2}}\omega^i\tau_x^{-2}
		(\delta_i^\alpha\delta_u^\beta+ 
		\delta_i^\beta\delta_u^\alpha)$.
\end{enumerate}
\end{lem}

\begin{proof}[Proof of formulas \eqref{L_comm_form} and  \eqref{S_comm_form}]
First note that formula \eqref{pi_hat_formula} above can be rewritten
as:
\begin{equation}
		 {}^{(X)} \widehat{\pi}
		  \ = \ 
		- d^{-\frac{1}{2}} \mathcal{L}_X(d^\frac{1}{2}g^{-1}-h) 
		- d^{-\frac{1}{2}}\mathcal{L}_X h
		-  \partial_\alpha^b X^\alpha g^{-1}  \ , \qquad
		trace ({}^{(X)} \widehat{\pi}) \ = \ -2  \partial_\alpha^b X^\alpha - X\ln(d)
		  \ . \label{pi_rewrite}
\end{equation}

Next, for each $X\in \{{\partial_u^b} , \Omega_{ij},\partial_i^b-\omega^i{\partial_u^b}\} $
we have $ \partial_\alpha^b X^\alpha =0$, and for 
$X\in \{{\partial_u^b} , \Omega_{ij} \} $ we also have $\mathcal{L}_X h=0$. On the other hand for 
$X =\partial_i^b-\omega^i{\partial_u^b}$ one computes $\mathcal{L}_X h^{\alpha\beta} = 0$
for all but the $uu$ component, and for this  one has
$(\mathcal{L}_X h)^{uu}=-2\omega^i\tau_x^{-3}$. Combining this information with \eqref{pi_rewrite} above
and  \eqref{X_comm_formula} gives \eqref{L_comm_form}. 

Finally, in the case when $X=S$ we compute $\mathcal{L}_S h + 2h = -\omega^i  \tau_x^{-2}
(\delta_i^\alpha\delta_u^\beta+ \delta_i^\beta\delta_u^\alpha)$. This allows us to write for $X=S$:
\begin{equation}
	d^{-\frac{1}{2}}\mathcal{L}_X h +  \partial_\alpha^b X^\alpha g^{-1} \ = \ 
	2d^{-\frac{1}{2}}( d^\frac{1}{2}g^{-1}-h) - \mathcal{R}_1 + 2g^{-1} \ , \qquad
	\hbox{where\ \ }  \mathcal{R}_1^{\alpha\beta} \ =\  d^{-\frac{1}{2}}\omega^i  \tau_x^{-2}
	(\delta_i^\alpha\delta_u^\beta+ \delta_i^\beta\delta_u^\alpha) \ . \notag
\end{equation}
Using $trace ({}^{(X)} \widehat{\pi})=-8-S\ln(d)$ and combining everything with 
\eqref{pi_rewrite} above
and  \eqref{X_comm_formula} gives \eqref{S_comm_form}. 
\end{proof}


\section{Asymptotic Estimates Involving Bondi Coordinates}\label{asym_bondi_sect}

We now move on to the main technical calculations of the paper.  We record these
here in a general form that will be used throughout the sequel.


\subsection{Basic Estimates for Derivatives}

\begin{lem}[Estimates for the determinant]
Let $d=|g|=|\det(g)|$ be the absolute determinant of $g$ computed in Bondi coodinates
$(u,x^i)$. Then for any $\mu\in\mathbb{R}$ one has the symbol bounds:
\begin{equation}
		 d^\mu-1 \ \in \ \mathcal{Z}^0 \ , \label{det_bnds}
\end{equation}
where the symbol spaces $\mathcal{Z}^k$ are defined on line \eqref{Zk_defn}.
\end{lem}

\begin{proof}[Proof of \eqref{det_bnds}]
We can write $d^\mu-1=q_\mu(d^{-1}-1)$ where $q_\mu(s)$ is smooth for $s>-1$ and $q(0)=0$. Thus, by Taylor 
expansion and the Leibniz rule it suffices to consider the case $\mu=-1$. From \eqref{mod_coords} we know that:
\begin{equation}
		d^{-1}-1+\tau_x^{-2} \ = \ \det(h)- \det(g^{-1}) \ \in\  \mathcal{Z}^0 \ , \notag
\end{equation}
where $h$ is given on line \eqref{h_tensor}. Since $\tau_x^{-2}\in \mathcal{Z}^0$ this completes the proof.
\end{proof}

A useful corollary of this last lemma and the assumptions \eqref{mod_coords} is the following:

\begin{cor}
Let $d=|g|=|\det(g)|$ be the absolute determinant of $g$ computed in Bondi coodinates
$(u,x^i)$, and $g^{-1}$ the inverse metric of $g$ in Bondi coodinates. Then if $h$ is as on
line \eqref{h_tensor} and one sets $\mathcal{R}^{\alpha\beta}=(d^\frac{1}{2}g^{-1}-h)^{\alpha\beta}$
there holds the bounds:
\begin{equation}
			  \mathcal{R}^{ij}   \ \in \
			   \mathcal{Z}^0  \ ,  \quad
			 \mathcal{R}^{ui}\ \in \
			    \mathcal{Z}^\frac{1}{2} \ , \quad
			 \mathcal{R}^{ui}
			-\omega^i\omega_j \mathcal{R}^{uj}   \ \in \
			    \mathcal{Z}^1 \ , \quad 
			 \mathcal{R}^{uu}\ \in \
			  \mathcal{Z}^2
			 \ . \label{R_sym_bnds}
\end{equation}
\end{cor}

Building on the last two results we have the following collection of symbol bounds which will underly
many of the error estimates in the sequel.

\begin{lem}[Basic Lie derivative estimates]\label{basic_lie_lem}
Let $X=X^\alpha\partial_\alpha$ be a vector field. Then the following hold:
\begin{enumerate}[I)]
	\item \label{main_LR_estimate} Suppose that in Bondi coordinates $X$ satisfies the symbol type 
	bounds for $a,b,c\in\mathbb{R}$:
	\begin{equation}
		\big| (\tau_-{\partial_u^b})^l (\tau_x \partial_x^b)^J X^u\big|
		\ \lesssim_{l,J} \  \tau_x ^{a}\tau_+^{b}\tau_-^{c+1}\Big(\frac{\tau_x}{\tau_+}\Big)^{\min\{|J|,1\}} 
		\ , \qquad\qquad
		\big| (\tau_-{\partial_u^b})^l (\tau_x \partial_x^b)^J X^i\big|
		\ \lesssim_{l,J} \ \tau_x^{a+1}\tau_+^{b}\tau_-^{c}  \ , \label{vect_sym_bnds}
	\end{equation}
	and obeys the conditions:
	\begin{equation}
		  \partial_r^b X^u \ = \ {\partial_u^b}(X^i) \ = \ 
		  \partial_r^b r^{-1}(X^i-r^{-2} x^i x_j X^j) 
		  \ =\  0 \ .
		\label{good_X_conds}
	\end{equation}
	Let $\mathcal{R}^{\alpha\beta}$ be any contravariant two tensor which 
	satisfies \eqref{R_sym_bnds} with similar estimates for $\mathcal{R}^{iu}$ (if it is nonsymmetric).
	Then its Lie derivative by $X$, denoted by $\mathcal{L}_X\mathcal{R}=\mathcal{R}_X$,
	satisfies:
	\begin{equation}
			 \mathcal{R}^{ij}_X   \ \in \
			 \tau_x^{a}\tau_+^{b}\tau_-^{c}\cdot \mathcal{Z}^0  \ ,  \quad
			 \mathcal{R}^{ui}_X\ \in \
			  \tau_x^{a}\tau_+^{b}\tau_-^{c}\cdot \mathcal{Z}^\frac{1}{2} \ , \quad
			 \mathcal{R}^{ui}_X
			-\omega^i\omega_j \mathcal{R}^{uj}_X   \ \in \
			  \tau_x^{a}\tau_+^{b}\tau_-^{c}\cdot \mathcal{Z}^1 \ , \quad 
			 \mathcal{R}^{uu}_X\ \in \
			  \tau_x^{a}\tau_+^{b}\tau_-^{c}\cdot \mathcal{Z}^2
			 \ , \label{R_X_sym_bnds}
	\end{equation}
	with similar bounds for $\mathcal{R}^{iu}_X$.
	\item  \label{R_alt1} Alternatively, if one drops the condition
	$ \partial_r^b X^u= 0$ but keeps the rest of \eqref{vect_sym_bnds} and \eqref{good_X_conds},
	then the previous conclusion holds with the last bound on
	line \eqref{R_X_sym_bnds} replaced by:
	\begin{equation}
			\mathcal{R}^{uu}_X  \ \in \  \tau_x^{a}\tau_+^{b}\tau_-^{c}\cdot \mathcal{Z}^\frac{3}{2} \ . \label{Ruu_alt}
	\end{equation}
	\item \label{R_alt2} Alternatively, if one
	drops the condition ${\partial_u^b} (X^r)=0$ but retains ${\partial_u^b}(X^i-r^{-2}x^ix_j X^j)=0$ and
	the rest of  \eqref{vect_sym_bnds} and \eqref{good_X_conds},
	then the result of part \ref{main_LR_estimate})
	holds with the first bound on line  \eqref{R_X_sym_bnds} replaced by the pair:
	\begin{equation}
		  \mathcal{R}^{ij}_X \ \in  \ 
		\tau_x^{a}\tau_+^{b}\tau_-^{c}\cdot \mathcal{Z}^{-\frac{1}{2}} \ , \qquad
		  \mathcal{R}^{ij}_X - 
		\omega^i\omega^j \omega_k\omega_l \mathcal{R}^{kl}_X
		\ \in  \ 
		\tau_x^{a}\tau_+^{b}\tau_-^{c}\cdot \mathcal{Z}^{0}
		  \ . \label{alt_ij}
	\end{equation}
	\item \label{main_LS_estimate} Likewise, if $\mathcal{S}^\alpha$ satisfies:
	\begin{equation}
		  \mathcal{S}^i   \ \in \
		\tau_x^{ -1}\cdot \mathcal{Z}^{-\frac{1}{2}}  \ , \qquad
		 \mathcal{S}^u   \ \in \
		\tau_x^{-1}\cdot \mathcal{Z}^\frac{1}{2} \ , \label{S_sym_bnds}
	\end{equation}
	then $\mathcal{L}_X\mathcal{S}=S_X$ 
	satisfies:
	\begin{equation}
		  \mathcal{S}^i_X   \ \in \
		 \tau_x^{a-1}\tau_+^{b}\tau_-^{c}\cdot \mathcal{Z}^{-\frac{1}{2}}  \ , \qquad
		 \mathcal{S}^u_X   \ \in \
		\tau_x^{a-1}\tau_+^{b}\tau_-^{c}\cdot \mathcal{Z}^\frac{1}{2} \ , \label{S_X_sym_bnds}
	\end{equation}
	 when $X$ satisfies  the symbol
	bounds \eqref{vect_sym_bnds} (in this case we do not need the extra conditions
	\eqref{good_X_conds}).
	\item \label{X0_part} Finally, let $X\in \mathbb{L}_0=\{{\partial_u^b} , \partial_i^b-\omega^i{\partial_u^b}\}$,
	and let $\mathcal{R}$ and $\mathcal{S}$ satisfy \eqref{R_sym_bnds} and \eqref{S_sym_bnds}
	resp. Then $ \mathcal{L}_X\mathcal{R}$ and $ \mathcal{L}_X\mathcal{S}$
	satisfy \eqref{R_X_sym_bnds} and \eqref{S_X_sym_bnds} with $a=c=-1$ and $b=1$.
\end{enumerate}
\end{lem}

\begin{rem}\label{zero_diff_rem}
As will become apparent in the proof, if one is only interested in the norms \eqref{Zk_defn} 
at level $N=0$ for
$\mathcal{R}_X^{\alpha\beta}$ and $\mathcal{S}_X^\alpha$ (i.e.~no derivatives) ,
then one can replace the full symbol bounds \eqref{vect_sym_bnds}
with  first order conditions:
\begin{equation}
		\sum_{l+|J|\leq 1}\big| (\tau_-{\partial_u^b})^l (\tau_+ \partial_x^b)^J X^u\big|
		\ \lesssim \  \tau_x ^{a}\tau_+^{b}\tau_-^{c+1}
		 \ , \qquad\qquad
		\sum_{l+|J|\leq 1} \big| (\tau_-{\partial_u^b})^l (\tau_x \partial_x^b)^J X^i\big|
		\ \lesssim \ \tau_x^{a+1}\tau_+^{b}\tau_-^{c}  \ . \notag
\end{equation}
In this case the various implications above are true with the inclusion $\mathcal{R}\in 
\tau_x^{a}\tau_+^{b}\tau_-^{c}\cdot \mathcal{Z}^k$ replaced by the bound
$\LLp{\tau_x^{-a}\tau_+^{-b}\tau_-^{-c}\mathcal{R}}{k,0}<\infty$.
\end{rem}

\begin{proof}[Proof of Lemma \ref{basic_lie_lem}]
We prove the various portions separately. First note that conditions \eqref{R_sym_bnds}
are invariant with respect to dyadic cutoffs in the $r$, $u$, and $t+r$ variables. Therefore by utilizing
such cutoffs and the Leibniz rule we may assume $a=b=c=0$. As a second preliminary note the
identity:
\begin{equation}
		\hat{x}^i-\omega^i \ = \ \omega^i r^{-1}(r+\tau_x)^{-1} \ , \qquad
		\hbox{where\ \ } \hat{x}^i \ = \ r^{-1}x^i \ . \label{hat_omega_iden}
\end{equation}
This allows us to trade $\hat{x}^i$ for $\omega^i $ in the region $r>1$ as long as errors
on the order of $r^{-2}$ are acceptable.

\Part{1}{The  $\mathcal{R}$ bounds involving condition \eqref{vect_sym_bnds}}
We begin with the proof of estimates \eqref{R_sym_bnds} for $\mathcal{L}_X\mathcal{R}$
assuming  conditions \eqref{vect_sym_bnds} and \eqref{good_X_conds} or one of the
alternatives listed in items \ref{R_alt1}) and \ref{R_alt2}) above.
The formula for the Lie derivative is $\mathcal{L}_X\mathcal{R}^{\alpha\beta} 
= X(\mathcal{R}^{\alpha\beta})-\partial_\gamma (X^\alpha)\mathcal{R}^{\gamma\beta}
-\partial_\gamma (X^\beta)\mathcal{R}^{\alpha \gamma}$. We check each component
separately:

\case{1a}{The $uu$ component assuming $\partial_r^b X^u=0$}
By   assumption we have $\omega^i\partial_i^b X^u=0$, thus:
\begin{equation}
		\mathcal{L}_X\mathcal{R}^{uu} \ = \ X(\mathcal{R}^{uu})-2{\partial_u^b}(X^u)\mathcal{R}^{uu}
		-\partial_i^b (X^u)(\mathcal{R}^{ui}-\omega^i\omega_j \mathcal{R}^{uj})
		-\partial_i^b (X^u)(\mathcal{R}^{iu}-\omega^i\omega_j \mathcal{R}^{ju}) \ . \notag
\end{equation}
Then the   estimate on line \eqref{R_X_sym_bnds} for $ \mathcal{R}_X^{uu}$
is immediate from the estimates \eqref{vect_sym_bnds} and \eqref{R_sym_bnds}.

\case{1b}{The $uu$ component when $\partial_r^b X^u\neq 0$}
By the previous case we have the desired estimate modulo 
the additional expression $r^2\tau_x^{-2}\partial_r^b(X^u)(\mathcal{R}^{ur}+\mathcal{R}^{ru})$,
which adds a $\mathcal{Z}^\frac{3}{2}$ term.

\case{2}{The $ui$ and $iu$ components}
By symmetry 
it suffices to treat the $ui$ case. We have:
\begin{equation}
		\mathcal{L}_X\mathcal{R}^{ui} \ = \ X(\mathcal{R}^{ui})
		- {\partial_u^b}(X^u)\mathcal{R}^{ui} - {\partial_u^b}(X^i)\mathcal{R}^{uu} 
		-\partial_j^b (X^u) \mathcal{R}^{ji}
		-\partial_j^b (X^i) \mathcal{R}^{uj}  \ . \notag
\end{equation}
Using estimates \eqref{vect_sym_bnds} and \eqref{R_sym_bnds} we get a $\mathcal{Z}^\frac{1}{2}$
symbol bound  for this term. In addition one sees that for all 
parts of the above formula 
save for the expression $\mathcal{B}^i=X(\omega^i \omega_j \mathcal{R}^{uj})- 
{\partial_u^b}(X^u)\omega^i \omega_j\mathcal{R}^{uj}-
\omega^k \partial_k^b (X^i) \omega_j \mathcal{R}^{uj}$ the
bound is on the order of $\mathcal{Z}^1$. To show improved bounds we only need to consider
the region  $r>1$. Using \eqref{hat_omega_iden}
we see that $\mathcal{B}\equiv \td{\mathcal{B}} \mod 
r^{-2}\cdot \mathcal{Z}^\frac{1}{2}$ where 
$\td{\mathcal{B}}^i=X(\hat x^i \mathcal{R}^{ur})- {\partial_u^b}(X^u)\hat x^i  \mathcal{R}^{ur}-
\partial_r^b (X^i)  \mathcal{R}^{ur}$. Again using \eqref{hat_omega_iden}, we see that in order
to show $  \mathcal{B}^i - \omega^i\omega_j  \mathcal{B}^j \in \mathcal{Z}^1$
it suffices to prove 
$\chi_{r>1}( \td{\mathcal{B}}^i - \hat{x}^i \td{\mathcal{B}}^r) \in \mathcal{Z}^1$.
This would follow immediately if   $\td{\mathcal{B}}$ is a radially directed vector field.
Using $X(\hat x^i)=r^{-1} (X^i-\hat x^i X^r)$ we compute:
\begin{equation}
		\td{\mathcal{B}}^i \ = \ \hat x^i \big( X(\mathcal{R}^{ur})-{\partial_u^b} (X^u)\mathcal{R}^{ur} - \partial_r^b (X^r)
		\mathcal{R}^{ur}\big) - r \partial_r^b \big[ r^{-1} (X^i-\hat x^i  X^r)\big] \mathcal{R}^{ur}  \ , \notag
\end{equation}
which is manifestly radial thanks to the last condition on line \eqref{good_X_conds}.

\case{3a}{The $ij$ components assuming ${\partial_u^b}(X^i)=0$}
Here we have:
\begin{equation}
		\mathcal{L}_X\mathcal{R}^{ij} 
		\ = \ X(\mathcal{R}^{ij})-\partial_k^b (X^i)\mathcal{R}^{k j}
		-\partial_k^b (X^j)\mathcal{R}^{i k} \ . \notag
\end{equation}
Then the   estimate on line \eqref{R_X_sym_bnds} for $ \mathcal{R}_X^{ij}$
is immediate from the estimates \eqref{vect_sym_bnds} and \eqref{R_sym_bnds}.

\case{3b}{The $ij$ components assuming ${\partial_u^b}(X^r)\neq 0$} 
In this case we are still assuming ${\partial_u^b}(X^i-\hat{x}^i X^r)=0$.
Therefore we have:
\begin{equation}
		\mathcal{L}_X\mathcal{R}^{ij}_X 
		\ = \ X(\mathcal{R}^{ij})-\hat{x}^i {\partial_u^b} (X^r)\mathcal{R}^{u j}
		-\hat{x}^i {\partial_u^b} (X^r)\mathcal{R}^{i u}
		-\partial_k^b (X^i)\mathcal{R}^{k j}
		-\partial_k^b (X^j)\mathcal{R}^{i k} \ . \notag
\end{equation}
A $\mathcal{Z}^{-\frac12}$ symbol bound   for this expression in $r>1$
is again immediate from \eqref{vect_sym_bnds} and \eqref{R_sym_bnds}. On the other hand
all but the second and third terms above yield an improved  $\mathcal{Z}^0$ bound.
Thus, using \eqref{hat_omega_iden} we have for $r>1$:
\begin{equation}
		 \mathcal{R}^{ij}_X -\omega^i\omega^j \omega_k\omega_l  \mathcal{R}^{kl}_X
		\ \equiv \ -\omega^i {\partial_u^b} (X^r)( \mathcal{R}^{u j} - \omega^j \omega_k \mathcal{R}^{u k})
		-\omega^j {\partial_u^b} (X^r)( \mathcal{R}^{i u} - \omega^i \omega_k \mathcal{R}^{k u} )
		\mod r^{-2}\cdot \mathcal{Z}^{-\frac{1}{2}} + \mathcal{Z}^0 \ . \notag
\end{equation}
By  \eqref{R_sym_bnds} and \eqref{vect_sym_bnds} 
we have a $\mathcal{Z}^0$  bound for this last term as well.

\Part{2}{The $\mathcal{S}$ bounds involving condition \eqref{vect_sym_bnds}}
Again we can reduce to $a=b=c=0$. Componentwise we have 
$\mathcal{L}_X\mathcal{S}^\alpha = X(\mathcal{S}^\alpha)-
\partial_\beta(X^\alpha)\mathcal{S}^\beta$.

\case{1}{The $u$ component}
Here we have:
\begin{equation}
		\mathcal{L}_X\mathcal{S}^u \ = \ X(\mathcal{S}^u)-
		{\partial_u^b}(X^u)\mathcal{S}^u-
		\partial_i^b(X^u)\mathcal{S}^i \ , \notag
\end{equation}
so the second estimate on line \eqref{S_sym_bnds}  
follows directly by multiplying together the bounds on line \eqref{vect_sym_bnds}  and
\eqref{S_sym_bnds}.

\case{2}{The $i$ components}
Here we have:
\begin{equation}
		\mathcal{L}_X\mathcal{S}^i \ = \ X(\mathcal{S}^i)-
		{\partial_u^b}(X^i)\mathcal{S}^u-
		\partial_j^b(X^i)\mathcal{S}^j \ , \notag
\end{equation}
and so the first estimate on line \eqref{S_sym_bnds}  
follows by directly from 
\eqref{vect_sym_bnds}   and
\eqref{S_sym_bnds}.

\Part{3}{Estimates involving $\mathbb{L}_0$}
This is largely a corollary of \textbf{Part 1} and \textbf{Part 2} above.

First, notice  that   $X\in\mathbb{L}_0$ does not directly satisfy the first condition on line
\eqref{vect_sym_bnds}  due to terms containing the factor
$(\tau_+\partial_x)^J \omega^i$. However, the claim of Part \ref{X0_part})
 is easy to verify in $r<\frac{1}{2}t$ using the fact that for any
 $X\in \mathbb{L}_0$ one has $|(\tau_-{\partial_u^b})^l(\tau_x\partial_x^b)^J X^\alpha|\lesssim_{l,J} 1$.

Next, for any $X\in\mathbb{L}_0$ one has 
both conditions  on line \eqref{vect_sym_bnds} with $a=c=-1$ and $b=1$ when restricting
to the region $r>\frac{1}{2} t$.
In this case one also has to deal with the fact that $\partial_r^b X^u\neq 0$ and
$\partial_r^b r^{-1}(X^i-\hat x^i\hat x_jX^j) \neq 0$ save for  when $X={\partial_u^b}$. 
Recall that these two special conditions were only used in \textbf{Case 1b} and \textbf{Case 2} of \textbf{Part 1} above.
So we review those cases here when $X=\partial_i^b-\omega^i {\partial_u^b} $.

\case{1}{The $uu$ component of $\mathcal{R}_X$}
Recall from  \textbf{Case 1b}   of \textbf{Part 1} above we only need to handle
an expression of the form $\partial_r^b(X^u)(\mathcal{R}^{ur}+\mathcal{R}^{ru})$ in the region $r>1$ where
 $X^u=\omega^i$. Using
 $\partial_r^b (\omega^i)=r^{-1}\tau_x^{-2}\omega^i$ we get  an 
$\tau_x^{-3}\cdot \mathcal{Z}^{\frac{1}{2}}\subseteq \tau_x^{-1}\tau_-^{-1}\tau_+\mathcal{Z}^2$
bound for  this expression which suffices.

\case{2}{The $ui$ component of $\mathcal{R}_X$}
Recall that the condition $\partial_r^b r^{-1}(X^i-\hat x^i\hat x_jX^j) = 0$
is only used to establish the improved bound 
$\mathcal{R}_X^{iu}-\omega^i\omega_j \mathcal{R}_X^{uj}\in\mathcal{Z}^1$.
Recall further that this improved bound automatically holds modulo an expression of the form  
$     r\partial_r^b[ r^{-1} (X^i-\omega^i\omega_j X^j) ] \mathcal{R}^{ur} $.  When
$X^i=0,1$ we see  this expression has symbol bounds
on the order of $\tau_x^{-1}\cdot \mathcal{Z}^\frac{1}{2}\subseteq \tau_x^{-1}\tau_-^{-1}\tau_+\
\mathcal{Z}^1$ 
which suffices.
\end{proof}


\subsection{A General Exterior Multiplier Estimate}

Next, we prove some general   multiplier
bounds which will be used a number of times in the sequel. To state them we first
define the form of an acceptable error. 

\begin{defn}[General form of multiplier estimate errors]
For a pair of parameters $0< a<1$ and $R>0$, and a  quantity 
$o_R(1)\to 0$ as $R\to\infty$, we set:
\begin{multline}
		 \mathcal{E}(a,R) \ = \ 
		\lp{\tau_x^a\partial\phi(0)}{L^2_x(r>\frac{1}{2}R)}^2 + o_R(1)\cdot \big(
		\sup_{0\leq t\leq T}\lp{\phi(t)}{E^a(r>\frac{1}{2}R)}^2 + \lp{\phi}{S^a(r>\frac{1}{2}R)[0,T]}^2
		\big)\\ 
		+ \lp{\phi}{S^a(r>\frac{1}{2}R)[0,T]} \cdot  \big( R^{-\frac{1}{2}}
		\lp{(\tau_-^a{\partial_u^b}\phi,\tau_x^a\partial_x^b\phi, \tau_x^{a-1}\phi)}
		{L^2(\frac{1}{2}R<r<R)[0,T]}
		+\lp{\Box_g\phi}{N^a(r>\frac{1}{2}R)[0,T]} \big) \ . \label{aR_error}
\end{multline}
Corresponding to the  cases  $a=0,1$, for parameter $R>0$, quantity $o_R(1)$,
and vector field $X$, we set:
 \begin{multline}
		 \mathcal{E}(0,R,X) \ = \ 
		\lp{ \partial\phi(0)}{L^2_x(r>\frac{1}{2}R)}^2 + o_R(1)\cdot \big(
		\sup_{0\leq t\leq T}\lp{\partial \phi(t)}{L^2_x(r>\frac{1}{2}R)}^2 + \lp{\phi}{ \LE^0(r>\frac{1}{2}R)[0,T]}^2
		\big)\\ 
		+   R^{-1} \lp{(\partial \phi,\tau_x^{-1}\phi) }{L^2(\frac{1}{2}R<r<R)[0,T]}^2 +  \Big|\dint 
		 \Box_g\phi \cdot \tau_x^{-1} X(\tau_x \phi) dV_g \Big| \ , \label{0R_error}
\end{multline}
and:
 \begin{multline}
		 \mathcal{E}(1,R,X) \ = \ 
		\lp{\tau_x \partial\phi(0)}{L^2_x(r>\frac{1}{2}R)}^2 + o_R(1)\cdot \big(
		\sup_{0\leq t\leq T}\lp{\phi(t)}{E^1(r>\frac{1}{2}R)}^2 + 
		\lp{\phi}{ S^{1,\infty}(r>\frac{1}{2}R)[0,T]}^2
		\big)\\ 
		+ \lp{\phi}{ S^{1,\infty}(r>\frac{1}{2}R)[0,T]} \cdot  R^{-\frac{1}{2}}
		\lp{(\tau_-{\partial_u^b}\phi,\tau_x \partial_x^b\phi,  \phi)}
		{\ell^1_t L^2(\frac{1}{2}R<r<R)[0,T]} +
		  \Big|\dint  
		 \Box_g\phi \cdot \tau_x^{-1} X(\tau_x\phi) dV_g \Big| \ . \label{1R_error}
\end{multline}
In the above notation the rate of  $o_R(1)$ may change from line to line, but is fixed for any line
on which an error of the form $\mathcal{E}$ appears. Also we denote $dV_g=\sqrt{|g|} dxdt$
where $|g|=|\det{g}|$ is computed in $(t,x)$ coordinates.
 \end{defn}

With this notation in mind we have:

\begin{prop}[Abstract Multiplier Estimate]\label{abs_mult_prop}
Fix $R>0$ sufficiently large so that $\mathcal{K}\subset \{r<\frac{1}{2}R\}$, and
let $Y$ be a  vector field such that  $Y=Y^u{\partial_u^b} + Y^r\partial_r^b$,
with both $Y^u\geq 0$ and $Y^r\geq 0$   depending  only on the $(u,r)$ variables. 
Then the following hold:
\begin{enumerate}[I)]
\item \label{a_mult_case}
Assume  for some $0< a <1$ there holds the symbol type bounds:
\begin{equation}
		\sum_{ i+|J|\leq 1}\big| (\tau_-{\partial_u^b})^i (\tau_+ \partial_r^b)^J  Y^u\big| \ \lesssim \ 
		\tau_+^{2a}\tau_0^{\max\{1,2a\}} \ , \qquad
		\sum_{ i+|J|\leq 1}\big| (\tau_-{\partial_u^b})^i (\tau_x \partial_r^b)^J Y^r)\big|
		\ \lesssim \ \tau_x^{2a}+\tau_+^{2a-1}\tau_x  \ . \label{mult_sym_bnds}
\end{equation}
Then one has the following multiplier estimate:
\begin{multline}
		\dint   \chi_{>R}\big(
		\mathcal{A}^u  ({\partial_u^b}\phi)^2 + \mathcal{A}^{ur}  {\partial_u^b}\phi\cdot  \tau_x^{-1}\partial_r^b(\tau_x\phi)
		+\mathcal{A}^{r}(\tau_x^{-1}\partial_r^b(\tau_x\phi))^2 + 
		\slash\!\!\!\! \mathcal{A} |\snabla^b \phi|^2
		\big)dxdt\\
		+ \lp{\tau_x ^{-1}\big(\sqrt{Y^u} \partial(\tau_x \phi),
		\sqrt{Y^r} \partial_x^b(\tau_x \phi)\big)(T) }{L^2_x(r>R)}^2 
		\ \lesssim \ \mathcal{E}(a,R) \ , \label{abs_mult_est}
\end{multline}
where $|\snabla^b \phi|^2$ denotes the (Euclidean) angular gradient of 
$\phi$ with respect to the spheres $u=const$ and $r=const$,
and there $A_0$ is given by:
\begin{subequations}\label{A0_coeff}
\begin{align}
		\mathcal{A}^u \ &= \ 
		-   \partial_r^b Y^u 
		\ , 
		&\mathcal{A}^{r} \ &= \ 
		\frac{1}{2}\partial_r^b Y^r -\frac{1}{2}{\partial_u^b} Y^u-  {\partial_u^b} Y^r
		\ , \label{A0_coeff1}\\
		\mathcal{A}^{ur} \ &= \  \partial_r^b Y^u \ ,
		&\slash\!\!\!\! \mathcal{A} \ &= \ 
		r^{-1}Y^r-\frac{1}{2}{\partial_u^b} Y^u-\frac{1}{2}\partial_r^b Y^r
		\ . \label{A0_coeff2}
\end{align}
\end{subequations}
Here $\chi_{>R}=\chi_{>1}(R^{-1}\cdot)$ 
is a  radial bump function with $\chi_{>R}\equiv 1$
on $r>R$ and $\chi_{>R}\equiv 0$
on $r<\frac{1}{2}R$.
The implicit constant on line \eqref{abs_mult_est}
depends only on the bounds from line \eqref{mult_sym_bnds}, and the metric $g$.

\item \label{1_mult_case}
In the case $a=1$, still assuming \eqref{mult_sym_bnds}, we have
\eqref{abs_mult_est}--\eqref{A0_coeff} with  RHS\eqref{abs_mult_est} replaced by
$\mathcal{E}(1,R)=\mathcal{E}(1,R,\chi_{r>R}Y)$
where $\chi_{>R}$ is as above.

\item \label{0_mult_case}
Alternatively, in  the case $a=0$ replace assumption \eqref{mult_sym_bnds} with:
\begin{equation}
		\sum_{ i+|J|\leq 1}\big| (\tau_-{\partial_u^b})^i (\tau_+ \partial_r^b)^J  Y^u \big| \ \lesssim \ 
		1\  , \qquad
		\sum_{ i+|J|\leq 1}\big| (\tau_-{\partial_u^b})^i (\tau_x \partial_r^b)^J Y^r\big| \ \lesssim \ 
		1 \ , \qquad
		{\partial_u^b} Y^r \ = \ 0 \ .
		 \label{mult_sym_bnds'}
\end{equation}
Then \eqref{abs_mult_est}--\eqref{A0_coeff} hold with RHS\eqref{abs_mult_est} replaced by 
$\mathcal{E}(0,R)=\mathcal{E}(0,R,\chi_{r>R}Y)$ where $\chi_{r>R}$ is as above.
\end{enumerate}
\end{prop}

In order to prove this proposition we need a few additional supporting lemmas.

\begin{lem}[Asymptotics of the conformal potential]
Let $\Omega=\tau_x$  and define the quantity
$V=\Omega^3 \Box_g(\Omega^{-1})$. Then in Bondi coordinates
$(u,x^i)$ one has the symbol bounds:
\begin{equation}
		 V \ \in \   \mathcal{Z}^{-\frac{1}{2}}  \ . \label{V_bound}
\end{equation}
\end{lem}

\begin{proof}
First write the wave operator in Bondi coodinates as
$\Box_g = d^{-\frac{1}{2}} \Box_h  + d^{-\frac{1}{2}} \partial_\alpha^b 
\mathcal{R}^{\alpha\beta}\partial_\beta^b $, where $g=|g|$ is the Bondi
coordinate metric determinant, $\Box_h=\partial_\alpha^b h^{\alpha\beta}\partial_\beta^b$ 
where $h$ is given on line \eqref{h_tensor},
and where $\mathcal{R}=d^\frac{1}{2}g-h$
satisfies the estimate \eqref{R_sym_bnds}.
A quick calculation shows $\Box_h (\tau_x^{-1})=-3\tau_x^{-5}$, and a little further work reveals:
\begin{equation}
		  V  \ = \ - d^{-\frac{1}{2}}\big(   r \partial_\alpha^b  \mathcal{R}^{\alpha r}
		+\tau_x^{-2}(1-3r^2)    \mathcal{R}^{r r} +3\tau_x^{-2} \big) \ . \notag
\end{equation}
By \eqref{det_bnds} we have $  d^{-\frac{1}{2}}-1\in\mathcal{Z}^0$,
so estimate \eqref{V_bound} follows from \eqref{R_sym_bnds}.
\end{proof}

\begin{lem}[Formulas for boundary terms]
Let $X^r, X^u$ be non-negative   and set $X=X^u{\partial_u^b} +X^r\partial_r^b$.
Then if $\Omega=\tau_x$ one has the following pointwise
estimate involving the quantity $P(X,\chi,\Omega,\phi)$ defined on line 
\eqref{P_chi_iden}:
\begin{multline}
		X^u |\tau_x^{-1}\partial (\tau_x \phi)|^2+
		X^r | \tau_x^{-1} \partial_x^b 
		(\tau_x \phi)|^2\ \lesssim \ -    P(X,\chi,\Omega,\phi)\\
		+ o_r(1)\cdot \big( 
		( X^u +\tau_0^2 X^r)  |\tau_x^{-1}\partial(\tau_x\phi)|^2
		+  \chi (X^u+ X^r  )
		\tau_0^{-\frac{1}{2}} \tau_x^{-2} \phi^2\big)  \ . \label{boundary} 
\end{multline}
\end{lem}

To prove estimate \eqref{boundary} we  need the following elementary result:

\begin{lem}[Approximate null frame]
Let $X$ and $Y_A$, $A=1,2$, be approximately unit length vectors in the Minkowski space 
in the sense that $\sup_\alpha |X^\alpha|\approx 1$ and $\sup_\alpha |Y_A^\alpha|\approx 1$. 
Suppose that there exists $\epsilon>0$
such that $\la X,X\ra =O(\epsilon^2)$, $\la X,Y_A\ra=O(\epsilon)$, and 
in addition $| \la Y_A,Y_B\ra -\delta_{AB}|\ll 1$. 
Then there exists an exact null frame $\{L,\bL,e_A\}$ with $\la L,L\ra=\la L,e_A\ra=\la \bL, L\ra=0$,
$\la L,\bL\ra=-1$,  and 
$\la e_A,e_B\ra=\delta_{AB}$, and coefficients $\gamma$, $c_X^A$ and $c_A^B$ for $A,B=1,2$, 
such that:
\begin{equation}
		X \ = \ L +   c^A_X e_A + \gamma \bL \ , \qquad
		Y_A \ = \   c_A^B e_B \ , \qquad\quad
		\hbox{where\ \ } \gamma=O(\epsilon^2)
		\ ,
		\ \ \hbox{and\ \ }
		c_X^A = O(\epsilon) \ ,
		\ \ \hbox{and\ \ } |c^A_B - \delta^A_B|\ll 1 \ . \notag 
\end{equation}
\end{lem}

\begin{proof}
Let $e_A$ form an orthonormal basis for the space-like two plane spanned by $Y_A$, with the
first $e_A$ in the direction of one of the $Y_A$.
Let $c^B_A$ be the corresponding change of basis. Then $|c^A_B - \delta^A_B|\ll 1$. Let $L,\bL$
generate the two null directions over the span of $e_A$ and $Y_A$, chosen so that $\la L,\bL\ra=-1$ and
$X=L + c^A_X e_A + \gamma \bL$ for some set of coefficients $c_X^A,\gamma$.
From $\la X,Y_A\ra=O(\epsilon)$ we have $\la X,e_A\ra=O(\epsilon)$ and 
so $c^A_X=O(\epsilon)$. Then $\la X,X\ra = -2\gamma + O(\epsilon^2)$,  
so $\gamma=O(\epsilon^2)$ follows from $\la X,X\ra=O(\epsilon^2)$.
\end{proof}

\begin{proof}[Proof of \eqref{boundary}]
Note that it suffices to prove this bound in the region $r\gg 1$. 
Consider the vector fields $\{\partial_r^b,Y_A\}$ where $Y_A$ is a (local) Euclidean 
ONB on the spheres $r=const,u=const$. Because the metric $g$ is asymptotically 
Minkowskian  $|\la Y_A, Y_B\ra -\delta_{AB}|\ll 1$. On the other hand
a quick application of the asymptotic formulas \eqref{mod_coords} and Cramer's 
rule shows that  $\langle \partial_r^b , \partial_r^b \rangle=o_r(1)\cdot \tau_0^2$ and
$\langle \partial_r^b,Y_A \ra = o_r(1)\cdot \tau_0$. Thus, an application of the previous
lemma shows that:
\begin{equation}
		\partial_r^b \ = \ L + o_r(1)\cdot \tau_0 \slash\!\!\!\partial_x^b +o_r(1)\cdot \tau_0^2 \partial
		  \ , \notag
\end{equation}
where $L$ is null, $\slash\!\!\!\partial_x^b$ denotes derivatives  tangent to $u=const,r=const$ which
are also orthogonal to $L$, and $\partial$ is arbitrary. 

Next, let $T=T[\psi]$ denote  the energy momentum tensor of
$\psi$ with respect to the metric $g$. Because $t=const$ are uniformly spacelike when $r\gg 1$ we have:
\begin{equation}
		{T}(L,-\nabla t) \ \approx  \ |L\psi |^2 + |\snabla_x^b\psi|^2 \ , \quad
		| {T}(\slash\!\!\!\partial_x^b,-\nabla t)| \ \lesssim \ |\snabla_x^b\psi| \cdot |\partial \psi| + 
		|g^{\alpha\beta}\partial_\alpha\psi \partial_\beta\psi| \ , \quad
		| {T}(\partial,-\nabla t)| \ \lesssim \ |\partial \psi|^2 \ .
		 \notag
\end{equation}
A quick application of  \eqref{mod_coords} and the middle bound above shows that uniformly for $C>0$:
\begin{equation}
		 \tau_0 | {T}(\slash\!\!\!\partial_x^b,-\nabla t)|  
		 \ \lesssim \ 
		 C^{-1}|\partial_x^b \psi|^2 + 
		 C\tau_0^2|\partial \psi|^2 \ . \notag
\end{equation}
Combining the last three lines gives the pointwise estimate:
\begin{equation}
		|\partial_x^b\psi|^2 \ \lesssim \ T(\partial_r^b,-\nabla t)   +o_r(1)\cdot  \tau_0^2
		|\partial \psi|^2  \ . \notag
\end{equation}

In addition to this we also have by the asymptotic flatness of $g$ and standard properties of $T$:
\begin{equation}
		|\partial \psi|^2 \ \lesssim \ T({\partial_u^b}, -\nabla t) + o_r(1)\cdot |\partial \psi|^2 \ . \notag
\end{equation}

Finally, let $P(X,\chi,\Omega,\phi)$ denote the quantity defines on line \eqref{P_chi_iden}. Then:
\begin{equation}
		- P(X,\chi,\Omega,\phi) \ = \ \Omega^{-2} 
		T[\Omega^{-1}\phi](X,-\nabla t) - \frac{1}{2}\Omega^{-2}X^0 \chi V \phi^2
		\ , \qquad \hbox{where \ \ } V \ = \ \Omega^3 \Box_g (\Omega^{-1}) \ , \notag
\end{equation}
and where $X^0$ denotes the time component of $X$ in $(t,x)$ coordinates. By \eqref{u_sym_bnd} we have
$|X^0|\lesssim X^u + X^r$.
Therefore estimate \eqref{boundary} follows from the last three lines above and
estimate \eqref{V_bound}.
\end{proof}

We now return to the proof of the main result of this subsection. Because of the split form of the error terms \eqref{aR_error} and \eqref{1R_error}
there are essentially two cases.

\begin{proof}[Proof of Proposition \ref{abs_mult_prop} for $0< a<1$]
We use formalism of Section \ref{conf_mult_sect}, in particular
Lemma \ref{conf_dividen_lem}. Let $X=X_R=\chi_{>R}Y$ where $\chi_{>R}$ is
as in the statement of the proposition.
We choose $\Omega=\tau_x$ and set the auxiliary
cutoff to $\chi=0$.  
Using the divergence identity \eqref{dividen1} we need to estimate each spacetime
term given by formulas \eqref{conf_div_chi} and \eqref{AB_formulas}, as well as
the  boundary term on RHS\eqref{dividen1}. 

\step{1}{Output of the $A^{\alpha\beta}$ contraction}
We'll do this calculation by switching over to  polar Bondi coordinates $(u,r,x^A)$, where locally
we can choose $x^A$ to be two members of $\hat x^i=r^{-1}x^i$.
From lines
\eqref{polar_bondi_defiden1} and \eqref{polar_bondi_defiden2}, and expansion of
$\mathcal{L}_X h$, we may write $A^{\alpha\beta}=\chi_{>R}d^{-\frac{1}{2}}A^{\alpha\beta}_0
+d^{-\frac{1}{2}}\mathcal{R}^{\alpha\beta}$ where:
\begin{equation}
		2A^{\alpha\beta}_0 \ = \  
		\partial_\gamma^b(Y^\alpha) h^{\gamma\beta}
		+\partial_\gamma^b(Y^\beta) h^{\alpha\gamma}
		-Y( h^{\alpha\beta})
		-({\partial_u^b} Y^u + \partial_r^b Y^r  )h^{\alpha\beta} \ , \label{A0_form}
\end{equation}
and $\mathcal{R}=\mathcal{R}_0+\chi_{>R}\mathcal{R}_1$. Here 
$2\mathcal{R}_0=  {}^{(X_R)} \widehat{\pi}- \chi_{>R}{}^{(Y)} \widehat{\pi} $
which according to formula \eqref{comm_pi_form} is:
\begin{equation}
		2\mathcal{R}_0^{\alpha\beta} \ = \  
		\tau_x^{ -1}\chi_{R} 
		\big( g^{\alpha r} Y^\beta
		+  g^{\beta r} Y^\alpha -Y^r g^{\alpha\beta} \big) 
		\ , \label{R0_form}
\end{equation}
and where $\chi_R=\tau_x \partial_r \chi_{>R}$ is a smooth bump function adapted to $r\approx R$.
The second remainder term is:
\begin{equation}
		2\mathcal{R}_1 \ = \ -\mathcal{L}_{Y}(d^\frac{1}{2} g^{-1}-h) 
		- ({\partial_u^b} Y^u+\partial_r^b Y^r)(d^\frac{1}{2}g^{-1}-h)  
		- 2r^{-1}\tau_x^{-2}d^\frac{1}{2}X^r g^{-1}
		 \ . \label{R1_form}
\end{equation}

A little further computation 
shows  the coefficients of the quadratic form $A^{\alpha\beta}_0$ from line \eqref{A0_form} 
are:
\begin{subequations}\label{calA_form}
\begin{align}
		A^{uu}_0 \ &= \ \mathcal{A}^u + \tau_x^{-1}(r+\tau_x)^{-1}\partial_r^b Y^u \ ,
		  &A^{rr}_0 \ &=\   \mathcal{A}^r + \tau_x^{-1}(r+\tau_x)^{-1}{\partial_u^b} Y^r \ , \label{calA_form1}\\
		2A^{ur}_0 \ &= \ \mathcal{A}^{ur} + \tau_x^{-3}Y^r \ , 
		&A^{AB} \ &= \ r^{-2}\delta^{AB} \slash\!\!\!\! \mathcal{A} \ , \label{calA_form2}
\end{align}
\end{subequations}
while $A^{rA}_0=A^{uA}_0=0$. Here the terms $\mathcal{A}$ are given on line  \eqref{A0_coeff}. 
Recalling the definitions of the norms \eqref{Sa_def} and using conditions \eqref{mult_sym_bnds}
we have:
\begin{multline}
		\dint   \chi_{>R}\big(
		\mathcal{A}^u  ({\partial_u^b}\phi)^2 + \mathcal{A}^{ur}  {\partial_u^b}\phi\cdot  \tau_x^{-1}\partial_r^b(\tau_x\phi)
		+\mathcal{A}^{r}(\tau_x^{-1}\partial_r^b(\tau_x\phi))^2 + 
		\slash\!\!\!\! \mathcal{A} |\snabla^b \phi|^2
		\big)dxdt \\
		 \leq \ \dint   \chi_{>R} \tau_x^{-2} d^{-\frac{1}{2}}
		A_0^{\alpha\beta}\partial_\alpha^b (\tau_x \phi)\partial_\beta^b (\tau_x \phi)dV_g + 
		o_R(1)\lp{\phi}{S^a(r>\frac{1}{2}R)[0,T]}^2
		 \ . \label{A_to_calA}
\end{multline}

To estimate the  remainder terms  from line \eqref{R0_form}, note that a straight forward
calculation involving the conditions \eqref{mult_sym_bnds} and \eqref{mod_coords} gives
in rectangular Bondi coordinates:
\begin{equation}
		|\mathcal{R}^{ij}_0| \ \lesssim \ (\tau_x ^{2a-1}+\tau_+^{2a-1})\chi_R \ , \qquad
		 |\mathcal{R}_0^{ui}|\ \lesssim \ \tau_x^{-1}\tau_+^{2a}\tau_0\chi_R\ , \qquad
		 |\mathcal{R}_0^{uu}| \ \lesssim \ \tau_x^{ -1}\tau_+^{2a}  \tau_0^{\max \{1,2a\}}\chi_R
		  \ , \label{R0_sym_bnds}
\end{equation}
where  $\chi_R$ is supported in $\frac{1}{2}R<r<R$.
This yields the estimate:
\begin{equation}
		\int_0^T\!\!\!\!\!\int 
		\big|\tau_x^{-2}\mathcal{R}_0^{\alpha\beta} \partial_\alpha(\tau_x\phi)
		\partial_\beta(\tau_x\phi)\big| dxdt  \lesssim  R^{-\frac{1}{2}} \slp{\phi}{S^a(\frac{1}{2}R<r<R)[0,T]} 
		\slp{(\tau_-^a{\partial_u^b}\phi,\tau_x^a\partial \phi, \tau_x^{a-1}\phi)}
		{L^2(\frac{1}{2}R<r<R)[0,T]} . \label{R0_est}
\end{equation}

To estimate the  remainder terms  from line \eqref{R1_form} we split the range into $0<a<\frac{1}{2}$
and $\frac{1}{2}\leq a<1$. When  $0<a<\frac{1}{2}$ the conditions \eqref{mult_sym_bnds} imply:
\begin{equation}
		\sum_{i+|J|\leq 1}\big| (\tau_-{\partial_u^b})^i (\tau_+ \partial_x^b)^J 
		Y^u\big| \ \lesssim \ \tau_x^{2a-1}\cdot  \tau_- \ , \qquad
		\sum_{i+|J|\leq 1}\big| (\tau_-{\partial_u^b})^i (\tau_x\partial_x^b)^J  Y^r\big|
		\ \lesssim \  \tau_x^{2a-1}\cdot \tau_x \ . \notag
\end{equation}
In addition we must assume $\partial_r^b Y^u\neq 0$ and ${\partial_u^b} Y^r\neq 0$,
although we do have ${\partial_u^b} (Y^i-\hat x^i Y^r)= 0$.  
Therefore by simultaneously combining  cases \ref{R_alt1}) and \ref{R_alt2}) of Lemma 
\ref{basic_lie_lem} and Remark \ref{zero_diff_rem},
using $|{\partial_u^b} Y^u+\partial_r^b Y^r|\lesssim \tau_x^{2a-1}$ (again for $0<a<\frac{1}{2}$) 
and \eqref{R_sym_bnds}, 
and directly using \eqref{mod_coords} for the last term on RHS\eqref{R1_form} 
we have for $0<a<\frac{1}{2}$:
\begin{subequations}\label{R1_sym_bnds}
\begin{align}
		\LLp{w_a^{-1} \mathcal{R}^{rr}_1}{-\frac{1}{2},0} \ &<\ \infty  \ , 
		&\LLp{w_a^{-1} \mathcal{R}^{uu}_1}{\max\{1,2a\},0}  \ &<\  
		\infty \ , 
		&\LLp{w_a^{-1}\mathcal{R}^{ru}_1}{\frac{1}{2},0} \ &<\   \infty  \ , \label{R1_sym_bnds1}\\
		\LLp{w_a^{-1} r\mathcal{R}^{uA}_1}{1,0}  \ &<\  
		\infty \ , \ \
		&\LLp{w_a^{-1} (r\mathcal{R}^{rA}_1,r^{2}\mathcal{R}^{AB}_1)}{0,0}
		 \ &<\  \infty \ , 
		   &\hbox{where \ \ } w_a&=(\tau_x^{2a-1}+\tau_+^{2a-1})
		  \ . \label{R1_sym_bnds2}
\end{align}
\end{subequations}
On the other hand in the range $\frac{1}{2}\leq a\leq1$ the conditions \eqref{mult_sym_bnds} imply:
\begin{equation}
		\sum_{i+|J|\leq 1}\big| (\tau_-{\partial_u^b})^i (\tau_+ \partial_x^b)^J 
		Y^u\big| \ \lesssim \ \tau_-^{2a-1}\cdot  \tau_- \ , \qquad
		\sum_{i+|J|\leq 1}\big| (\tau_-{\partial_u^b})^i (\tau_x\partial_x^b)^J  Y^r\big|
		\ \lesssim \  \tau_+^{2a-1}\cdot \tau_x \ . \notag
\end{equation}
Therefore by separately 
applying cases \ref{R_alt1}) and \ref{R_alt2}) of Lemma \ref{basic_lie_lem} and Remark \ref{zero_diff_rem}
to the vector fields $Y^u{\partial_u^b}$ and $Y^r\partial_r^b$ (resp), and this time using
$|{\partial_u^b} Y^u|+|\partial_r^b Y^r|\lesssim \tau_+^{2a-1}$
we again have \eqref{R1_sym_bnds}. Finally, 
after several rounds of H\"older's inequality and a straightforward check
of definitions \eqref{Zk_defn} and \eqref{Sa_def}, the error bounds \eqref{R1_sym_bnds}
 yields the following asymptotic estimate:
\begin{equation}
		\dint \chi_{>R}\big|\tau_x^{-2}\mathcal{R}_1^{\alpha\beta} \partial_\alpha(\tau_x\phi)
		\partial_\beta(\tau_x\phi)\big| dxdt \ \lesssim \ o_R(1)\lp{\phi}{S^a(r>\frac{1}{2}R)[0,T]}^2 \ . \label{R1_est}
\end{equation}
 
As a last step  we combine  estimates \eqref{A_to_calA}, \eqref{R0_est}, and \eqref{R1_est},  while recalling that
$d^\frac{1}{2} A^{\alpha\beta}=\chi_{>R} A^{\alpha\beta}_0
+\mathcal{R}_0+\chi_{>R}\mathcal{R}_1$. This gives us:
\begin{multline}
		\dint   \chi_{>R}\big(
		\mathcal{A}^u  ({\partial_u^b}\phi)^2 + \mathcal{A}^{ur}  {\partial_u^b}\phi\cdot  \tau_x^{-1}\partial_r^b(\tau_x\phi)
		+\mathcal{A}^{r}(\tau_x^{-1}\partial_r^b(\tau_x\phi))^2 + 
		\slash\!\!\!\! \mathcal{A} |\snabla^b \phi|^2
		\big)dxdt \ \leq \\
		\dint 
		\tau_x^{-2}A^{\alpha\beta}\partial_\alpha(\tau_x\phi)\partial_\beta(\tau_x\phi)
		dV_g +
		\mathcal{E}(a,R) \ , \label{A_output}
\end{multline}
where the coefficients $\mathcal{A}$ are given by lines \eqref{A0_coeff}.

\step{2}{Estimating the $C^\chi$ term}
Using \eqref{V_bound} we have $C^\chi \in \tau_x^{-2}\mathcal{Z}^{-\frac{1}{2}}$, while \eqref{mult_sym_bnds}
give the pointwise estimate  $|\tau_x^{-1}X_R(\tau_x\phi)|\lesssim \chi_{>R}\tau_+^{2a}
( \tau_0^{\max\{1,2a\}}|{\partial_u^b}\phi |
+|\tau_x^{-1}\partial_r^b(\tau_x\phi)|)$. Thus H\"older's inequality and \eqref{Sa_def} gives:
\begin{equation}
		 \big| \dint  
		 C^\chi \phi \tau_x^{-1}X_R (\tau_x \phi)
		 dV_g \big|  \ \lesssim \ o_R(1)\lp{\phi}{S^a[0,T](r>\frac{1}{2}R)}^2
		 \ . \label{C_output}
\end{equation}

\step{3}{Output of the boundary terms}
Using the bound from line \eqref{boundary} we 
directly have:
\begin{multline}
		\lp{\tau_x^{-1}\big(\sqrt{Y^u} \partial(\tau_x\phi),
		\sqrt{Y^r} \partial_x^b(\tau_x\phi)\big)(T) }{L^2_x(r>R)}^2 \ \lesssim\
		\lp{\tau_x^a\partial \phi(0)}{L^2_x}^2\\
		 +o_R(1)\sup_{0\leq t\leq T}\lp{\phi(t)}{E^a(r>\frac{1}{2}R)}^2-
		  \int_{\mathbb{R}^3}  P(X,\chi,\Omega,\phi)\sqrt{|g|}dx \big|_{t=0}^{t=T} \ . \label{boundary_output}
\end{multline}
 
\step{4}{Output of the source term}
Finally, another application of  H\"older's inequality  shows:
\begin{equation}
		  \big| \dint 
		  \Box_g\phi \cdot \tau_x^{-1}X_R(\tau_x\phi) 
		 dV_g\big|  \ \lesssim \ 
		 \lp{\Box_g}{N^a[0,T]}\lp{\phi}{S^a(r>\frac{1}{2}R)[0,T]}
		  \ . \label{source_output}
\end{equation}

Adding together formulas \eqref{A_output}--\eqref{source_output} and using  
\eqref{dividen1}, \eqref{conf_div_chi}, and \eqref{AB_formulas} gives
\eqref{abs_mult_est}.
\end{proof}


For the proof of Proposition \ref{abs_mult_prop} with  $a=1$ it will help to have the
following Lemma:

\begin{lem}\label{L2L00_lem}
Let $\Phi$ be a bounded function supported on $[0,T]$,
and let $\LLp{p^2}{0,0}<\infty$. Then for $R>0$ one has the estimate:
\begin{equation}
		\lp{\sqrt{\tau_x/\tau_+}p\cdot\Phi}{L^2(r>R)} \ 
		\lesssim \ o_R(1)\big( \lp{(1- \chi_{\frac{1}{2}t<r<2t})\Phi}{\ell^\infty_t\ell^\infty_r L^2}
		+ \lp{ \chi_{\frac{1}{2}t<r<2t} \Phi}{\ell^\infty_t\ell^\infty_u\ell^\infty_r L^2} +
		\lp{\tau_x^\frac{1}{2}\Phi }{L^\infty_t L^2_x}\big) \ . \label{Phi_L2_to_Linfty}
\end{equation}
\end{lem}

\begin{proof}[Proof of \eqref{Phi_L2_to_Linfty}]
We
split  the RHS into regions $r<\frac{1}{2}t$,  $ \frac{1}{2}t<r<2t$, and $r>2t$, 
all restricted to $r>R$ 
(we will largely suppress this last condition in the following notation).

In the first region we use:
\begin{equation}
		\lp{\sqrt{\tau_x /\tau_+}p\cdot \Phi}{L^2( r<\frac{1}{2}t)} \ \lesssim \ \lp{ 
		\sqrt{\tau_x /\tau_+}p}{\ell^2_t\ell^2_r L^\infty(r<\frac{1}{2} t)}
		\lp{\Phi}{\ell^\infty_t\ell^\infty_r L^2( r<\frac{1}{2}t)} \ ,  \notag
\end{equation}
followed by Young's inequality which gives $\lp{\sqrt{\tau_x /\tau_+}p}{\ell^2_t\ell^2_r L^\infty(r<\frac{1}{2}t)}\lesssim 
\lp{ p}{\ell^2_r L^\infty(r>R)}=o_R(1)$. 

In the region $ \frac{1}{2}t<r<2t$ we  use:
\begin{equation}
		\lp{\sqrt{\tau_x /\tau_+}p\cdot \Phi}{L^2( \frac{1}{2}t<r<2t)} \ \lesssim \ \lp{ 
		  p}{\ell^2_u \ell^2_r L^\infty ( \frac{1}{2}t<r<2t)}
		\lp{\Phi}{\ell^\infty_t\ell^\infty_u \ell^\infty_r L^2( \frac{1}{2}t<r<2t)} \ , \notag
\end{equation}
followed by $\lp{ p}{\ell^2_u\ell^2_r L^\infty(\frac{1}{2}t<r<2t)(r>R)}=o_R(1)$.

Finally, in the region $r>2t$ we use:
\begin{equation}
		\lp{\sqrt{\tau_x /\tau_+} p\cdot \Phi}{L^2( r>2t)} \ \lesssim \ \lp{ \tau_x^{-\frac{1}{2}}p}{L^2_t L^\infty_x(r>2t)}
		\lp{\tau_x^\frac{1}{2} \Phi}{L^\infty_t L^2_x} \ , \notag
\end{equation}
followed by 
$\lp{ \tau_x^{-\frac{1}{2}}p}{L^2_t L^\infty_x(r>2t)}\lesssim  
\lp{\chi_{t<\frac{1}{2}r } \tau_x^{-\frac{1}{2}}p}{\ell^2_r L^2_t L^\infty_x } \lesssim
\lp{p}{\ell^2_r L^\infty(r>R)}=o_R(1)$.
\end{proof}


\begin{proof}[Proof of Proposition \ref{abs_mult_prop} for $a=1$]
The demonstration is largely similar to the previous proof, with a few key
differences. We again choose $\Omega=\tau_x$, but this time set the auxiliary
cutoff in Lemma \ref{conf_dividen_lem} to be $\chi=\chi_{<\frac{1}{2}}(r/t)$
which vanishes in $r>\frac{3}{4}t$ with $\chi\equiv 1$
when $r<\frac{1}{2}t$.

\step{1}{Output of the $A^{\alpha\beta}$ contraction}
We again have formulas \eqref{A0_form}--\eqref{calA_form}.

An inspection of the remainder terms on RHS\eqref{calA_form}
using condition \eqref{mult_sym_bnds} shows that we also have
estimate \eqref{A_to_calA} with the last RHS term replaced by
$o_R(1)\lp{\phi}{S^{1,\infty}(r>\frac{1}{2}R)[0,T]}^2$.

Next, the estimates \eqref{R0_sym_bnds} are again valid except this
time we use H\"older's inequality to replace \eqref{R0_est} with:
\begin{equation}
		\int_0^T\!\!\!\!\!\int  \big|\tau_x^{-2}\mathcal{R}_0^{\alpha\beta} \partial_\alpha(\tau_x\phi)
		\partial_\beta(\tau_x\phi)\big| dxdt \ \lesssim \ R^\frac{1}{2}
		\slp{\phi}{ S^{1,\infty}(\frac{1}{2}R<r<R)[0,T]} 
		\slp{(\tau_- {\partial_u^b}\phi,\tau_x \partial\phi, \phi)}
		{\ell^1_t L^2(\frac{1}{2}R<r<R)[0,T]}  . \label{R0_est1}
\end{equation}

To estimate the  remainder term involving $\mathcal{R}_1^{\alpha\beta}$ note that 
\eqref{R1_sym_bnds} is still valid with $a=1$. This allows us to replace \eqref{R1_est}
in the case $a=1$ with the slight improvement:
 \begin{multline}
		\dint  \chi_{>R}\big|\tau_x^{-2}
		\mathcal{R}_1^{\alpha\beta} \partial_\alpha(\tau_x \phi)
		\partial_\beta(\tau_x\phi)\big| dxdt \lesssim \\
		\slp{\sqrt{\tau_x/\tau_+}p\cdot  \tau_x^{-\frac{1}{2}}\tau_+(\tau_0 \partial \phi ,  \partial_x^b \phi,
		\tau_0^{-\frac{1}{2}} \tau_x^{-1}\partial_r^b(\tau_x\phi),\tau_x^{-1}\phi)}{L^2(r>\frac{1}{2}R)[0,T]}^2 \ , \label{R1_est1'}
\end{multline}
where  $\LLp{p^2}{0,0}<\infty$. Then an application of Lemma \ref{L2L00_lem} produces:
\begin{equation}
		LHS\eqref{R1_est1'} \ \lesssim \ o_R(1)\cdot \big(
		\sup_{0\leq t\leq T}\lp{\phi(t)}{E^1(r>\frac{1}{2}R)}^2 + \lp{\phi}{ S^{1,\infty}
		(r>\frac{1}{2}R)[0,T]}^2
		\big) \ . \ \label{R1_est1}
\end{equation}

Combining \eqref{R0_est1}, \eqref{R1_est1}, and the analog of \eqref{A_to_calA},
we have estimate \eqref{A_output}   for the case $a=1$.

\step{2}{Estimating the $B^\chi$ term}
Inspection of  \eqref{AB_formulas} and  estimate \eqref{V_bound} shows
 the main thing is to compute $trace(A)$.
By definition $trace {}^{(X_R)} \widehat{\pi}=-2\nabla_\alpha X^\alpha_R$, so   $trace(A)=4r\tau_x^{-2}X^r_R-\nabla_\alpha X^\alpha_R$.
Using the conditions \eqref{mult_sym_bnds}  and the support of $\chi$,
we conclude $\LLp{\tau_x^{2}\tau_+^{-1}B^\chi}{0,0}<\infty$. Thus:
\begin{equation}
		 \big| \dint 
		 B^\chi \phi^2
		 dV_g \big|  \ \lesssim \ 
		 \lp{\sqrt{\tau_x /\tau_+}p\cdot \tau_x^{-\frac{1}{2}}  \tau_+( \tau_x^{-1} \phi)}{L^2(r>\frac{1}{2}R)[0,T]}^2
		 \ , \label{B_output}
\end{equation}
where $\LLp{p^2}{0,0}<\infty$. From this Lemma \ref{L2L00_lem} produces:
\begin{equation}
		LHS\eqref{B_output}\ \lesssim\  o_R(1)(  \lp{\phi}{E^{1}(r>\frac{1}{2}R)[0,T]}^2+
		\lp{\phi}{S^{1,\infty}(r>\frac{1}{2}R)[0,T]}^2) \ . \notag
\end{equation}

\step{3}{Estimating the $C^\chi$ term}
Using \eqref{V_bound} and the support property of 
$\chi$ we have $\tau_+^2 C^\chi \in \mathcal{Z}^{-\frac{1}{2}}$, while \eqref{mult_sym_bnds}
give the pointwise bound $|\tau_x^{-1}X_R(\tau_x\phi)|\lesssim \chi_{>R}\tau_+^{2}( \tau_0^2 |{\partial_u^b}\phi |
+|\tau_x^{-1}\partial_r^b(\tau_x\phi)|)$. Thus:
\begin{equation}
		 \big| \dint 
		 C^\chi \phi \tau_x^{-1}X_R (\tau_x\phi)
		 dV_g \big|  \ \lesssim \ 
		 \slp{\sqrt{\tau_x/\tau_+}p\cdot \tau_x^{-\frac{1}{2}} \tau_+ (\tau_0 \partial \phi,  
		\tau_0^{-\frac{1}{2}} \tau_x^{-1}\partial_r^b(\tau_x\phi),\tau_x^{-1}\phi)}{L^2(r>\frac{1}{2}R)[0,T]}^2
		 \ , \notag
\end{equation}
where again $p^2\in \mathcal{Z}^0$ so we  conclude via Lemma \ref{L2L00_lem}.

\step{4}{Output of the boundary terms}
Here we simply note that \eqref{boundary_output} is also valid for $a=1$.
 
\step{5}{Output of the source term}
This term is included directly in  the definition of $\mathcal{E}(1,R,X)$.

\end{proof}


\begin{proof}[Proof of Proposition \ref{abs_mult_prop} for $a=0$]
This follows the pattern of the previous two proofs. We set $\Omega=\tau_x$ and choose $\chi=0$.

\step{1}{Output of the $A^{\alpha\beta}$ contraction}
We again have formulas \eqref{A0_form}--\eqref{calA_form}.
This time one replaces 
\eqref{R0_sym_bnds} with $|\mathcal{R}_0^{\alpha\beta}|\lesssim \tau_x^{-1}\chi_R$,
in which case RHS\eqref{R0_est} becomes  $R^{-\frac{1}{2}}
\lp{(\partial \phi,\tau_x^{-1}\phi) }{L^2(\frac{1}{2}R<r<R)[0,T]}^2$.
The analog of \eqref{A_to_calA} is also valid with the second RHS term replaced by 
$o_R(1)\lp{\phi}{ \LE^0(r>\frac{1}{2}R)[0,T]}^2$.

The main difference is that this time the conditions \eqref{mult_sym_bnds'} give  
$|{\partial_u^b} X^u+\partial_r^b X^r|\lesssim \tau_x^{-1}\tau_0^{-1}$. Together with the condition ${\partial_u^b} Y^r=0$
and an application of Lemma \eqref{basic_lie_lem},
this means we need to replace \eqref{R1_sym_bnds} with:
\begin{equation}
		\LLp{\tau_x \mathcal{R}^{ij}_1}{-1,0} \ < \ \infty  \ , \qquad
		\LLp{\tau_x \mathcal{R}^{ui}_1}{-\frac{1}{2},0} \ < \ \infty  \ , \qquad
		\LLp{\tau_x \mathcal{R}^{uu}_1}{\frac{1}{2},0} \ < \ \infty  \ . \notag
\end{equation}
This is enough to show the analog of \eqref{R1_est} with RHS replaced by 
$o_R(1)\lp{\phi}{ \LE^0(r>\frac{1}{2}R)[0,T]}^2$.

\step{2}{Estimating the $C^\chi$ term}
Using \eqref{V_bound} gives  $C^\chi \in \tau_x^{-2}\mathcal{Z}^{-\frac{1}{2}}$, while \eqref{mult_sym_bnds'}
gives $|\tau_x^{-1}X_R(\tau_x\phi)|\lesssim \chi_{r>R}|(\partial\phi,\tau_x^{-1}\phi)|$.
Thus \eqref{C_output} is valid with RHS replaced by $o_R(1)\lp{\phi}{ \LE^0(r>\frac{1}{2}R)[0,T]}^2$.

\step{4}{Output of the boundary terms}
Here we  note that \eqref{boundary_output} is also valid for $a=0$.
 
\step{5}{Output of the source term}
This term is included directly in  the definition of $\mathcal{E}(0,R,X)$.
\end{proof}

 

\subsection{Abstract Bounds for Commutators}

We now turn to some further consequences of Lemma \ref{basic_lie_lem}.

\begin{lem}[Abstract bounds for  commutators]\label{abs_comm_prop}
Let $\mathcal{R}$ and $\mathcal{S}$ be a contravariant 2-tensor and vector field (resp). From
them define the operator   $\mathcal{Q}=\nabla_\alpha \mathcal{R}^{\alpha\beta} \nabla_\beta + 
\mathcal{S}^\alpha\nabla_\alpha$. Then the following results hold:
\begin{enumerate}[I)]
	\item \label{Q_X_part} Suppose that $\mathcal{R}$ and $\mathcal{S}$ satisfy \eqref{R_sym_bnds} 
	and \eqref{S_sym_bnds} (resp). Then if $X$ is any vector field which
	satisfies \eqref{vect_sym_bnds} and \eqref{good_X_conds} with $a=b=c=0$,  one has
	$[X,\mathcal{Q}]=\nabla_\alpha \td{\mathcal{R}}^{\alpha\beta} \nabla_\beta + 
	\td{\mathcal{S}}^\alpha\nabla_\alpha$ where $\td{\mathcal{R}}$ and 
	$\td{\mathcal{S}}$ satisfy \eqref{R_sym_bnds} 
	and \eqref{S_sym_bnds} (resp) as well. 
	\item \label{Q_X_part'}  Alternately suppose $X\in \mathbb{L}_0=\{{\partial_u^b} , \partial_i^b-\omega^i{\partial_u^b}\}$
	with the same conditions on  $\mathcal{R}$ and $\mathcal{S}$ as above.
	Then the previous result holds with  bound  \eqref{R_X_sym_bnds}  for 
	$  \td{\mathcal{R}}$ and bound \eqref{S_X_sym_bnds} for $  \td{\mathcal{S}}$ with $a=c=-1$ and $b=1$.
	\item For any  $\mathcal{Q}$ 
	as defined above with  $\mathcal{R}$ and $\mathcal{S}$ satisfying \eqref{R_sym_bnds} 
	and \eqref{S_sym_bnds}
	we have the  pointwise bound:
	\begin{equation}
		|\mathcal{Q}\phi| \ \lesssim \ q \cdot \tau_x^{-2} 
		\sum_{1\leq l+|J|\leq 2 }
		| (\tau_x \tau_0{\partial_u^b})^l(\tau_x \partial_x^b)^J  \phi| \ ,
		\qquad \hbox{where\ \ } q\in \mathcal{Z}^{-\frac{1}{2}} \ . 
		 \label{Q_bondi_est}
	\end{equation}
	Alternately, suppose $\mathcal{R}$ and $\mathcal{S}$ satisfying \eqref{R_X_sym_bnds} 
	and \eqref{S_X_sym_bnds} with $a=c=-1$ and $b=1$. Then:
	\begin{equation}
		|\mathcal{Q}\phi| \ \lesssim q\cdot \tau_x^{-1} 
		\sum_{|I|\leq 1}( \tau_0^\frac{1}{2}  | {\partial_u^b} \Gamma^I  \phi|
		+ | \partial_x^b \Gamma^I  \phi| ) \ , 
		\qquad \hbox{where\ \ } q\in  \mathcal{Z}^{-1} \ ,
		\label{Q_bondi_est'}
	\end{equation}
	and where $\Gamma\in  \mathbb{L}_0=\{{\partial_u^b} , \partial_i^b-\omega^i{\partial_u^b}\}$.
	\item \label{Q_alg_prop} 
	Finally, let $ \mathcal{R}$ satisfy any combination of  conditions \eqref{R_X_sym_bnds}, 
	\eqref{Ruu_alt}, and \eqref{alt_ij}, and let $\mathcal{S}$ satisfy 
	\eqref{S_X_sym_bnds}.
	Let $w$ be a smooth weight function with:
	\begin{equation}
		|(\tau_- {\partial_u^b})^l(\tau_x\partial_x^b)^J w| \ 
		\lesssim_{l,J} \ \tau_x^{a'}\tau_+^{b'}\tau_-^{c'} \ .
		\label{w_sym_bnds}
	\end{equation}
	Then  with $\mathcal{Q}$ as defined above 
	we have $w\mathcal{Q}=\nabla_\alpha \td{\mathcal{R}}^{\alpha\beta} \nabla_\beta + 
	\td{\mathcal{S}}^\alpha\nabla_\alpha$, where $ \td{\mathcal{R}}^{\alpha\beta}$
	satisfies the appropriate combination of conditions \eqref{R_X_sym_bnds},
	\eqref{Ruu_alt}, and \eqref{alt_ij}, and $\td{\mathcal{S}}^\alpha$ satisfies
	\eqref{S_X_sym_bnds}, in each case with coefficients
	$a+a'$, $b+b'$, and $c+c'$.
\end{enumerate}
\end{lem}

\begin{proof}
We'll show each  part separately.

\Part{1}{The commutator property for $X$ satisfying \eqref{vect_sym_bnds} and \eqref{good_X_conds}}
By formula \eqref{R_comm_form} and parts \ref{main_LR_estimate}) and \ref{main_LS_estimate})
of Lemma \ref{basic_lie_lem},
it suffices to show that   $\mathcal{R}^{\alpha\beta}\nabla_\alpha(\nabla_\gamma X^\gamma)$
satisfies   \eqref{S_sym_bnds}. We'll do this in a bit more generality here for use in the sequel. 

Let $w$ be any weight function which satisfies \eqref{w_sym_bnds}, and let $\mathcal{R}$ be any 
contravariant two-tensor which satisfies the weaker conditions \eqref{Ruu_alt} and \eqref{alt_ij},
and the remaining conditions on line \eqref{R_X_sym_bnds}.
Then we claim    \eqref{S_X_sym_bnds} holds for 
$\mathcal{S}^\alpha = \mathcal{R}^{\alpha\beta}\nabla_\alpha(w)$ with weights $a+a',b+b',c+c'$. 
To see this note:
\begin{equation}
		\mathcal{S}^i \ = \ \mathcal{R}^{u i}{\partial_u^b} w +
		\mathcal{R}^{j i}\partial_j^b (w) \ , 
		\qquad \mathcal{S}^u \ = \ \mathcal{R}^{u u}{\partial_u^b} w +
		\mathcal{R}^{j u}\partial_j^b (w) \ , \notag
\end{equation}
so the desired estimates follow easily by multiplying together  bounds
\eqref{w_sym_bnds} and the appropriate combination of \eqref{R_X_sym_bnds},
\eqref{Ruu_alt}, and \eqref{alt_ij}. 

To show Part \ref{Q_X_part}) of the Lemma
note that if $X$ is as stated, then  \eqref{det_bnds} shows $w=\nabla_\gamma X^\gamma$
satisfies \eqref{w_sym_bnds} with $a'=b'=c'=0$. The desired result now 
follows from the discussion of the previous paragraph.

\Part{2}{The commutator property for $X\in\mathbb{L}_0$}
This follows at once from part \ref{X0_part}) of Lemma \ref{basic_lie_lem}
and  the main  calculation of \textbf{Part 1} above. Note that if $X\in\mathbb{L}_0$
then  $w= \nabla_\gamma X^\gamma$ satisfies \eqref{w_sym_bnds} with $a'=c'=-1$ and $b'=1$.

\Part{3}{The pointwise estimates \eqref{Q_bondi_est} and \eqref{Q_bondi_est'}}
The bound for the $\mathcal{S}$ potion of $\mathcal{Q}\phi$ 
follows at once from \eqref{S_X_sym_bnds}. For the $\mathcal{R}$ contraction we write
in Bondi coordinates:
\begin{equation}
		\nabla_\alpha \mathcal{R}^{\alpha\beta}\nabla_\beta\phi \ = \ 
		\mathcal{R}^{\alpha\beta}\partial_\alpha^b ( \ln\sqrt{|g|}) \partial_\beta^b \phi
		+ (\partial_\alpha^b \mathcal{R}^{\alpha\beta})\partial_\beta^b \phi +
		\mathcal{R}^{\alpha\beta}\partial_\alpha^b \partial_\beta^b \phi \ = \ 
		\td{S}^\alpha\partial_\alpha^b u + \mathcal{R}^{\alpha\beta}\partial_\alpha^b \partial_\beta^b \phi
		\ . \notag
\end{equation}
We only need to show that $\td{S}$ satisfies \eqref{S_X_sym_bnds} with $a=b=c=0$ in case of 
estimate \eqref{Q_bondi_est}, and $a=c=-1$, $b=1$ in case of estimate \eqref{Q_bondi_est};  then
study $\mathcal{R}^{\alpha\beta}\partial_\alpha^b \partial_\beta^b \phi$.

For the first term of $\td{S}$
we use the fact that $w=\ln\sqrt{|g|}$ satisfies \eqref{w_sym_bnds} with $a'=b'=c'=0$, which follows from
\eqref{det_bnds}. Then by the 
main calculation of \textbf{Part 1}  above we have \eqref{S_X_sym_bnds} for
$\mathcal{R}^{\alpha\beta}\partial_\alpha^b ( \ln\sqrt{|g|})$.

For the expression $ \partial_\alpha^b\mathcal{R}^{\alpha\beta} $ the appropriate version of
\eqref{S_X_sym_bnds} follows at once from \eqref{R_X_sym_bnds}.

For the final term note that if $\mathcal{R}$ satisfies \eqref{R_X_sym_bnds} then
one has the pointwise estimate:
\begin{equation}
		|\mathcal{R}^{\alpha\beta}\partial_\alpha^b\partial_\beta^b \phi| \ \lesssim \
		q\cdot \tau_x^{a}\tau_+^b\tau_-^c \big( \tau_0^\frac{5}{2} |(\partial_u^b)^2 \phi|
		+ \tau_0 |{\partial_u^b}\partial_x^b\phi|
		+ \tau_0^\frac{1}{2} |(\partial_x^b)^2\phi|
		\big) \ , \qquad \hbox{where \ \ } q\in \mathcal{Z}^{-\frac{1}{2}} \ . \notag
\end{equation}
This is bounded by RHS\eqref{Q_bondi_est} when $a=b=c=0$, and 
 RHS\eqref{Q_bondi_est'} when $a=c=-1$ and $b=1$.

\Part{4}{Proof of the algebra property \ref{Q_alg_prop})}
It is immediate that bound \eqref{S_X_sym_bnds}
for $\mathcal{S}$ is stable under   multiplication by $w$ satisfying \eqref{w_sym_bnds} with
the appropriate change of weights. For the quadratic term
of $\mathcal{Q}$ we write $w\nabla_\alpha \mathcal{R}^{\alpha\beta}\nabla_\beta=
 \nabla_\alpha w\mathcal{R}^{\alpha\beta}\nabla_\beta-  \mathcal{R}^{\alpha\beta}\nabla_\alpha(w)
 \nabla_\beta$. For the first  term we use  bounds \eqref{R_X_sym_bnds} which are also stable
 under multiplication by $w$. For the second  term we use the main calculation  of \textbf{Part 1} 
above.

\end{proof}

Parts \ref{Q_X_part}) and \ref{Q_X_part'}) of the last lemma imply the following:

\begin{cor}[Estimates for multicommutators]\label{multicom_coro}
Let $g$ be a metric which satisfies \eqref{mod_coords}, and as usual set
$ \mathbb{L}_0=\{{\partial_u^b} , \partial_i^b-\omega^i{\partial_u^b}\}$ and 
$ \mathbb{L}=\{S,\Omega_{ij} \}\cup \mathbb{L}_0$.
Then the following hold:
\begin{enumerate}[I)]
		\item \label{L_multicom_part} If $I$ is any multiindex then for products of vector fields
		in  $ \mathbb{L}$ one has the identity:
		\begin{equation}
				[\Box_g,\Gamma^I] \ = \ \sum_{I'\subsetneq I}\big( \big[ \nabla_\alpha
				\mathcal{R}_{I'}^{\alpha\beta}\nabla_\beta
				 + \mathcal{S}_{I'}^\alpha\nabla_\alpha\big] \Gamma^{I'}
				+w_{I'}\Gamma^{I'} \Box_g \big) \ , \label{mult_comm_iden1}
		\end{equation}
		where the sum is taken over the collection of all multiindices
		$I'$ strictly contained in $I$; in particular each $|I'|\leq |I|-1$. Here
		$\mathcal{R}_{I'}$ and $ \mathcal{S}_{I'}$ satisfy \eqref{R_sym_bnds} and 
		\eqref{S_sym_bnds} (resp), while there exists constants $w^0_{I'}\in\mathbb{R}$ such that:
		\begin{equation}
			  w_{I'}- w^0_{I'} \ \in \ \mathcal{Z}^0 \ . \label{w_sym_bnds'}
		\end{equation}
		\item  \label{L0_multicom_part}
		If the product in the previous part is restricted to vector fields in $\mathbb{L}_0$,
		then one has identity \eqref{mult_comm_iden1}  with estimates 
		 \eqref{R_X_sym_bnds} and \eqref{S_X_sym_bnds} for 
		 $ \mathcal{R}_{I'}$ and  $ \mathcal{S}_{I'}$ with $a=c=-1$ and $b=1$.
		 In addition \eqref{w_sym_bnds'} in this case is replaced by
		 $w_{I'}  \in  \tau_x^{-1} \cdot \mathcal{Z}^{-1}$.
		  \item  \label{nabla_L_multicom_part} Let $\Gamma^I$ be a product of vector fields
		 in $\mathbb{L}$ and $\td{\Gamma}^J$ a product of vector fields in $\mathbb{L}_0$.
		 Then one has the identity:
		 \begin{equation}
				\td{\Gamma}^J[\Box_g,\Gamma^I] 
				\ = \ \sum_{I'\subsetneq I , \ J'\subseteq J} \Big( \big[ \nabla_\alpha
				\mathcal{R}_{I',J'}^{\alpha\beta}\nabla_\beta 
				+ \mathcal{S}_{I',J'}^\alpha\nabla_\alpha\big]\td{\Gamma}^{J'}\Gamma^{I'}
				+w_{I',J'}\td{\Gamma}^{J'}\Gamma^{I'}
				\Box_g \Big)  \ , \label{mult_comm_iden2}
		\end{equation}
		where $\mathcal{R}_{I',J'}$, $\mathcal{S}_{I',J'}$, and $w_{I',J'}$ satisfy respectively
		\eqref{R_sym_bnds}, \eqref{S_sym_bnds}, and \eqref{w_sym_bnds'}.
\end{enumerate}
\end{cor}

\begin{proof}
We'll show the different parts separately.

\Part{1}{Proof of \eqref{mult_comm_iden1} for $\mathbb{L}$ and $\mathbb{L}_0$}
Here we will focus only on the case of products of vector fields in $\mathbb{L}$,
as the case of   $\mathbb{L}_0$ is completely analogous. 
Using the algebra property of Part \ref{Q_alg_prop} ) of Lemma \ref{abs_comm_prop} and an induction,
it suffices to  show:
\begin{equation}
		[\Box_g,\Gamma^I] \ = \ \sum_{I'\subsetneq I} \big[ \nabla_\alpha
		\mathcal{R}_{I'}^{\alpha\beta}\nabla_\beta + \mathcal{S}_{I'}^\alpha\nabla_\alpha
		+w_{I'}\Box_g \big] \Gamma^{I'} \ , \qquad \hbox{where\ \ }
		\Gamma^0 \ = \ Id \ . \label{mult_comm_iden1'}
\end{equation}
We shall  prove this last bound itself by induction
on the length $|I|$ of the product $\Gamma^I$.

\case{1}{$|I|=1$}
When $\Gamma^I$ consists of a single vector field in $\mathbb{L}$,
formula \eqref{mult_comm_iden1'} follows from a combination  of 
formulas \eqref{L_comm_form} and \eqref{S_comm_form}, followed by estimates \eqref{det_bnds},
\eqref{R_sym_bnds}, and part \ref{main_LR_estimate}) of Lemma \ref{basic_lie_lem}.

\case{2}{$|I|\geq 2$}
Assume formula \eqref{mult_comm_iden1'} holds for all multiindices $|I|<k$ and 
choose some $|I|=k$, and write $\Gamma^I=\Gamma^{I_0}\Gamma^{I_1}$ for some $|I_0|=1$. 
By the Leibniz rule  we have 
$[\Box_g,\Gamma^I]= [\Box_g, \Gamma^{I_0}]\Gamma^{I_1}+ \Gamma^{I_0}[\Box_g,\Gamma^{I_1}]$.
By the same calculations as in the previous step we have:
\begin{equation}
		[\Box_g, \Gamma^{I_0}]\Gamma^{I_1} \ = \ \nabla_\alpha\mathcal{R}^{\alpha\beta}
		\nabla_\beta \Gamma^{I_1} + w\Box_g \Gamma^{I_1} \ , \notag
\end{equation}
where $\mathcal{R},w$ are of the desired form. On the other hand by induction we have:
\begin{equation}
		\Gamma^{I_0}[\Box_g,\Gamma^{I_1}] \ = \ \sum_{I'\subsetneq I_1} \big( \nabla_\alpha
		\mathcal{R}_{I'}^{\alpha\beta}\nabla_\beta + \mathcal{S}_{I'}^\alpha\nabla_\alpha
		+w_{I'}\Box_g \big) \Gamma^{I_0}\Gamma^{I'} + 
		 \sum_{I'\subsetneq I_1} \big[ \Gamma^{I_0}, \big( \nabla_\alpha
		\mathcal{R}_{I'}^{\alpha\beta}\nabla_\beta + \mathcal{S}_{I'}^\alpha\nabla_\alpha
		+w_{I'}\Box_g \big) \big]\Gamma^{I'} \ . \notag
\end{equation}
By formula \eqref{R_comm_form} and
Part \ref{Q_X_part}) of Lemma \ref{abs_comm_prop} the commutator 
$\big[ \Gamma^{I_0}, \big( \nabla_\alpha
\mathcal{R}_{I'}^{\alpha\beta}\nabla_\beta + \mathcal{S}_{I'}^\alpha\nabla_\alpha\big) \big]$
again yields an operator of the form $\mathcal{Q}_{I'}=
\nabla_\alpha \td{\mathcal{R}}_{I'}^{\alpha\beta}\nabla_\beta 
+ \td{\mathcal{S}}_{I'}^\alpha\nabla_\alpha$. Using \eqref{det_bnds},  the same
calculations of the previous step, and part \ref{Q_alg_prop}) of Lemma \ref{abs_comm_prop}, 
we see the commutator $[\Gamma^{I_0},w_{I'}\Box_g]$
yields  another such operator. Combining all this 
yields \eqref{mult_comm_iden1'}.

\Part{2}{Proof of \eqref{mult_comm_iden2}}
Applying $\td{\Gamma}^J$ to formula \eqref{mult_comm_iden1} and then computing 
  $\big[ \td{\Gamma}^{J}, \big( \nabla_\alpha
\mathcal{R}_{I'}^{\alpha\beta}\nabla_\beta + \mathcal{S}_{I'}^\alpha\nabla_\alpha\big) \big]$
through a repeated use of Part \ref{Q_X_part'}) of Lemma \ref{abs_comm_prop}
yields the desired result.
\end{proof}


\subsection{Klainerman-Sideris Inequalities}

We now prove an analog of the  Klainerman-Sideris identity (see \cite{KS}).

\begin{lem}[Klainerman-Sideris type identity]\label{KS_lem}
One has the following  pointwise estimates: 
\begin{align}
		\sum_{1\leq l+|J|= 2} | (\tau_x \tau_0 {\partial_u^b})^l(\tau_x\partial_x^b)^J \phi|
		\ &\lesssim \ \sum_{\substack{ l+|J|=1\\ |I|\leq 1}}
		| (\tau_x \tau_0{\partial_u^b})^l(\tau_x\partial_x^b)^J \Gamma^I \phi| + \tau_x^2\tau_0 |\Box_g\phi |
		\ ,  \label{KS_iden}\\
		\sum_{1\leq l+|J|\leq k} | (\tau_-{\partial_u^b})^l(\tau_x\partial_x^b)^J \phi|
		\ &\lesssim \ \sum_{\substack{ l+|J|=1\\ |I|\leq k-1}}
		| (\tau_-{\partial_u^b})^l(\tau_x\partial_x^b)^J \Gamma^I \phi| +  \!\!\!\!
		\sum_{l+|J|\leq k-2}\tau_x^2 \tau_0|
		(\tau_-{\partial_u^b})^l(\tau_x\partial_x^b)^J\Box_g\phi |
		\ ,  \label{KS_iden_higher}
\end{align}
where in the second bound the implicit constant depends on $k\geq 2$. Here all $\Gamma\in 
\mathbb{L} $.
\end{lem}

\begin{proof}[Proof of estimates \eqref{KS_iden} and \eqref{KS_iden_higher}]
Both estimates follow from essentially the same  computation. 

\step{1}{A preliminary reduction}
First, note that in 
either case it suffices to restrict to  $r>R\gg 1$, as  estimates
\eqref{KS_iden} and \eqref{KS_iden_higher}
 with some implicit constant $C=C(R,k)$ is automatic in $r\leq R$ 
by choosing all $\Gamma\in \mathbb{L}_0\cup \{ S\}$.

Next, let $\Box_\eta$ be the Minkowski wave operator in Bondi coordinates which satisfies: 
\begin{equation}
		\Box_\eta  =   -2{\partial_u^b}\partial_r 
		+ (\partial_r^b)^2 - 2r^{-1}{\partial_u^b} + 2r^{-1}\partial_r^b 
		+ r^{-2}\sum_{i<j}(\Omega_{ij})^2 \ , \quad
		\Box_\eta \ = \ d^\frac{1}{2} \Box_g -
		  \partial_\alpha^b 
		\mathcal{R}^{\alpha\beta}\partial_\beta^b + O( r^{-2})\partial_x^b{\partial_u^b} + O(r^{-3}){\partial_u^b} \ , \notag
\end{equation}
where $d=|g|$ is the metric determinant in Bondi coordinates, and $\mathcal{R}=d^\frac{1}{2}g-h$
where $h$ is given on line \eqref{h_tensor} and $\mathcal{R}$ satisfies \eqref{R_sym_bnds}.
Thanks to  \eqref{Q_bondi_est}
one has:
\begin{equation}
		\tau_x^2\tau_0 |\Box_\eta \phi| \ \lesssim \ \tau_x^2\tau_0 |\Box_g \phi| 
		+  o_{R}(1)\cdot 
		\sum_{1\leq l+|J|\leq 2} | (\tau_x\tau_0{\partial_u^b})^l(\tau_x\partial_x^b)^J \phi|
		\ , \qquad \hbox{in\ \ } r>R \ . \notag
\end{equation}
Therefore it suffices to replace $\Box_g$ by $\Box_\eta$ in \eqref{KS_iden},
and also in \eqref{KS_iden_higher}  when proving it for $k=2$.

\step{2}{Proof of \eqref{KS_iden} and \eqref{KS_iden_higher} for $k=2$ and
$\Box_g$ replaced by $\Box_\eta$} Start with the two identities:
\begin{align}
		r^2\tau_+^{-1} \partial_r^b S \ &= \ r^2 u \tau_+^{-1}{\partial_u^b}\partial_r^b + r^3\tau_+^{-1}
		(\partial_r^b )^2 +r^2\tau_+^{-1}\partial_r^b \ , \notag\\
		\frac{1}{2} r^2u\tau_+^{-1}\Box_\eta \ &= \ -r^2u\tau_+^{-1}{\partial_u^b}\partial_r^b
		+ \frac{1}{2} r^2u\tau_+^{-1}(\partial_r^b)^2 -ru\tau_+^{-1}{\partial_u^b} + ru\tau_+^{-1}\partial_r^b
		+ \frac{1}{2}u\tau_+^{-1} \sum_{i<j}(\Omega_{ij})^2
		\ . \notag
\end{align}
Adding the two operators on the LHS above applied to $\phi$ yields:
\begin{equation}
		|(\tau_ x \partial_r^b)^2\phi| \ \lesssim \ \sum_{ l+|J|=1, \ |I|\leq 1}
		| (\tau_x\tau_0{\partial_u^b})^l(\tau_x\partial_x^b)^J \Gamma^I \phi| 
		+  \tau_x^2 \tau_0|\Box_\eta \phi | \ . \notag
\end{equation}
Note that by Remark \ref{coord_rem} we can assume $u+2r\approx \tau_+$ in $r>R$.

Next,   the   vector fields $S$ and ${\partial_u^b}$ alone give the pair of inequalities:
\begin{align}
		|(\tau_x\tau_0 {\partial_u^b})^2\phi| + |(\tau_x\tau_0 {\partial_u^b})(\tau_x\partial_r^b)\phi| \ &\lesssim \
		|(\tau_x \partial_r^b)^2\phi| +
		\sum_{ l+|J| = 1, \ |I|\leq 1 }
		| (\tau_x\tau_0{\partial_u^b})^l(\tau_x\partial_x^b)^J \Gamma^I\phi  |
		 \ , \notag\\
		|(\tau_-{\partial_u^b})^2\phi| + |(\tau_-{\partial_u^b})(\tau_x\partial_r^b)\phi| \ &\lesssim \
		|(\tau_x\partial_r^b)^2\phi| +
		\sum_{ l+|J|=1, \ |I|\leq 1 }
		| (\tau_-{\partial_u^b})^l(\tau_x\partial_x^b)^J \Gamma^I\phi  | 
		\ . \notag
\end{align}

Finally, note that all other combinations of derivatives on LHS \eqref{KS_iden} (resp
\eqref{KS_iden_higher} when $k=2$)
are  automatically controlled
by the first sums on RHS \eqref{KS_iden} (resp \eqref{KS_iden_higher} when $k=2$) thanks to the rotation
vector fields.

\step{3}{Estimate \eqref{KS_iden_higher} when  $k>2$}
It suffices to show \eqref{KS_iden_higher}  assuming its true for $k-1$.
Applying \eqref{KS_iden_higher} with $k-1$ to $X\phi$, where 
$X \in\{\tau_- {\partial_u^b}, \tau_x\partial_i^b\}$, and then feeding the results back
into \eqref{KS_iden_higher}  with $k=2$ applied to $\Gamma^I\phi$ where 
$\Gamma\in\mathbb{L}$ and $|I|\leq k-2$,
we need to show the  commutator estimates:
\begin{equation}
		\sum_{\substack{ |I'|,|I''|=1\\ |I| \leq k-2}} \!\!\!\!\!
		| X^{I'} [\Gamma^I,  X^{I''}] \phi|  \lesssim 
		\!\!\!\!\!\!\!\!
		\sum_{1\leq |I| \leq k-1} 
		\!\!\!\!\! \!\!\! | X^I \phi| \ , 
		\sum_{ |I|\leq k-2} \!\!\!\!\! \tau_x^2 \tau_0|
		[\Box_g,\Gamma^{I}]\phi |  \lesssim \!\!\!\!\!\!\!\!
		\sum_{1\leq |I| \leq k-1} \!\!\!\!\! \!\!\! | X^I \phi| \ , 
		\sum_{\substack{ |I| \leq k-3\\ |I'|=1 }}\!\!\!\!\! \tau_x^2 \tau_0|
		X^I [\Box_g,X^{I'}]\phi |  \lesssim \!\!\!\!\!\!\!\!
		\sum_{1\leq |I| \leq k-1} 
		\!\!\!\!\! \!\!\! | X^I \phi| \ , \notag
\end{equation}
where each $X^I$, $X^{I'}$, and $X^{I''}$ is a product of members of $\{\tau_- {\partial_u^b}, \tau_x\partial_i^b\}$.
The validity of these last three bounds is easily checked 
by using \eqref{mult_comm_iden1}  and \eqref{Q_bondi_est}
to evaluate $[\Box_g,\Gamma^{I}]\phi$, and
by referring to the following lemma.
\end{proof}

\begin{lem}[Products of ``standard'' vector fields]
Let $X^I$ and $Y^J$ products of  vector fields whose Bondi coordinate 
coefficients satisfy:
\begin{equation}
		\big| (\tau_-{\partial_u^b})^l (\tau_x \partial_x^b)^J X^u\big|
		\ \lesssim_{l,J} \   \tau_-  
		\ , \qquad\qquad
		\big| (\tau_-{\partial_u^b})^l (\tau_x \partial_x^b)^J X^i\big|
		\ \lesssim_{l,J} \ \tau_x  \ . \notag
\end{equation}
with the convention a product of length zero  
is a scalar satisfying bounds of the form \eqref{w_sym_bnds} with $a=b=c=0$.
Then the following hold:
\begin{enumerate}[I)]
	\item The vector field $[X^I,Y^J]$ is a sum of products of similar vector fields
		each with word length $|I|+|J|-1$.
	\item For any nonzero multiindex $I$ we have:
		\begin{equation}
				|X^I\phi| \ \lesssim \ \sum_{1\leq l+|J|\leq |I|} |(\tau_-{\partial_u^b})^l (\tau_x\partial_x^b)^J\phi|
				\ . \notag
		\end{equation}
	\item If $X^I$ is any product of such vector fields then:
		\begin{equation}
			 \tau_x^2\tau_0 X^I\tau_x^{-2}\tau_0^{-1} - X^I \ = \ \sum_{|J|\leq |I|-1} Y^J \ . \notag
		\end{equation}
		for some other collection of products of similar vector fields $Y^I$.
	\item One has $\tau_x^2\tau_0\Box_g=\sum_{1\leq |I|\leq 2} X^I$ for some collection of such 
		vector fields $X^I$.
\end{enumerate}
\end{lem}

\begin{proof}
The proof of the first three parts boils down to more or less elementary calculations.
The last part follows from a direct calculation involving the conditions \eqref{mod_coords} and \eqref{det_bnds}
is also left to the reader.
\end{proof}

\subsection{A pointwise bound for multi-commutators}

To conclude this Section, we record  a combined consequence of Lemma \ref{abs_comm_prop}, Corollary
\ref{multicom_coro}, and Lemma \ref{KS_lem}. This will be our main tool for controlling commutators in the sequel.

\begin{lem}[Pointwise bound for commutators]
For pair of  multiindices $I,J$ with $|I|\geq 1$ one has:
\begin{equation}
		\big|\td{\Gamma}^J [\Box_g,\Gamma^I]\phi\big|  \ \lesssim \
		 \sum_{\substack{ I'\subsetneq I\\ |J'| \leq |J|}}
		 \big(\sum_{\substack{ l+|K|=1\\ |I''|\leq 1}}
		q
		\cdot \tau_x^{-1}  | (\tau_0{\partial_u^b})^l( \partial_x^b)^K  
		\td{\Gamma}^{J'} \Gamma^{I'+I''} \phi| 
		+   |\td{\Gamma}^{J'}  \Gamma^{I'} \Box_g\phi | \big) \ , 
		\quad \hbox{where \ \ } q\in \mathcal{Z}^{-\frac{1}{2}} \ , 
		\label{main_ptws_comm}
\end{equation}
where $\Gamma^I$ (etc)
denotes a product of vector fields in 
$ \mathbb{L}=\{S,\Omega_{ij},{\partial_u^b} , \partial_i^b-\omega^i{\partial_u^b}\}$, and
$\td{\Gamma}^{J}$ (etc) denotes a product of vector fields in 
$ \mathbb{L}_0=\{ {\partial_u^b} , \partial_i^b-\omega^i{\partial_u^b}\}$.
\end{lem}

\begin{proof}
Using \eqref{mult_comm_iden2} followed by \eqref{Q_bondi_est} and then \eqref{KS_iden} we have:
\begin{equation}
		LHS\eqref{main_ptws_comm} \ \lesssim \ 
		RHS\eqref{main_ptws_comm} + \sum_{\substack{ I'\subsetneq I\\ J'\subseteq J}}
		|\Box_g(\td{\Gamma}^{J'}  \Gamma^{I'}\phi)| \ . \notag
\end{equation}
Thus, modulo induction on $|I|$ we have reduced matters to estimating the commutator
$|[ \Box_g,\td{\Gamma}^{J'}] ( \Gamma^{I'}\phi)|$. Applying Part \ref{L0_multicom_part})
of Corollary \ref{multicom_coro}, again  followed by \eqref{Q_bondi_est} and then \eqref{KS_iden}, and inducting
on $|I|$  we have:
\begin{equation}
		|[ \Box_g,\td{\Gamma}^{J'}] ( \Gamma^{I'}\phi)| \ \lesssim \
		RHS\eqref{main_ptws_comm} + \sum_{ J''\subsetneq J'}
		|\Box_g(\td{\Gamma}^{J''}  \Gamma^{I'}\phi)| \ , \notag
\end{equation}
so the proof concludes with an additional round of induction, this time with respect to $|J|$.
\end{proof}
  

\section{Proof of the Weighted  $L^2$ Estimates for $k=0$}\label{L2_est_sect}

\begin{thm}[Generalized Local Energy Decay Estimates]\label{NLE_thm}
Assume estimates \eqref{basic_LE} and \eqref{basic_stat_LE} for $s=0$, then the following are true:
\begin{enumerate}[I)]
\item For $R$ sufficiently 
large there exists  $C_R>0$ and vector fields $X_j$
such that one has the uniform bound:
\begin{equation}
		\slp{ \phi }{\WLE^0 [0,T]} 
		  \lesssim   \sup_{0\leq t\leq T}\slp{ \partial  \phi(t) }{L^2_{x}} 
		  + C_R\slp{\phi}{\WLE^0_{class}[0,T]} 
		 + \sup_j \Big|\int_0^T\!\!\!\!\!\int_{\mathbb{R}^3\setminus \mathcal{K}} 
		 \Box_g\phi \cdot  X_j \phi dV_g \Big|^\frac{1}{2}
		 \ , \label{LE_null_est}
\end{equation}
where $X_j=\chi_{>R}q_j {\partial_u^b}$,
with $q_j=q_j(u)$ obeying the uniform bounds $| (\tau_-{\partial_u^b})^k q_j|
\lesssim_{k} 1$.
In addition $\chi_{>R}=\chi_{>1}(R^{-1}\cdot)$ 
is a  radial bump function with $\chi_{>R}\equiv 1$
on $r>R$ and $\chi_{>R}\equiv 0$
on $r<\frac{1}{2}R$ for some sufficiently large that $\mathcal{K}\subseteq \{r<\frac{1}{2}R\}$.

\item 
For each $ 0 < a < 1$ there exists an $R_a$ sufficiently large so that:
\begin{equation}
		\sup_{0\leq t\leq T}\lp{\phi(t)}{\CE^a}+
		\lp{\phi}{S^a[0,T]} 
		\ \lesssim_a \ \lp{\tau_x^a\partial\phi(0)}{L^2_x}
		+ \lp{\tau_+^{a-1} \phi }{ H^1_1(r<R_a)[0,T]} + 
		\lp{\Box_g \phi}{N^a [0,T]}
		\ , \label{NLE_est}
\end{equation}
where the  implicit constant depending continuously on $a\in (0,1)$.

\item
For $R$ sufficiently 
large there exists  $C_R>0$ and vector fields $X_j$
such that  one has uniformly:
\begin{multline}
		\sup_{0\leq t\leq T}\lp{\phi(t)}{\CE^1}
		+ \lp{ \phi }{ S^{1,\infty}[0,T]}
		 \ \lesssim  \  
		\lp{\tau_x \partial\phi(0)}{L^2} +
		C_R\lp{ \phi }{ \ell^1_t H^1_1(r<R)[0,T]} \\
		   \ +\ \lp{ \Box_g\phi}{\ell^\infty_t N^1[0,T]}  
		 \ +\  \sup_j \Big|\dint  
		 \Box_g\phi \cdot \tau_x^{-1} X_j(\tau_x \phi) dV_g \Big|^\frac{1}{2}
		\ ,  \label{CE_est} 
\end{multline}
where 
 $X_j=\chi_{>{R_1}}q_j K_0$,
with $K_0$  given by the formula  $K_0=(1+u^2) {\partial_u^b} + 2(u+r)r\partial_r^b $,   where
 $q_j=q_j(u)$ has the uniform bounds $| (\tau_-{\partial_u^b})^k q_j|
\lesssim_{k} 1$,
and where $\chi_{>{R}}$ is the same as in \eqref{LE_null_est} above.
\end{enumerate}
\end{thm}


\subsection{Splitting Into Interior and Exterior Estimates}

To control the solution in the interior we use:

\begin{prop}[Weighted $\LE$ bounds in time-like regions]\label{timelike_LE_prop}
Let $R_0\geq 1$ be as in Definition \ref{class_LE_defn}.
Then for any $a\geq 0$ and $1\leq p\leq \infty$ one has the uniform bound:
\begin{multline}
		\sup_{0\leq t\leq T}\lp{\tau_+^{a}  (\partial\phi,\tau_x^{-1}\phi) (t)}{L^2_x(r<\frac{1}{2}t)}+
		\lp{
		\tau_+^a  (\partial\phi,\tau_x^{-1}\phi)}{\ell^p_t \LE(r<\frac{1}{2}t)[0,T]} \   \lesssim \\ 
		\lp{\tau_x^{a}\tau_+^{-\frac{1}{2}}   (\partial\phi, \tau_+^{-1}\phi) }{\ell^p_t L^2( r<\frac{3}{4}t)[0,T]}
		  + \lp{ \tau_+^{a-1} \phi }{ \ell^p_t  H^1_1(r<R_0)[0,T]}
		  +  \lp{\tau_+^a \tau_0 \Box_g\phi}{\ell^p_t  \WLE^{*,0}[0,T]}  \ . \label{timelike_LE}
\end{multline}
Here the $\ell^p_t$ sum is taken over a collection of dyadic regions 
$\la t\ra \approx 2^j\geq 1$.
\end{prop}

In the exterior we use multipliers to show that:

\begin{prop}[Weighted exterior $\LE$ bounds ]\label{main_ext_prop}
One has the following estimates uniformly for $R$ sufficiently large that $\mathcal{K}\subseteq \{r\leq \frac{1}{2}R\}$:
\begin{enumerate}[I)]
\item
The null energy bounds:
\begin{equation}
		\lp{(\partial_x^b\phi ,\tau_x^{-1}\phi)}{\NLE(r>R)[0,T]} \ \lesssim \ R^\frac{1}{2}\lp{\phi}{\WLE^0_{class}[0,T]}
		+\sqrt{\mathcal{E}(0,R)} \ , \label{0_ext_est}
\end{equation}
where $\mathcal{E}(0,R)=\sup_j\mathcal{E}(0,R,X_j)$ is given by formula \eqref{0R_error}
with $X_j$ as in Theorem \ref{NLE_thm}.

\item
For fixed $0<a<1$ there holds:
\begin{multline}
		\sup_{0\leq t\leq T}\lp{\phi(t)}{\CE^a(r>\max\{R,\frac{1}{2}t\})}+
		\lp{\phi}{S^a(r>\max\{R,\frac{1}{2}t\})[0,T]} + \lp{\tau_x^{a}\tau_+^{-\frac{1}{2}}  
		(\partial\phi, \tau_x^{-1}\phi) }{ L^2(R<r<\frac{3}{4}t)[0,T]}\\
		 \lesssim_a  \sqrt{\mathcal{E}(a,R)} \ , \label{a_ext_est} 
\end{multline}
where $\mathcal{E}(a,R)$ is given by formula \eqref{aR_error}. 

\item
Corresponding
to $a=1$ there holds the estimate:
\begin{multline}
		\sup_{0\leq t\leq T}\lp{\phi(t)}{\CE^1(r>\max\{R,\frac{1}{2}t\})}+
		\lp{\phi}{  S^{1,\infty}(r>\max\{R,\frac{1}{2}t\})[0,T]} +   
		\lp{\tau_x \tau_+^{- \frac{1}{2}}   (\partial\phi, \tau_+^{-1}\phi) }{\ell^\infty_t L^2( R<r<\frac{3}{4}t)[0,T]}\\
		 \lesssim \
		 \sup_{0\leq t\leq T}
		 R^\frac{1}{2}\lp{\tau_-^\frac{1}{2} (\partial\phi,\tau_x^{-1}\phi)(t)}{L^2_x(\frac{1}{2}R<r<R)}+
		  \sqrt{\mathcal{E}(1,R)} \ , \label{1_ext_est} 
\end{multline}
where $\mathcal{E}(1,R)=\sup_j\mathcal{E}(1,R,X_j)$ is given in terms of formula \eqref{1R_error}
with $X_j$ as in Theorem \ref{NLE_thm}.

\end{enumerate}
\end{prop}

\begin{proof}[Proof that Proposition \ref{timelike_LE_prop} and Proposition \ref{main_ext_prop} 
imply Theorem \ref{NLE_thm}]
We do this separately for each estimate.

\case{1}{$a=0$} 
Here we need to show \eqref{LE_null_est} follows directly from \eqref{0_ext_est} and the assumed bounds 
\eqref{class_LE_assm}. From inspection of $\sqrt{\mathcal{E}(0,R)}$ and taking $R$ sufficiently large,
we only need to bound $\slp{(\partial^b_x\phi,\tau_x^{-1}\phi) }{\NLE(R_0<r<R)[0,T]}$ in terms
of $C_R$ times $\lp{\phi}{\WLE^0_{class}[0,T]}$. This follows by taking $C_R\approx R^\frac{1}{2}$.

\case{2}{ $0<a<1$ } Adding together a suitable linear combination of 
estimates \eqref{timelike_LE} with $p=2$ and  \eqref{a_ext_est},
and using the inclusions $\ell^2_t\LE\subseteq \LE$ and $\LE^*\subseteq \ell^2_t\LE^*$
(from Minkowski's inequality), we have uniformly:
\begin{multline}
		\sup_{0\leq t\leq T}\lp{\phi(t)}{\CE^a}+
		\lp{\phi}{S^a[0,T]} 
		\ \lesssim_a  \sup_{0\leq t\leq 2R }\lp{\phi(t)}{\CE^a(r<R)}+
		 \lp{\phi}{S^a(r<R)[0,2R]}\\ 
		 + \lp{\tau_x^{a}\tau_+^{-\frac{1}{2}}   (\partial\phi, \tau_x^{-1}\phi) }
		 {  L^2( r<\min\{R,\frac{3}{4}t\})[0,T]}
		  + \lp{ \tau_+^{a-1} \phi }{ H^1_1(r<R_0)[0,T]}
		  +  \lp{ \Box_g\phi}{N^a[0,T]}
		 + \sqrt{\mathcal{E}(a,R)} \ , \notag
\end{multline}
Next,   uniformly for  $T_0\geq 2R\geq 1$ there holds the pair of bounds:
\begin{equation}
		\lp{\tau_x^{a}\tau_+^{-\frac{1}{2}}   (\partial\phi, \tau_x^{-1}\phi) }{  L^2( r<\min\{R,\frac{3}{4}t\})[0,T]}
		\ \lesssim \  \ln(R) \lp{\phi}{S^a(r<R)[0,T_0]} + (R/T_0)^\frac{1}{2}\lp{\phi}{S^a[T_0,T]} \ , \notag
\end{equation}
and:
\begin{multline}
		\mathcal{E}(a,R) \ \lesssim \ \lp{\tau_x^a\partial\phi(0)}{L^2_x}^2 + ( o_{R}(1)+ (R/T_0)^{a}) \cdot \big(
		\sup_{0\leq t\leq T}\lp{\phi(t)}{E^a}^2 + \lp{\phi}{S^a[0,T]}^2
		\big)\\ 
		+ \lp{\phi}{S^a[0,T]} \cdot  \big(
		\lp{\phi}{S^a(r<R)[0,T_0]}
		+R^\frac{3}{2} \lp{\tau_+^{a-1} \phi }{H^0_1(r<R)[0,T]}
		+\lp{\Box_g\phi}{N^a[0,T]} \big) \ . \notag
\end{multline}
In addition for $T_0\geq 2R$ 
there is the simple bound:
\begin{equation}
		\sup_{0\leq t\leq 2R }\lp{\phi(t)}{\CE^a(r<R)}+
		\ln(R) \lp{\phi}{S^a(r<R)[0,T_0]} \ \lesssim_{R,T_0} \
		\lp{ \tau_+^{a-1} \phi }{ H^1_1(r<R)[0,T_0]} \ . \notag
\end{equation}
Therefore combining the last four inequalities   
with $R=R_a\geq R_0$  sufficiently large depending on $a$, and 
for some $T_0=T_a\geq R_a$ which depends on both the size of $R_a$ and $a$, we have
estimte \eqref{NLE_est}.


\case{3}{$a=1$}
Here we apply \eqref{timelike_LE} with $a=1$ and $p=\infty$, and add to this a suitable linear combination of
\eqref{1_ext_est}.
Note that an application of
the weighted trace estimate \eqref{w_trace3} followed by \eqref{hardy0} with $a=1$ gives:
\begin{equation}
		\sup_{0\leq t\leq T}
		 R^\frac{1}{2}\lp{\tau_-^\frac{1}{2} (\partial\phi,\tau_x^{-1}\phi)(t)}{L^2_x(\frac{1}{2}R<r<R)}
		 \ \lesssim \ \lp{\tau_x \partial\phi(0)}{L^2} + R^\frac{3}{2}\lp{ \phi }{ \ell^1_t H^1_1(r<R)[0,T]} 
		 \ . \notag
\end{equation}
The rest of the proof follows a similar pattern to \textbf{Case 2} above. \end{proof}


\subsection{Proof of the interior estimate}

Before moving on to the proof of the exterior estimates which are more involved, we 
first demonstrate Proposition \ref{timelike_LE_prop}.

\begin{proof}[Proof of estimate \eqref{timelike_LE}]
Without loss of generality we work with the time interval $[1,T]$.
We apply estimate \eqref{concat_LE} to $2^{ak}\chi_0(2^{-k}t)\chi_{<1}(r/t)\phi$ for $k\geq 0$,
where $\chi_0(s)$ is a smooth bump function adapted to $1\leq s\leq 2$, and
$\chi_{<1}(s)$ is a smooth function $=1$ for $s\leq \frac{1}{2}$ and $=0$ for $s>\frac{3}{4}$.
Using the Hardy estimate \eqref{hardy0} this yields:
\begin{align}
		\sup_{ 2^k\leq t\leq  2^{k+1} }\!\!\!\! \slp{\tau_+^a(\partial\phi,\tau_x^{-1}\phi)(t)}{L^2_x(r<\frac{1}{2}t)}
		+ \slp{\tau_+^a(\partial\phi,\tau_x^{-1}\phi)}{\LE(r<\frac{1}{2}t)[2^k,2^{k+1}] }  \lesssim 
		\slp{\tau_+^a ({\partial_u^b}\phi ,\tau_+^{-1}\phi )}{H^1(r<R_0)[2^{k-1}, 2^{k+2}]}\notag \\ +
		\slp{\tau_+^{a-1}(\partial\phi,\tau_x^{-1}\phi)}{\WLE^{*,0}(r<\frac{3}{4}t)[ 2^{k-1} ,  2^{k+2}]} 
		+ \slp{\tau_+^a\tau_0\Box_g\phi}{\WLE^{*,0}[ 2^{k-1} , 2^{k+2}]} 
		\ . \notag
\end{align}
For a fixed value of $k$, dyadic summation in $r$ gives the
uniform estimate:
\begin{equation}
		\slp{\tau_+^{a-1}(\partial\phi,\tau_x^{-1}\phi)}{\WLE^{*,0}(r<\frac{3}{4}t)[ 2^{k-1} , 2^{k+2}]} 
		  \lesssim 
		 \slp{\tau_+^{a-\frac{1}{2}}(\partial\phi,\tau_x^{-1}\phi)}{L^2(r<\frac{3}{4}t)[ 2^{k-1} ,  2^{k+2}]} 
		 + \slp{\tau_+^{a-1}\phi }{ H^2(r<R_0)[2^{k-1} ,  2^{k+2}]}
		  . \notag
\end{equation}
By splitting   into into regions $r<\gamma t$ and $r>\gamma t$
we have the following uniform estimate for $0<\gamma<\frac{1}{2}$:
\begin{multline}
		\lp{\tau_+^{a-\frac{1}{2}} (\partial\phi,\tau_x^{-1}\phi)}{L^2(r<\frac{3}{4}t)[ 2^{k-1} ,  2^{k+2}]}
		\ \lesssim  \ \gamma^\frac{1}{2} 
		\lp{\tau_+^{a} (\partial\phi,\tau_x^{-1}\phi)}{\LE(r<\frac{1}{2}t)[ 2^{k-1} ,  2^{k+2}]}\\
		+ \gamma^{- a-1 }
		\lp{\tau_x^{a}\tau_+^{-\frac{1}{2}} (\partial\phi,\tau_+^{-1}\phi) }{L^2(r<\frac{3}{4}t)[ 2^{k-1} ,  2^{k+2}]} \ . \notag
\end{multline}
On the one hand with the help of the scaling vector field we have:
\begin{equation}
		 \lp{\tau_+^a ({\partial_u^b}\phi ,\tau_+^{-1}\phi )}{H^1(r<R_0)[ 2^{k-1},   2^{k+2}]}
		\ \lesssim \ \lp{\tau_+^{a-1}\phi  }{H^1_1(r<R_0)[ 2^{k-1},  2^{k+2}]} \ . \notag
\end{equation}
Thus \eqref{timelike_LE} follows by summing the last four lines in $\ell^p_t$ and taking $\gamma \ll 1$.
\end{proof}


\subsection{Proof of the Exterior Estimates}

At the level of multipliers there are essentially four cases here: $a=0$, $0<a\leq \frac{1}{2}$, 
$\frac{1}{2}\leq a<1$, and $a=1$. Collectively these are stated in the
following lemma.

\begin{lem}[Core multiplier bounds]\label{core_mult_lem}
One has the  following collection of estimates uniformly for $R$ large
enough that $\mathcal{K}\subseteq \{r\leq \frac{1}{2}R\}$.
\begin{enumerate}[I)]
\item 
Corresponding to the case $a=0$ one has:
\begin{equation}
		  \lp{  \tau_-^{-\frac{1}{2}} \tau_x^{-1}\partial_x^b(\tau_x\phi)
		}{\ell^\infty_u L^2(r>R)[0,T]}^2 \ \lesssim \  \mathcal{E}(0,R) \ . \label{0_mult_est}
\end{equation}

\item In the range $0 <   a\leq \frac{1}{2}$ one has:
\begin{multline}
		\sup_{0\leq t\leq T}\slp{\tau_+^{a} (\tau_0^\frac{1}{2}
		{\partial_u^b}\phi, \tau_x^{-1}\partial_x^b (\tau_x \phi) ) (t)}{L^2_x(r>R)}^2
		+ \slp{\tau_x ^{a-\frac{3}{2}}\partial_x^b(\tau_x\phi) \big)}{L^2(r>R)[0,T]}^2
		+    \slp{\tau_+^{a-\frac{1}{2}}\tau_0^2  {\partial_u^b}\phi}{L^2(r>R)[0,T]}^2 \\
		+ \slp{\tau_-^{-\frac{1}{2}} \tau_x^{a-1} \partial_r^b(\tau_x \phi)}{\ell^\infty_u L^2(r>R)[0,T]}^2
		+  \slp{ \tau_+^{a-\frac{1}{2}}
		\tau_0^\frac{1}{2} {\partial_u^b}  \phi }{\ell^\infty_{\bu} L^2(r>R)[0,T]}^2
		\ \lesssim a^{-1}
		\mathcal{E}(a,R)\  . \label{0-1/2_mult_est}
\end{multline}
Here we have set $\bu=u+2r$.

\item 
When $\frac{1}{2}\leq a<1$ one has:
\begin{multline}
		\sup_{0\leq t\leq T}\!\! \slp{\tau_+^{a} (\tau_0^a 
		{\partial_u^b}\phi, \tau_x^{-1}\partial_x^b (\tau_x \phi) ) (t)}{L^2_x(r>R)}^2 
		\!+\! \slp{ \tau_x^{a-\frac{3}{2}}\partial_x^b(\tau_x \phi) \big)}{L^2(r>R)[0,T]}^2
		\!+\!    \slp{\tau_+^{a-\frac{1}{2}}\tau_0^2  {\partial_u^b}\phi}{L^2(r>R)[0,T]}^2\\
		+  \slp{\tau_-^{-\frac{1}{2}} \tau_x^{a-1} \partial_r^b(\tau_x \phi)}{\ell^\infty_u L^2(r>R)[0,T]}^2
		+ \slp{\tau_+^{-\frac{1}{2}} \tau_-^a {\partial_u^b}  \phi }{\ell^\infty_{\bu}L^2(r>R)[0,T]}^2
		\ \lesssim (1-a)^{-1} 
		\mathcal{E}(a,R)\  . \label{1/2-1_mult_est}
\end{multline}
Here again we have set $\bu=u+2r$.

\item 
Finally, corresponding to the case $a=1$ we have:
\begin{equation}
		\sup_{0\leq t\leq T}\lp{\tau_+ \big(\tau_0{\partial_u^b} \phi, \tau_x^{-1}
		 \partial_x^b(\tau_x \phi)\big)(t)}{ L^2_x(r>R)}^2
		+  \lp{ \tau_-^{-\frac{1}{2}} \tau_+ \tau_x^{-1} 
		\partial_r^b(\tau_x \phi)}{\ell^\infty_u L^2(r>R)[0,T]}^2
		\ \lesssim \   \mathcal{E}(1,R) \ . \label{K0_est}
\end{equation}
\end{enumerate}
\end{lem}

We'll prove each of these estimates separately. In each case we invoke Proposition \ref{abs_mult_prop}.

\begin{proof}[Proof of estimate \eqref{0_mult_est}]
Fix a $j\in\mathbb{N}$ and define  the multiplier:
\begin{equation}
		Y^u \ = \  \chi_{<j}(u) \ , \qquad
		Y^r \ = \ 0 \ , \notag
\end{equation}
were $\partial_s \chi_{<j}(s) = -2^{-j}\chi_j(s)$, with $\chi_j(s)\approx 1$ when $\la s\ra\approx 2^j$
and vanishing away from this region. It can also be arranged  that $\chi_{<j}(s)= 0$ for $s\gtrsim 2^j$
by addition of a suitable constant.
It is immediate these coefficients satisfy
 \eqref{mult_sym_bnds'}, so we may use Case \ref{0_mult_case}) of 
Proposition \ref{abs_mult_prop}. A short calculation then shows: 
\begin{equation}
		\mathcal{A}^u \ = \ \mathcal{A}^{ur} \ =\  0 \ , \qquad
		\mathcal{A}^r \ = \ \slash\!\!\!\! \mathcal{A} \ =\ 2^{-j-1}\chi_j(u) \ . \notag
\end{equation}
for the coefficients
on lines \eqref{A0_coeff}. Repeatedly applying estimate \eqref{abs_mult_est} for different $j\in\mathbb{N}$,
then taking $\sup_j$ of the result concludes the proof.
\end{proof}


\begin{proof}[Proof of estimate \eqref{0-1/2_mult_est}]
First freeze $j,k\in \mathbb{N}$ and define the multiplier: 
\begin{equation}
		Y^u \ = \ a( 2^{(2a-1)k}\tau_-\chi_{<k}(\bu) + \tau_+^{2a}\tau_0^4  )\ , \qquad
		Y^r \ = \ Cr^{2a}(1+ \chi_{<j}(u) ) \ , \qquad \bu = u+2r \ , \quad \tau_+ = 1+\bu \ , \notag
\end{equation}
where $\chi_{<k}$, $\chi_{<j}$ is as above and $C>0$ will be chosen shortly. It is easy to check that these coefficients 
satisfy the conditions on  \eqref{mult_sym_bnds} (in $r>R$). Computing the formulas from 
lines \eqref{A0_coeff} and choosing  
$C\gg1$ we have the following uniform estimates for $0<a\leq \frac{1}{2}$ in the region $r>1$: 
\begin{align}
		\mathcal{A}^u \ &\gtrsim \ a \big( \tau_+^{2a-1}\tau_0 \chi_k(\bu)+ \tau_+^{2a-1}\tau_0^4\big)
		\ , 
		&\mathcal{A}^r \ &\gtrsim \    \tau_-^{-1} \tau_x^{2a}\chi_j(u)+a \tau_x^{2a-1}   
		\ , \notag\\
		|\mathcal{A}^{ur}| \ &\ll \ \sqrt{\mathcal{A}^u \mathcal{A}^{r}} \ ,
		& \slash\!\!\!\! \mathcal{A} \ &\gtrsim \ \tau_x^{2a-1}
		\ . \notag
\end{align}
 Repeatedly applying estimate \eqref{abs_mult_est} for different $j,k\in\mathbb{N}$ and on time intervals $[0,t]\subseteq [0,T]$
 and then taking $\sup_{j,k,t}$   concludes the proof.
\end{proof}


\begin{proof}[Proof of estimate \eqref{1/2-1_mult_est}]
In this case we set:
\begin{equation}
		Y^u \ = \ (1-a)( \tau_-^{2a} \chi_{<k}(\bu) + \tau_+^{2a}\tau_0^4 ) \ , \qquad
		Y^r \ = \ C^2r^{2a}(1+ \chi_{<j}(u) ) + C(1-a)\tau_+^{2a-1}r \ , \notag
\end{equation}
These satisfy the conditions on line \eqref{mult_sym_bnds}, and
the formulas from lines \eqref{A0_coeff} yield when
$C\gg1$   the following estimates   uniformly in $\frac{1}{2} \leq a <1$: 
\begin{align}
		\mathcal{A}^u \ &\gtrsim \  (1-a)( \tau_+^{-1}\tau_-^{2a}\chi_k(\bu)+ \tau_+^{2a-1}\tau_0^4)
		\ , 
		&\mathcal{A}^r \ &\gtrsim \    \tau_-^{-1}\tau_x^{2a}\chi_j(u)+ (1-a)\tau_x^{2a-1}   
		\ , \notag\\
		|\mathcal{A}^{ur}| \ &\ll \ \sqrt{\mathcal{A}^u \mathcal{A}^{r}} \ ,
		&\slash\!\!\!\! \mathcal{A} \ &\gtrsim \ (1-a)\tau_x^{2a-1}
		\ . \notag
\end{align}
These suffice to give \eqref{1/2-1_mult_est} through an application of  \eqref{abs_mult_est}.
\end{proof}


\begin{proof}[Proof of estimate \eqref{K0_est}]
For this estimate we use the multiplier:
\begin{equation}
		Y^u \ = \ (1+\chi_{<j}(u))\tau_-^2  \ , \qquad
		Y^r \ = \  (1+\chi_{<j}(u))2(u+r)r \ . \notag
\end{equation}
This satisfies the conditions on line \eqref{mult_sym_bnds} with $a=1$. 
The  formulas from line \eqref{A0_coeff}   yield:
\begin{align}
		\mathcal{A}^u \ &= \  0
		\ , 
		&\mathcal{A}^r \ &\gtrsim \    \tau_-^{-1}\tau_+^2 \chi_j(u) 
		\ , \notag\\
		\mathcal{A}^{ur} \ &=\ 0 \ ,
		&\slash\!\!\!\! \mathcal{A} \ &\geq \ 0
		\ . \notag
\end{align}
This suffices to give \eqref{K0_est} through an application of  Case \ref{1_mult_case})
of Proposition \ref{abs_mult_prop}.
\end{proof}


It remains to close the gap between Proposition \ref{main_ext_prop} and Lemma \ref{core_mult_lem}.
 
\begin{proof}[Proof of  \eqref{0_ext_est}]
Applying the Hardy estimate \eqref{hardy3} with $a=\frac{1}{2}$ to the funciton
$2^{-\frac{1}{2}j} \chi_j(u) \phi$, followed by the Hardy estimate \eqref{hardy0} with $a=0$ for the term on
RHS\eqref{hardy3} at $t=0$, one has:
\begin{equation}
		 \lp{ \tau_-^{-\frac{1}{2}} \tau_x^{-1}\phi
		}{\ell^\infty_u L^2(r>R)[0,T]} \ \lesssim \ 
		 \lp{\tau_-^{-\frac{1}{2}} \tau_x^{-1}\partial_x^b(\tau_x \phi)
		}{\ell^\infty_u L^2(r> R)[0,T]}+ R^\frac{1}{2}\lp{\phi}{\WLE^0_{class}[0,T]}
		 +\lp{\partial \phi(0)}{L^2_x} \ , \label{0_mult_est'}
\end{equation}
where the extra factor of $R^\frac{1}{2}$ results from replacing $\tau_-^{-\frac{1}{2}}$
with $\tau_x^{-\frac{1}{2}}$ in the spacetime region $\frac{1}{2}R<r<R$. Combining this with estimate 
\eqref{0_mult_est} and using the definition \eqref{0R_error}
we have   \eqref{0_ext_est}.
\end{proof}

\begin{proof}[Proof of \eqref{a_ext_est}]
Estimate \eqref{hardy3} of the Appendix and the definition of $\mathcal{E}(a,R)$ from line 
\eqref{aR_error} imply:
\begin{equation}
		\lp{\tau_x^{a-1}\phi(T)}{L^2_x(r>R) }^2+
		\lp{\tau_x^{a-\frac{3}{2}}\phi }{L^2(r>R)[0,T]}^2 \
		\lesssim_a  \ 
		\lp{\tau_x^{a-\frac{3}{2}}\partial_r^b(\tau_x \phi)}{L^2(r> R)[0,T]}^2 +
		\mathcal{E}(a,R)
		\ . \notag
\end{equation}
Adding this to \eqref{0-1/2_mult_est} and \eqref{1/2-1_mult_est}
we have \eqref{a_ext_est}. 
\end{proof}

\begin{proof}[Proof of \eqref{1_ext_est}]
Adding estimate \eqref{hardy2} to 
\eqref{K0_est} gives:
\begin{multline}
	 \sup_{0\leq t\leq T} \slp{\phi(t)}{E^1(r>\max\{R, \frac{1}{2}t\})} 
	 +\lp{\tau_x \tau_+^{- \frac{1}{2}}   (\partial\phi, \tau_+^{-1}\phi) }{\ell^\infty_t L^2( R<r<\frac{3}{4}t)[0,T]}
	 + \slp{\tau_+ \tau_x^{-1}\partial_r^b(\tau_x \phi)}{\NLE(r>R)[0,T]}\\
	  \lesssim 
	  \sup_{0\leq t\leq T}
		 R^\frac{1}{2}\slp{\tau_-^\frac{1}{2} (\partial\phi,\tau_x^{-1}\phi)(t)}{L^2_x(\frac{1}{2}R<r<R)}+
	 \sqrt{\mathcal{E}(1,R)} \ . \notag
\end{multline}
On the other hand a straight forward integration of fixed time norms  in the region $r>\frac{1}{2}t$ also gives:
\begin{equation}
		\lp{\phi}{\ell^\infty_t S^{1}(r>\max\{R,\frac{1}{2}t\})[0,T]} 
		 \ \lesssim\  \sup_{0\leq t\leq T} \slp{\phi(t)}{E^1(r>\max\{R,\frac{1}{2}t\}) }
		 \ . \notag
\end{equation}
Combining the previous two lines  with the  definition of $S^{1,\infty}$ finishes the proof.
\end{proof}

 
\section{Estimates for Commutators}\label{comm_sect}

In this section we complete the proof of  estimates \eqref{LE_null_est_k}, \eqref{Sa_est_k},
and \eqref{CE_est_k}.


\subsection{Splitting into interior and exterior estimates}

In the interior we use:

\begin{prop}[General interior estimates for commutators]\label{int_comm_prop}
Fix a multiindex $|I|=k\geqslant 1$ and
let   $\Gamma^I$ denote a product in
$\mathbb{L}=\{S,\Omega_{ij},
{\partial_u^b},\partial_i^b-\omega^i{\partial_u^b}\}$.
Then for $R$ sufficiently large and any $R'>R$ 
one has the following  estimates:
\begin{align}
		  \slp{[\Box_g, \Gamma^I] \phi}{\WLE^{*,s}(r<R)[0,T]}  &\lesssim_R  
		  \slp{\phi }{H^{s+3}_{k-1}(r<R)[0,T]}  \ , \label{int_LEk_est}\\
		\slp{ [\Box_g, \Gamma^I] \phi}{  N^a(r<R)[0,T]}  &\lesssim_{R,R'} 
		\slp{\tau_+^{a-1}\phi }{   H^1_{k+1}(r<R')[0,T]}+
		 \slp{ \phi }{  S^a_{k-1}[0,T]}+
		   \slp{\Box_g\phi }{  N^a_{k}[0,T]} \ , \label{int_Sk_to_Nk} \\
		\slp{ [\Box_g, \Gamma^I] \phi}{ \ell^\infty_t N^1(r<R)[0,T]}   &\lesssim_{R,R'}  
		\lp{ \phi }{ \ell^\infty_t H^1_{k+1}(r<R')[0,T]}+
		\slp{ \phi }{  S^{1,\infty}_{k-1}[0,T]}+
		 \slp{\Box_g\phi }{\ell^\infty_t  N^1_{k}[0,T]} \ . \label{end_int_Sk_to_Nk}
\end{align}
The bound \eqref{int_Sk_to_Nk} is uniform in $0\leq a\leq 1$.
\end{prop}

For the exterior we use:

\begin{prop}[General exterior bounds for commutators]
Fix multiindices $|I|=k\geqslant 1$ and $|J|=s\geqslant 0$, and
let $\Gamma^I$ denote a product of vector fields in  ${\mathbb{L}=\{S,\Omega_{ij},
{\partial_u^b},\partial_i^b-\omega^i{\partial_u^b}\}}$ while
 $\td{\Gamma}^J$ denotes a product of vector fields in 
$\mathbb{L}_0=\{{\partial_u^b},\partial_i^b-\omega^i{\partial_u^b}\}$.
Let $R_0$ be as in the definition of the norms \eqref{class_LE_norms}. 
Then for $R\geq 2R_0>0$ sufficiently large one has the following:
\begin{enumerate}[I)]

\item 
Corresponding to  $a=0$ one has:
\begin{equation}
		\slp{ [\Box_g, \Gamma^I] \phi}{(\WLE^{*,s}+L^1_t H^s_x)(r>R)[0,T]}  \lesssim 
		o_R(1)\cdot \slp{\phi }{\WLE^s_k(r>R)[0,T]} + \slp{ \Box_g\phi  }{(\WLE^{*,s}_{k-1}+L^1_t H^s_{x,k-1})(r>R)[0,T]}
		  \ . \label{ext_LEk_est}
\end{equation}
In addition let 
 $X=\chi_{>R}q {\partial_u^b}$,  where
$q=q(u)$ has the uniform bounds $| (\tau_-{\partial_u^b})^l q|
\lesssim_{l} 1$, 
where $\chi_{>R}$
is supported in $r>\frac{1}{2}R$ with the usual derivative bounds.
Then one has the integral estimate:
\begin{multline}
		\Big|\dint   
		 [\Box_g,\td{\Gamma}^J] \phi \cdot X\td{\Gamma}^J \phi  dV_g \Big| 
		 \  \lesssim  \
		  \ o_R(1)\cdot \big( \lp{\phi}{\WLE^s[0,T]}^2  +
		 \sup_{0\leq t\leq T}\lp{  \partial\phi(t)}{H^s_x }^2\big)\\
		 + \big(  \sup_{0\leq t\leq T}\lp{  \partial\phi(t)}{H^s_x }
		 +\lp{\phi}{\WLE^s[0,T]} \big) \cdot \lp{\Box_g\phi }{( \LE^{*,s}+L^1_t H^s_x)[0,T]}
		 \ . \label{dt_int_comm_est}
\end{multline}

\item For $0<a<1$ one has uniformly the collection of estimates:
\begin{equation}
		\lp{[\Box_g, \Gamma^I] \phi}{  N^a(r>R)[0,T]} \ \lesssim \
		o_R(1)\cdot \lp{\phi}{  S^a_k [0,T]} + \lp{\Box_g\phi }{  N^a_{k-1}[0,T]}
		  \ . \label{ext_Sk_to_Nk}
\end{equation}

\item Finally, corresponding to the case $a=1$ we have:
\begin{equation}
		\lp{[\Box_g, \Gamma^I] \phi}{ \ell^\infty_t N^{1}(r>R)[0,T]} \ \lesssim \
		o_R(1)\cdot \lp{\phi}{  S^{1,\infty}_k[0,T]} + \lp{\Box_g\phi }{  \ell^\infty_t N^{1}_{k-1}[0,T]}
		\ . \label{end_ext_Sk_to_Nk}
\end{equation}
In addition  let
$X=\chi_{>R}q K_0$,  where  $K_0=(1+u^2) {\partial_u^b} + 2(u+r)r\partial_r^b $,  and 
where $q$ and $\chi_{>R}$ are as previously stated.
Then one has:
\begin{multline}
		\Big|\dint  
		 [\Box_g,\Gamma^I] \phi \cdot \tau_x^{-1} X ( \tau_x \Gamma^I \phi)  dV_g \Big| 
		\  \lesssim   \ 
		  \ o_R(1)\cdot \big( \lp{\phi}{S^{1,\infty}_k[0,T]}^2  +
		 \sup_{0\leq t\leq T}\lp{\phi(t)}{E^1_k }^2\big) \\ 
		 + \lp{\phi}{S^{1,\infty}_k[0,T]}\cdot \lp{\Box_g\phi }{ N^{1,1}_{k-1} [0,T]}
		 \ . \label{K0_int_comm_est}
\end{multline}
\end{enumerate}
\end{prop}

We will also need the following initial data bound.

\begin{lem}[Initial data bound]\label{data_lem}
Assume that the level sets $t=const$ are uniformly spacelike. Then
one has the uniform estimate for $0\leqslant a\leqslant 1$:
\begin{equation}
		  \lp{\tau_x^a \partial \phi (0)}{H^s_{x,k}}
		\ \lesssim \ 
		\sum_{|I|\leq k}\sum_{|J|\leq s}  \lp{\tau_x^a (\tau_x \partial_x)^I \partial_x^J\partial  \phi(0) }{ L^2_x}
		+ \lp{\tau_x^a  \Box_g \phi  }{L^1_t H^{s}_{x,k}[0,1]} \ . \label{data_bound}
\end{equation}
\end{lem}

Before giving the proof of these individual components we use them to establish   Theorem \ref{main_thm}.

\begin{proof}[Proof of estimate \eqref{LE_null_est_k}]
We need to treat separately the cases $k=0$ and $k>0$.

\case{1}{$k=0$}
In light of the Assumption \eqref{basic_LE} and the data bound \eqref{data_bound} we need to show:
\begin{equation}
		\sum_{|J| \leq s} 
		\lp{   (\partial^b)^J \phi}{\WLE^0[0,T]}
		  \ \lesssim\   \lp{\partial \phi(0)}{H^{s}_x} 
		+ \lp{\Box_g \phi }{ (\WLE^{*,s}+L^1_t H^s_x) [0,T]} \ . \notag
\end{equation}
It suffices to show this bound for $(\partial^b)^J\phi$ replaced by   $\td{\Gamma}^J\phi$, where 
the product is taken over vector fields in $\mathbb{L}_0=\{{\partial_u^b},\partial_i^b-\omega^i{\partial_u^b}\}$.
Applying estimate \eqref{LE_null_est} to $\td{\Gamma}^J\phi$, then using  Assumption \eqref{basic_LE}
and \eqref{dt_int_comm_est}, we find that for $|J|\leq s$:
\begin{multline}
		\lp{   (\partial^b)^J \phi}{\WLE^0[0,T]}^2
		 \ \lesssim \ C_R\big( \lp{ \partial\phi(0)}{H^s_x}^2 + \lp{\Box_g \phi }{( \WLE^{*,s}+L^1_t H^s_x) [0,T]}^2\big) +
		 o_R(1)\cdot  \lp{\phi}{\WLE^s[0,T]}^2 \\
		  \lp{\phi}{\WLE^s[0,T]}\cdot  \lp{\Box_g\phi }{ ( \WLE^{*,s}+L^1_t H^s_x)[0,T]}) \ . \notag
\end{multline}
Summing this in $|J|\leq s$ and  taking $R$ sufficiently large finishes the proof.

\case{2}{$k>0$}
This is a straight forward application of estimate \eqref{LE_null_est_k} with $k=0$
applied to $\Gamma^I\phi$, where $\Gamma^I$ denotes a product in
$\mathbb{L}=\{S,\Omega_{ij},
{\partial_u^b},\partial_i^b-\omega^i{\partial_u^b}\}$. Using estimates
  \eqref{int_LEk_est} and \eqref{ext_LEk_est} for  sufficiently large $R$ to control the commutator, 
followed by the data bound \eqref{data_bound}, completes the proof.  Note that
the value  $R$ becomes the definition of $B_0(s,k)$ on line \eqref{LE_null_est_k}.
\end{proof}

\begin{proof}[Proof of estimate \eqref{Sa_est_k}]
This follows by combining estimate \eqref{NLE_est} with \eqref{int_Sk_to_Nk},  
\eqref{ext_Sk_to_Nk}, and \eqref{data_bound}, and then using an induction on $k$.
Note that the value  $R$ becomes the definition of $B_a(s,k)$ on line \eqref{Sa_est_k}.
\end{proof}

\begin{proof}[Proof of estimate \eqref{CE_est_k}]
This  follows from   estimates \eqref{CE_est}, \eqref{end_int_Sk_to_Nk}, \eqref{end_ext_Sk_to_Nk},
\eqref{K0_int_comm_est}, and \eqref{data_bound} in a similar pattern as the previous proof.
\end{proof}


\subsection{Proof of the initial data bound and the interior estimates}

\begin{proof}[Proof of \eqref{data_bound}]
By Remark \ref{coord_rem} we can assume $u\equiv t-\tau_x$ when $t\in [0,1]$. Therefore 
without loss of generality we may replace  $\Gamma^I$   with products of vector fields
 in $\mathbb{L}_{minkowski}=\{t\partial_t+r\partial_r,
x^i\partial_j-x^j\partial_i,\partial_\alpha\}$. 
Notice that for products in $\mathbb{L}_{minkowski}$ we have the following equivalence at $t=0$:
\begin{equation}
		\lp{\tau_x^a   \phi^{(k)}(0) }{ H^s_x} \ \approx \ \sum_{|I|+j\leq k}  \lp{\tau_x^a 
		(\tau_x\partial_x)^I  \partial_t^j \phi(0)}{ H^s_x}     \ , \label{minkowski_vf_eq}
\end{equation}
for a similarity constant that depends only on $k$. In addition to this 
from the uniform spacelike condition of $t=0$ we have the identity:
\begin{equation}
		\partial_t^2 \ = \ (g^{00})^{-1}\big(\Box_g - P(t,x;D)\big) \ , \label{d_t^2_iden}
\end{equation}
where $P(t,x;D)$ is second order with uniformly $(r\partial)^J$ homogeneous coefficients,
is at most first order in $\partial_t$, and contains no zero order term. 

First we repeatedly  use the $\lesssim$
direction of \eqref{minkowski_vf_eq} for $\phi(0)$, followed by the substitution \eqref{d_t^2_iden} for terms
containing more that one copy of $\partial_t$. We then use the $\gtrsim$ direction
of \eqref{minkowski_vf_eq} for any terms which are produced which contain
 $\Box_g\phi(0)$. From this sequence of steps  we have the bound:
\begin{equation}
		 \lp{\tau_x^a \partial  \phi (0)}{H^s_{x,k}}
		\ \lesssim \ \sum_{|I|\leq k}   \lp{\tau_x^a   (\tau_x\partial_x)^I  \partial  \phi(0) }{ H^s_x}
		+   \lp{\tau_x^a 
		(\Box_g \phi)^{(k-1)} (0) }{ H^{s}_x}  \ . \notag
\end{equation}
After a local $L^1(dt)$ trace estimate the second RHS term above matched RHS\eqref{data_bound}. 

For the first RHS term on the previous line
we again use identity \eqref{d_t^2_iden} to successively get rid of all additional $\partial_t$ derivatives 
from the $H^s_x$ norm,
followed by another application of the $\gtrsim$ direction 
\eqref{minkowski_vf_eq} for terms produced containing $\Box_g\phi(0)$. This gives:
\begin{equation}
		 \sum_{|I|\leq k}   \lp{\tau_x^a   (\tau_x\partial_x)^I  \partial  \phi(0) }{ H^s_x} \ \lesssim \
		 \sum_{|I|\leq k} \sum_{|J|\leq s}  \lp{\tau_x^a   (\tau_x\partial_x)^I  \partial_x^J \partial  \phi(0) }{ L^2_x} 
		 +     \lp{\tau_x^a    (\Box_g \phi)^{(k)}(0) }{ H^{s-1}_x} \ . \notag
\end{equation}
 The proof now concludes with a final local $L^1(dt)$ trace estimate the second RHS term above.
\end{proof}

We now move on to the proof of Proposition \ref{int_comm_prop}. Note that estimate 
\eqref{int_LEk_est} follows more or less immediately from \eqref{mult_comm_iden1}.
Therefore we focus attention on the last two estimates. 

\begin{proof}[Proof of estimates \eqref{int_Sk_to_Nk} and \eqref{end_int_Sk_to_Nk}]
We will only focus here in showing \eqref{int_Sk_to_Nk}. The proof of \eqref{end_int_Sk_to_Nk}
follows from identical calculations by replacing all  $L^2_t$ norms with  $\ell^\infty_tL^2_t$ norms.

\step{1}{Inductive setup}
It suffices to show that for multiindex $|I|=k$ and $R $ sufficiently large, then for 
 $R'>R$ we have:
\begin{multline}
		\lp{\tau_+^a [\Box_g, \Gamma^I] \phi}{  H^1(r<R)[0,T]} \lesssim_{R,R'}
		\lp{\tau_+^{a-1}\phi^{(k+1)}}{    H^1(r<R')[0,T]}
		 + \lp{\tau_+^a \phi^{(k-1)} }{   L^2(r<R')[0,T]}\\
		 +   \lp{\tau_+^a  ( \Box_g \phi)^{(k)}}{ H^1(r<R_0) [0,T]}
		+   \lp{\tau_+^a  ( \Box_g \phi)^{(k)}}{ L^2(r<R') [0,T]}
		\ . \label{gamma_to_nabla_bnd}
\end{multline}
This boils down to an induction. Indeed, for  fixed nontrivial
 $I$  and integer $s\geq 1$ it suffices to show that 
under the same assumptions one has:
\begin{multline}
		\lp{\tau_+^a [\Box_g, \Gamma^I] \phi}{  H^s(r<R)[0,T]} \lesssim_{R,R'}
		\lp{\tau_+^{a-1}\phi^{(|I|+s)}}{    H^1(r<R')[0,T]}
		+ \lp{\tau_+^a \phi^{(|I|-1)} }{    L^2(r<R')[0,T]}\\
		+\lp{\tau_+^a(\Box_g\phi)^{(|I|-1)} }{H^{s+1}(r<R_0)}
		+\lp{\tau_+^a(\Box_g\phi)^{(|I|-1)} }{H^{s}(r<R')}
		+ \sum_{I'\subsetneq I} \lp{\tau_+^a  [\Box_g, \Gamma^{I'}] \phi}{  H^{s+1}(r<R')[0,T]}
		\ . \label{gamma_to_nabla_bnd'}
\end{multline}
By repeatedly applying this last estimate for $|I|+s=k+1$, where
$s=1,\ldots k$, and a sequence $R<R'_{s}<R'_{s+1}<R'$
we have \eqref{gamma_to_nabla_bnd}.

\step{2}{Elliptic estimate in $\frac{1}{2}R_0<r<R$}
To prove \eqref{gamma_to_nabla_bnd'} start with \eqref{mult_comm_iden1} which implies that:
\begin{equation}
		\lp{\tau_+^a [\Box_g, \Gamma^I] \phi}{  H^s(r<R)[0,T]} \ \lesssim \ 
		\sum_{I'\subsetneq I}\lp{\tau_+^a\Gamma^{I'}\phi }{H^{s+2}(r<R)[0,T]} 
		\ . \label{local_comm_pt1}
\end{equation}
Without loss of generality we may assume   $R_0$ is chosen large enough 
 that in the region  $r>\frac{1}{2}R_0$
the operator $P(x,D)=\Box_g-Q_0(x,D)$ is uniformly elliptic, where $Q_0$ contains all terms
with a $g^{0\alpha}$ factor. Standard elliptic estimates then give:
\begin{equation}
		\lp{\tau_+^a \Gamma^{I'}\phi }{H^{s+2}(\frac{1}{2}R_0<r<R) [0,T]}
		\ \lesssim_{R,R'} \ \lp{\tau_+^a P(x,D) \Gamma^{I'}\phi }{H^{s}(r<R') [0,T]}
		+ \lp{\tau_+^a   \Gamma^{I'}\phi }{L^2(r<R') [0,T]} \ . \notag
\end{equation}
On the other hand due to the fact that any term in $Q_0(x,D)$  contains at least one
time derivative, and using the metric conditions \eqref{mod_coords},
we have for $I'\subsetneq I$:
\begin{equation}
		\lp{\tau_+^a Q_0(x,D) \Gamma^{I'}\phi }{H^{s}(r<R') [0,T]} \ \lesssim_{R'} \ 
		\lp{\tau_+^{a-1}\phi^{(|I|+s)}}{    H^1(r<R')[0,T]} \ . \notag
\end{equation}
Thus, combining the last two lines we have:
\begin{multline}
		\sum_{I'\subsetneq I}\lp{\tau_+^a\Gamma^{I'}\phi }{H^{s+2}(\frac{1}{2}R_0<r<R)[0,T]}
		\ \lesssim_{R,R'} \ \lp{\tau_+^{a-1}\phi^{(|I|+s)}}{    H^1(r<R')[0,T]}
		+ \lp{\tau_+^a \phi^{(|I|-1)} }{    L^2(r<R')[0,T]}\\
		+\lp{\tau_+^a(\Box_g\phi)^{(|I|-1)} }{H^{s}(r<R')}
		+ \sum_{I'\subsetneq I} \lp{\tau_+^a  [\Box_g, \Gamma^{I'}] \phi}{  H^{s}(r<R')[0,T]} \ . \label{elliptic_out}
\end{multline}

\step{3}{$\LE$ estimate in $r<\frac{1}{2}R_0$}
It remains to bound the portion of RHS\eqref{local_comm_pt1} which is contained 
in $r<\frac{1}{2}R_0$. Note that we only need to focus on the region $t>1$ as RHS\eqref{local_comm_pt1}
restricted to the time slab $[0,1]$ is automatically bounded by $\lp{\tau_+^{a-1}\phi^{(|I|+s)}}{    H^1(r<R')[0,T]}$.

We begin by applying the stationary $\LE$ bound \eqref{basic_stat_LE} at regularity $s+1$ 
to $\chi_{t>1}\chi_{r<\frac{1}{2}R_0}\tau_+^a\Gamma^{I'}\phi$, 
where $\chi_{r<\frac{1}{2}R_0}=1$ on $r<\frac{1}{2}R_0$ and $\chi_{r<\frac{1}{2}R_0}=0$ on $r>R_0$, and with similar 
properties for $\chi_{t>1}$.
This results in the estimate:
\begin{equation}
		\sum_{I'\subsetneq I}\lp{\tau_+^a\Gamma^{I'}\phi }{H^{s+2}(r<\frac{1}{2}R_0)[1,T]}  \ \lesssim_{R,R'} \  T_1+T_2+T_3+T_4
		\ ,  \notag
\end{equation}
where:
\begin{align}
		T_1 \ &= \ \sum_{I'\subsetneq I}  \lp{ \Gamma^{I'}\phi}{    H^{s+2}(r<R_0)[0,2]} \ ,  
		&T_2 \ &= \ \sum_{I'\subsetneq I}  \lp{( \partial_t \tau_+^a \Gamma^{I'}\phi ,
		 \tau_+^a\Gamma^{I'}\phi)}{H^{s+1}\times L^2(r<R_0)[0,T]} \ , \notag\\
		T_3 \ &=\ \sum_{I'\subsetneq I} \lp{[\Box_g,\chi_{r<\frac{1}{2}R_0}\tau_+^a]\Gamma^{I'}\phi }{H^{s+1} [0,T]} \ ,
		&T_4 \ &= \ \sum_{I'\subsetneq I} \lp{\tau_+^a\Box_g \Gamma^{I'}\phi }{H^{s+1}(r<R_0)[0,T]} \ . \notag
\end{align}
The term $T_1$ results from differentiation of $\chi_{t>1}$.
The terms $T_1$ and $T_4$ are already compatible with RHS\eqref{gamma_to_nabla_bnd'}. It is also easy to see that:
\begin{equation}
		T_2 \ \lesssim \ \lp{\tau_+^{a-1}\phi^{(|I|+s)}}{    H^1(r<R_0)[0,T]} 
		+ \lp{\tau_+^a \phi^{(|I|-1)} }{    L^2(r<R_0)[0,T]}
		\ . \notag
\end{equation}
It remains to bound $T_3$. Expanding the commutator gives:
\begin{equation}
		T_3 \ \lesssim \ \lp{\tau_+^{a-1}\phi^{(|I|+s)}}{    H^1(r<R_0)[0,T]}+
		\sum_{I'\subsetneq I} \lp{\tau_+^a \Gamma^{I'}\phi }{H^{s+2}(\frac{1}{2}R_0<r<R_0) [0,T]} \ . \notag
\end{equation}
The first term above is  of the correct form. The second term is handled by estimate \eqref{elliptic_out}.
\end{proof}


\subsection{Proof of the exterior estimates}

We first list some general calculations which  take care
of a large portion of the desired estimates.

\begin{lem}\label{gen_comm_lem}
For integers $s,k\geq 0$   define:
\begin{equation}
		\Phi^{(s,k)}  \ = \ \sum_{|J|\leq s}\tau_x^{-\frac{1}{2}}  \big| (\tau_0 
		\partial (\partial^b)^{J} \phi^{(k)}, \partial^b_x(\partial^b)^{J} \phi^{(k)}
		,\tau_x^{-1}(\partial^b)^{J} \phi^{(k)})\big| \ , \qquad
		\slash\!\!\!\! \Phi^{(k)} \ = \ \tau_x^{-\frac{1}{2}}  \tau_0^{-\frac{1}{2}}\, \big| \tau_x^{-1}\partial_r^b(\tau_x\phi^{(k)})\big| 
		\ . \notag
\end{equation}
We also use the shorthand $\Phi^{(0,k)} =\Phi^{(k)} $. 
With this notation we have:
\begin{align}
		\lp{ \Phi^{(s,k)}}{ \ell^\infty_r L^2(r>R)[0,T]} 
		+ \lp{\tau_0^{-\frac{1}{2}} \Phi^{(s,k)} }{\ell^\infty_u  \ell^\infty_r
		L^2(\frac{1}{2}t<r<2t)\cap(r>R)[0,T]} \ &\lesssim \ \lp{\phi }{\LE^s_k(r>R)[0,T]}
		\ , \label{Phi_LE}\\
		\lp{\tau_+^a \Phi^{(k)}}{ \ell^\infty_r L^2[0,T]} 
		+ \lp{\tau_+^a \ \slash\!\!\!\! 
		\Phi^{(k)}  }{\ell^\infty_u L^2(\frac{1}{2}t<r<2t)[0,T]} \ &\lesssim \ \lp{\phi}{S^{a}_k[0,T]}
		\ , \label{Phi_Sa}\\
		\lp{\tau_+ \Phi^{(k)}}{\ell^\infty_t\ell^\infty_r L^2[0,T]} 
		+ \lp{\tau_+\ \slash\!\!\!\! 
		\Phi^{(k)}}{\ell^\infty_u L^2(\frac{1}{2}t<r<2t)[0,T]} \ &\lesssim \ \lp{\phi}{S^{1,\infty}_k[0,T]}
		\ , \label{Phi_S1}\\
		\lp{\tau_x^\frac{1}{2} \tau_+ \Phi^{(k)}}{L^\infty_t L^2_x[0,T]} 
		 \ &\lesssim \ \sup_{0\leq t\leq T}\lp{\phi(t)}{E^1_k}
		\ . \label{Phi_E1}
\end{align}
In addition if $\Gamma^I$ is a product of vector fields
in $ \mathbb{L}=\{S,\Omega_{ij},{\partial_u^b} , \partial_i^b-\omega^i{\partial_u^b}\}$, and $\td{\Gamma}^J$ a 
product of vector fields in $\mathbb{L}_0=\{ {\partial_u^b} , \partial_i^b-\omega^i{\partial_u^b}\}$, where $|I|=k$ and $|J|=s$,
then we have the pointwise estimates:
\begin{equation}
		\big|\td{\Gamma}^{J} [\Box_g,\Gamma^I] \phi \big| \ 
		\lesssim \ q\cdot \tau_0^{-\frac{1}{2}}\tau_x^{-\frac{1}{2}}
		 \Phi^{(s,k)} + \sum_{|J|\leq s}|\partial^J (\Box_g \phi )^{(k-1)}| \ , \qquad
		\hbox{where\ \ \ } q\in  \mathcal{Z}^{0}
		 \label{Phi_comm}
\end{equation}
\end{lem}

\begin{proof}[Proof of Lemma \ref{gen_comm_lem}]
The proof of \eqref{Phi_LE}--\eqref{Phi_E1} is a straightforward application of the definition
of the various spaces involved. On the other hand 
 \eqref{Phi_comm} is immediate from 
\eqref{main_ptws_comm}.
\end{proof}

Note that a direct combination of \eqref{Phi_comm} with either \eqref{Phi_LE}, \eqref{Phi_Sa}, or 
\eqref{Phi_S1} shows \eqref{ext_LEk_est}, \eqref{ext_Sk_to_Nk}, or  \eqref{end_ext_Sk_to_Nk}
(resp). Thus, the remainder of the subsection is devoted to showing the   integral
estimates \eqref{dt_int_comm_est} and \eqref{K0_int_comm_est}. In both cases the key step
is to integrate by parts the bilinear operator resulting from the commutator with $\Box_g$. The
relevant result here is:

\begin{lem}\label{IBP_lem}
Let $R$ be sufficiently large so that $\mathcal{K}\subseteq \{|x|<\frac{1}{2} R\}$, 
and let $\chi_{>R}$ be a cutoff
supported in $r>\frac{1}{2}R$,  constant for $r>R$,  
with the usual derivative bounds. Then one has the
following integral estimates:
\begin{enumerate}[I)]

\item Let
$q=q(u)$ be a smooth function such that $| (\tau_-{\partial_u^b})^l q|
\lesssim_{l} 1$.
Furthermore let $|J|=s\geq 1$ and $|J'|\leq s-1$ and let $\td{\Gamma}^J$ and 
$\td{\Gamma}^{J'}$ denote products of vector fields in 
$\mathbb{L}_0=\{{\partial_u^b},\partial_i^b-\omega^i{\partial_u^b}\}$.
Then if $\mathcal{R}$ obeys the conditions  
\eqref{R_X_sym_bnds}   with $a=c=-1$ and $b=1$ one has:
\begin{equation}
		 \Big| \int_0^T\!\!\!\!\int_{\mathbb{R}^3}  \chi_{>R}\nabla_\alpha \mathcal{R}^{\alpha\beta}
		 \nabla_\beta \td{\Gamma}^{J'}\phi  \cdot q {\partial_u^b} \td{\Gamma}^J\phi  dV_g \Big| \ \lesssim \ 
		 o_R(1)\cdot \big( \lp{\phi}{\LE^s(r>\frac{1}{2}R)[0,T]}^2  +
		 \sup_{0\leq t\leq T}\lp{  \partial\phi(t)}{H^s_x }^2\big) \ . \label{du_IBP}
\end{equation}

\item Alternatively, suppose that $q=q(u)$ satisfies $|(\tau_-{\partial_u^b})^l q|
\lesssim_{l} \tau_-$. Let
$\Gamma^I$ and $\Gamma^{I'}$ be products of vector fields in 
${\mathbb{L}=\{S,\Omega_{ij},{\partial_u^b},\partial_i^b-\omega^i{\partial_u^b}\}}$ for multiindices
$|I|=k\geq 1$ and $|I'|\leq k-1$. Then if 
$\mathcal{R}$ obeys the conditions  \eqref{R_sym_bnds} one has:
\begin{equation}
		 \Big| \int_0^T\!\!\!\!\int_{\mathbb{R}^3}  \chi_{>R} \nabla_\alpha \mathcal{R}^{\alpha\beta}
		 \nabla_\beta \Gamma^{I'}\phi  \cdot q S \Gamma^I\phi  dV_g \Big| \ \lesssim \ 
		 o_R(1)\cdot \big( \lp{\phi}{S^{1,\infty}_k[0,T]}^2  +
		 \sup_{0\leq t\leq T}\lp{\phi(t)}{E^1_k }^2\big)
		  \ . \label{K0_IBP}
\end{equation}

\item
Finally, with  the same setup of estimate \eqref{K0_IBP} let $\mathcal{S}$ obey the conditions \eqref{S_sym_bnds}.
Then one has:
\begin{equation}
		 \Big|\int_0^T\!\!\!\!\int_{\mathbb{R}^3}  \chi_{>R}  
		 \mathcal{S}^\alpha \partial_\alpha \Gamma^{I'}\phi  \cdot q S \Gamma^I\phi  dV_g \Big| \ \lesssim \ 
		 o_R(1)\cdot \big( \lp{\phi}{S^{1,\infty}_k[0,T]}^2  +
		 \sup_{0\leq t\leq T}\lp{\phi(t)}{E^1_k }^2\big)
		  \ . \label{SS_IBP}
\end{equation}
\end{enumerate}
\end{lem}

\begin{proof}[Proof of estimate \eqref{du_IBP}]
For functions $F$ and $G$, a vector field $X$, and quadratic operator 
$\nabla_\alpha \mathcal{R}^{\alpha\beta}\nabla_\beta$, we have the pointwise identity:
\begin{align}
		\nabla_\alpha \mathcal{R}^{\alpha\beta}\nabla_\beta F \cdot XG \ &= \ \nabla_\alpha \Big[ 
		\mathcal{R}^{\alpha\beta}\nabla_\beta F \cdot XG  -  X^\alpha 
		\mathcal{R}^{\beta\gamma} \partial_\beta  F \cdot \partial_\gamma G \Big]  \label{DBP}\\
		&\hspace{.3in}+ (\nabla_\gamma  X^\gamma )
		\mathcal{R}^{\alpha\beta} \partial_\alpha  F \cdot \partial_\beta G
		+ \mathcal{R}_X^{\alpha\beta} \partial_\alpha  F \cdot \partial_\beta G
		+ \mathcal{R}^{\alpha\beta} \partial_\alpha X F \cdot \partial_\beta G \ , \notag \\
		&= \ \nabla_\alpha T_1^\alpha +T_2 + T_3 +T_4 \ , \notag
\end{align}
where $\mathcal{R}_X=\mathcal{L}_X \mathcal{R}$. Setting $F=\td{\Gamma}^{J'}\phi$, 
$G=\td{\Gamma}^J\phi$, and $X=q{\partial_u^b}$
in the above formula we estimate the integral of each term separately.

\case{1}{The $T_1$ term}
Using the divergence theorem gives:
\begin{equation}
		\big| \int_0^T\!\!\!\!\int_{\mathbb{R}^3}  \chi_{>R} 
		\nabla_\alpha T_1^\alpha dV_g\big| \ \lesssim \ 
		\sup_{0\leq t\leq T} \int_{|x| >\frac{1}{2} R } |\nabla_\alpha t|\cdot  |T^\alpha_1| dx  
		+ R^{-1}  \big| \int_0^T\!\!\!\!\int_{\mathbb{R}^3} \chi'_{R}
		 T_1^r  dV_g  \big|
		 \ , \label{dividen}
\end{equation}
where $|\chi'_{R}|\lesssim 	1$ and is supported where $\frac{1}{2}R<r<R$.
Based on the fact that
all components of  $X$ are uniformly bounded, and all components of
$\mathcal{R}$ are $o_R(1)$ (in either Bondi or $(t,x)$ coordinates),
we directly have the pointwise estimate:
\begin{equation}
		|\nabla_\alpha t|\cdot  |T^\alpha_1| \ \lesssim \ \sup_\alpha  |T^\alpha_1|
		\ \lesssim \ o_R(1)\sum_{|J''|\leq s} |\partial  \partial^{J''}\phi|^2 \ , \notag
\end{equation}
which in turn produces a bound for RHS\eqref{dividen}
in terms of $o_R(1)\cdot \big( \slp{\phi}{\LE^s(r>\frac{1}{2}R)[0,T]}^2  +
\sup_{0\leq t\leq T}\slp{  \partial\phi(t)}{H^s_x }^2\big)$.

\case{2}{The $T_2$ term}
Notice that estimate \eqref{det_bnds} gives $|\nabla_\gamma X^\gamma|\lesssim 1$ for $X=q{\partial_u^b}$. On the
other hand for $\mathcal{R}$ satisfying conditions  
\eqref{R_X_sym_bnds}   with $a=c=-1$ and $b=1$ we have the pointwise estimate:
\begin{equation}
		|\mathcal{R}^{\alpha\beta}\partial_\alpha \td{\Gamma}^{J'}\phi \partial_\beta \td{\Gamma}^J \phi |
		\ \lesssim \ p\cdot \sum_{|J''|\leq s}\tau_x^{-1}( |\partial \partial^{J''}\phi|^2 
		+ \tau_0^{-1} |\partial_x^b\partial^{J''}\phi|^2) \ , 
		\qquad \hbox{where \ \ \ } p\in \mathcal{Z}^0 \ . \label{du_R_ptws}
\end{equation}
This suffices to give:
\begin{equation}
		\big| \int_0^T\!\!\!\!\int_{\mathbb{R}^3}  \chi_{>R}
		 T_2\, dV_g \big| \ \lesssim \  o_R(1)\cdot  \lp{\phi}{\LE^s(r>\frac{1}{2}R)[0,T]}^2 \ . \notag
\end{equation}

\case{3}{The $T_3$ term}
This is similar to the previous step. Notice that $X=q{\partial_u^b}$ obeys the symbol bounds
\eqref{vect_sym_bnds} with $a=b=0$ and $c=-1$, and satisfies all conditions on line \eqref{good_X_conds}.
Therefore, thanks to Case \ref{main_LR_estimate}) of Lemma \ref{basic_lie_lem}
we have that $\mathcal{R}_X$ satisfies the bounds on line  \eqref{R_X_sym_bnds} 
with $a=c=-1$ and $b=1$.
This is enough to show \eqref{du_R_ptws} holds for $\mathcal{R}$ replaced by $\mathcal{R}_X$.

\case{4}{The $T_4$ term}
Modulo another bound similar to \eqref{du_R_ptws} it suffices to show:
\begin{equation}
		|(\partial_\alpha  q) \mathcal{R}^{\alpha\beta}  {\partial_u^b} \td{\Gamma}^{J'}\phi \partial_\beta \td{\Gamma}^J\phi |
		\ \lesssim \ RHS\eqref{du_R_ptws} \ . \notag
\end{equation}
This follows from direct inspection of various terms involved.
\end{proof}

\begin{proof}[Proof of estimate \eqref{K0_IBP}]
We again use the identity \eqref{DBP} and estimate each term separately. As a preliminary note that 
with the assumptions of  \eqref{K0_IBP} and notation of Lemma \ref{gen_comm_lem}
one has the pointwise estimate:
\begin{equation}
		\tau_- | \mathcal{R}^{\alpha\beta}
		 \partial_\alpha \Gamma^{I'}\phi  \cdot  \partial_\beta \Gamma^I\phi | \ \lesssim \ p\cdot 
		 \tau_x  \tau_+ |\Phi^{(k)}|^2  \ , \qquad \hbox{where \ \ \ } p\in\mathcal{Z}^{0} \ . \label{R_quad_est}
\end{equation}
A similar bound holds if we replace $\Gamma^{I'}\phi$ by $S \Gamma^{I'}\phi $. 

Likewise, when $X=q S$
with $q=q(u)$ and bounds $| (\tau_-{\partial_u^b})^l q|
\lesssim_{l} \tau_-  $ we have condition \eqref{vect_sym_bnds}
with $a=b=0$ and $c=1$, and also the conditions on line \eqref{good_X_conds}
save for the second identity.  
Therefore by   \ref{R_alt2}) of Lemma \ref{basic_lie_lem}
the tensor $\mathcal{R}_X=\mathcal{L}_X\mathcal{R}$ satisfies (in Bondi coordinates):
\begin{equation}
		\mathcal{R}_X^{ij} \ \in  \ \tau_+ \cdot \mathcal{Z}^0 \ , \qquad
		\mathcal{R}_X^{ui} \ \in \ \tau_+ \cdot \mathcal{Z}^1 \ , \qquad
		\mathcal{R}_X^{uu} \ \in \ \tau_+ \cdot \mathcal{Z}^2 \ . \notag
\end{equation}
In particular we have the pointwise estimate:
\begin{equation}
		| \mathcal{R}^{\alpha\beta}_X
		 \partial_\beta \Gamma^{I'}\phi  \cdot  \partial_\beta \Gamma^I\phi | \ \lesssim \ p\cdot 
		 \tau_x\tau_+  |\Phi^{(k)}|^2  \ , \qquad \hbox{where \ \ \ } p\in\mathcal{Z}^{0} \ . \label{R_quad_est'}
\end{equation}

\case{1}{The $T_1$ term}
Here we again use \eqref{dividen}. For the first term on 
RHS\eqref{dividen} estimate \eqref{Phi_E1} shows it 
suffices to prove:
\begin{equation}
		\sup_\alpha |T_1^\alpha| \ \lesssim \ o_R(1)\cdot \tau_x \tau_+^2 |\Phi^{(k)}|^2 \ . \label{T1_ptws}
\end{equation}
For the first term in $T_1^\alpha$ we use:
\begin{equation}
		\sup_\alpha \tau_- | \mathcal{R}^{\alpha\beta}
		 \partial_\beta \Gamma^{I'}\phi  \cdot S \Gamma^I\phi | \ \lesssim \ o_R(1)
		 \tau_x  \tau_+^2 |\Phi^{(k)}|^2  \ ,   \notag
\end{equation}
which follows from $|\mathcal{R}^{\alpha\beta}|\lesssim 1$ and expanding $S$ into ${\partial_u^b}$
and $\partial_x^b$ derivatives. For the second term in $T_1^\alpha$ we have \eqref{T1_ptws}
thanks to \eqref{R_quad_est} and $|X^\alpha|\lesssim \tau_+\tau_-$.

For the second term on RHS\eqref{dividen} the pointwise bound \eqref{T1_ptws} is not sufficient 
to recover the $\ell^\infty_t$ structure needed on RHS\eqref{K0_IBP}. However,
similar calculations to those  above show  $T^r_1=T^r_{11}+T^r_{12}$ where:
\begin{equation}
		R^{-1}\chi'_RT^r_{11} \ =\  \chi'_R\mathcal{S}^\alpha 
		 \partial_\alpha \Gamma^{I'}\phi  \cdot  qS \Gamma^I\phi
		\ , \qquad
		R^{-1}|\chi'_R T^r_{12}| \ \lesssim \ o_R(1)\cdot
		 \chi_R(r) \tau_x  \tau_+ |\Phi^{(k)}|^2
		 \ . \notag
\end{equation}
Here $\chi_R$ is a cutoff on a dyadic region
$\approx R$, and $\mathcal{S}^\alpha =R^{-1}\chi_R(r)
\mathcal{R}^{r\alpha}$
is a smooth vector field satisfying the assumptions of estimate \eqref{SS_IBP} (which will 
be proved independently). Therefore we only need bound the second term on the line above.
Using \eqref{Phi_E1} in the region $t\leq R$, and \eqref{Phi_S1} 
in the region $t>R$, we have:
\begin{equation}
		R^{-1}  \int_0^T\!\!\!\!\int_{\mathbb{R}^3} |\chi'_R T^r_{12}| dxdt \ \lesssim
		\ o_R(1)\cdot \big( \lp{\phi}{S^{1,\infty}_k[0,T]}^2  +
		 \sup_{0\leq t\leq T}\lp{\phi(t)}{E^1_k }^2\big) \ . \notag
\end{equation}
 
\case{2}{The $T_2$ term}
For the remaining three terms on line \eqref{DBP} we will set things up so as to appeal to Lemma \ref{L2L00_lem}.
For $i=2,3,4$ we will show:
\begin{equation}
		   \int_0^T\!\!\!\!\int_{\mathbb{R}^3}\chi_{>R}
		    |T_i| \, dV_g   \ \lesssim \  \lp{\sqrt{ \tau_x/\tau_+}p\cdot \tau_+ \Phi^{(k)}}{L^2[0,T](r>R)}^2
		   \ , \qquad\hbox{where \ \ \ } p^2\in\mathcal{Z}^0 \ , \label{Ti_L2_bound}
\end{equation}
which by a combination of estimate \eqref{Phi_L2_to_Linfty} and estimates 
\eqref{Phi_S1}  and \eqref{Phi_E1} produces:
\begin{equation}
		 \int_0^T\!\!\!\!\int_{\mathbb{R}^3} \chi_{>R}
		 |T_i| \, dV_g   \ \lesssim \  o_R(1)\cdot \big( \lp{\phi}{S^{1,\infty}_k[0,T]}^2  +
		 \sup_{0\leq t\leq T}\lp{\phi(t)}{E^1_k }^2\big) \ . \notag
\end{equation}
For the specific case of the $T_2$ term note that \eqref{det_bnds} shows $|\nabla_\gamma X^\gamma|\lesssim \tau_-$.
Then \eqref{R_quad_est} immediately gives \eqref{Ti_L2_bound}.

\case{3}{The $T_3$ term}
In this case \eqref{Ti_L2_bound} follows at once from  \eqref{R_quad_est'}.

\case{4}{The $T_4$ term}
Modulo an application of \eqref{R_quad_est} with $\Gamma^{I'}\phi $ replaced by $S \Gamma^{I'}\phi $,
to produce \eqref{Ti_L2_bound} for this case
it suffices to prove the pointwise estimate:
\begin{equation}
		|(\partial_\alpha q ) \mathcal{R}^{\alpha\beta} S \Gamma^{I'}\phi\cdot  \partial_\beta \Gamma^{I}\phi |
		\ \lesssim \ p\cdot \tau_x\tau_+ |\Phi^{(k)}|^2 \ , \qquad
		\hbox{where\ \ \ } p \in \mathcal{Z}^0 \ , \notag
\end{equation}
which follows from expanding $S$ to get:
\begin{equation}
		|  \mathcal{R}^{u\beta} S \Gamma^{I'}\phi\cdot  \partial_\beta \Gamma^{I}\phi |
		\ \lesssim \ p\cdot \tau_+ (\tau_0^2 |\partial \phi^{(k)}|^2 +   |\partial_x^b \phi^{(k)}|^2) 
		\ , \qquad \hbox{where\ \ \ } p \in \mathcal{Z}^0 
		\ . \notag
\end{equation}
This completes the proof of \eqref{K0_IBP}.
\end{proof}

\begin{proof}[Proof of \eqref{SS_IBP}]
For functions $F$ and $G$, a vector fields $\mathcal{S}$ and $X$,  we have the pointwise identity:
\begin{align}
		\mathcal{S}^\alpha \partial_\alpha F\cdot XG \ &= \ \nabla_\alpha( 
		X^\alpha \mathcal{S}^\beta \partial_\beta F\cdot G )
		- (\nabla_\alpha X^\alpha) \mathcal{S}^\beta \partial_\beta F\cdot G 
		- \mathcal{S}_X^\beta \partial_\beta F\cdot G - \mathcal{S}^\beta \partial_\beta X F\cdot G \ . \label{S_dividen}
\end{align}
Here $\mathcal{S}_X=\mathcal{L}_X \mathcal{S}=[X,\mathcal{S}]$. We again need to estimate each term separately.

As a general first step note if $\mathcal{S}$ satisfies \eqref{S_sym_bnds}, $|I'|\leq k-1$, and $|I|\leq k$ then:
\begin{equation}
		\tau_- | \mathcal{S}^{\alpha}
		 \partial_\alpha \Gamma^{I'}\phi  \cdot    \Gamma^I\phi | \ \lesssim \ p\cdot 
		 \tau_x  \tau_+ |\Phi^{(k)}|^2  \ , \qquad \hbox{where \ \ \ } p\in\mathcal{Z}^{0} \ . \label{S_vect_est}
\end{equation}
A similar bound holds if we replace $\Gamma^{I'}\phi$ by $S \Gamma^{I'}\phi $. 

Likewise, when $X=q S$ where $q=q(u)$
with bounds $|(\tau_-{\partial_u^b})^l q| \lesssim_{l} \tau_- $ we have from
Part \ref{main_LS_estimate}) of Lemma \ref{basic_lie_lem}
the estimates:
\begin{equation}
		\mathcal{S}_X^{i} \ \in  \ \tau_x^{-1} \tau_+ \cdot \mathcal{Z}^\frac{1}{2} \ , \qquad
		\mathcal{S}_X^{u} \ \in \ \tau_x^{-1} \tau_+ \cdot \mathcal{Z}^{\frac{3}{2}}   \ . \notag
\end{equation}
This gives the pointwise estimate:
\begin{equation}
		| \mathcal{S}^{\alpha}_X
		 \partial_\alpha \Gamma^{I'}\phi  \cdot    \Gamma^I\phi | \ \lesssim \ p\cdot 
		 \tau_x\tau_+  |\Phi^{(k)}|^2  \ , \qquad \hbox{where \ \ \ } p\in\mathcal{Z}^{0} \ . \label{S_vect_est'}
\end{equation}

Finally, 
\begin{equation}
		|(\partial_\alpha q ) \mathcal{S}^{\alpha} S \Gamma^{I'}\phi\cdot    \Gamma^{I}\phi |
		\ \lesssim \ p\cdot \tau_x\tau_+ |\Phi^{(k)}|^2 \ , \qquad
		\hbox{where\ \ \ } p \in \mathcal{Z}^0 \ . \label{S_T4_last}
\end{equation}
which follows from $(\partial_\alpha q ) \mathcal{S}^{\alpha}\in \tau_x^{-2}\tau_+\cdot \mathcal{Z}^\frac{1}{2}$
and $\tau_x^{-2}\tau_+   | S \Gamma^{I'}\phi\cdot  \Gamma^{I}\phi|  \lesssim 
\tau_x^{-2} \tau_+    |\phi^{(k)}|^2$.  
With \eqref{S_dividen}--\eqref{S_T4_last} in hand, the remainder of the proof of \eqref{SS_IBP}
is essentially identical to the proof of \eqref{K0_IBP} above.
\end{proof}

\begin{proof}[Proof of estimate \eqref{dt_int_comm_est}]
Using Part \ref{L0_multicom_part}) of Corollary \ref{multicom_coro}  we may write:
\begin{multline}
		LHS\eqref{dt_int_comm_est}\ \lesssim\  | 
		 \int_0^T\!\!\!\!\int_{\mathbb{R}^3}\chi_{>R}
		T_1 dV_g |+ \int_0^T\!\!\!\!\int_{\mathbb{R}^3}\chi_{>R}  |T_2|  dxdt\\
		 +
		  \big(  \sup_{0\leq t\leq T}\lp{  \partial\phi(t)}{H^s_x }+
		 \lp{\phi}{\LE^s(r>R)[0,T]} \big)\cdot  \lp{\Box_g\phi }{ (\LE^{*,s}+L^1_t H^s_x)(r>R)[0,T]} \ , \notag
\end{multline}
where:
\begin{equation}
		T_1 \ = \ \sum_{J'\subsetneq J}
		  \nabla_\alpha \mathcal{R}^{\alpha\beta}_{J'}
		 \nabla_\beta \td{\Gamma}^{J'}\phi  \cdot q{\partial_u^b} \td{\Gamma}^J\phi   \ , \qquad
		 T_2 \ = \ \sum_{J'\subsetneq J}
		      \mathcal{S}^{\alpha}_{J'}
		 \partial_\alpha \td{\Gamma}^{J'}\phi  \cdot q {\partial_u^b} \td{\Gamma}^J\phi     \ , \notag
\end{equation}
and where $ \mathcal{R}_{J'}$ and  $ \mathcal{S}_{J'}$ satisfy 
estimates 
\eqref{R_X_sym_bnds} and \eqref{S_X_sym_bnds} (resp) with $a=c=-1$ and $b=1$. We are assuming the weight $q=q(u)$ 
satisfies $|(\tau_-{\partial_u^b})^l q|
\lesssim_{l} 1 $.
The term $T_1$ is therefore handled by estimate \eqref{du_IBP}. On the other hand 
the conditions on $ \mathcal{S}_{I'}$ and inspection
gives the pointwise bound:
\begin{equation}
		|T_2| \ \lesssim \ p\cdot \sum_{|J''|\leq s}\tau_x^{-1}( |\partial \partial^{J''}\phi|^2 
		+ \tau_0^{-1} |\partial_x^b(\partial^b)^{J''}\phi|^2) \ , 
		\qquad \hbox{where \ \ \ } p\in \mathcal{Z}^0 \ . \notag
\end{equation}
This suffices to produce 
$\int_0^T\!\!\!\int_{\mathbb{R}^3} \chi_{>R} |T_2| \lesssim o_R(1)\cdot \lp{\phi}{\LE^s(r>R)[0,T]}^2$.
\end{proof}

\begin{proof}[Proof of estimate \eqref{K0_int_comm_est}]
Note that  the definition of $X$   and the notation of Lemma \ref{gen_comm_lem}
gives the  pointwise bound:
\begin{equation}
		\big|\tau_x^{-1} X ( \tau_x \Gamma^I \phi)\big| \ \lesssim \ \tau_x^\frac{1}{2}\tau_+^2 \tau_0^\frac{1}{2} 
		\cdot (  \Phi^{(k)} + 
		\slash \!\!\!\!\Phi^{(k)})    \ . \notag
\end{equation}
Next, we decompose $\tau_x^{-1} X\tau_x= uS+ X'$    where
$X'=q({\partial_u^b} + (u+2r)r\tau_x^{-1} \partial_r^b \tau_x + ur^2\tau_x^{-2})$.
This leads to the   following  improvement of the previous line:
\begin{equation}
		\big|  X'  ( \Gamma^I \phi)\big| \ \lesssim \   \tau_x^\frac{3}{2}\tau_+ \tau_0^\frac{1}{2} 
		\cdot (  \Phi^{(k)} + 
		\slash \!\!\!\!\Phi^{(k)})   \ . \notag
\end{equation}
Therefore combining \eqref{mult_comm_iden1}, \eqref{Phi_comm}, the previous two lines,
and \eqref{Phi_S1} gives:
\begin{equation}
		LHS\eqref{dt_int_comm_est}  \lesssim   \sum_{i=1,2} | 
		 \int_0^T\!\!\!\!\int_{\mathbb{R}^3}\chi_{>R}
		 T_i dV_g |  +
		 \slp{\sqrt{\tau_x/\tau_+} p \cdot \tau_+(  \Phi^{(k)} + 
		\slash \!\!\!\!\Phi^{(k)})}{L^2[0,T](r>R)}^2+
		\slp{\phi}{S^{1,\infty}_k[0,T]}\cdot \slp{\Box_g\phi }{ N^{1,1}_{k-1} [0,T]} \ , \notag
\end{equation}
where $p^2\in \mathcal{Z}^0 $,  where:
\begin{equation}
		T_1 \ = \ \sum_{I'\subsetneq I}
		  \nabla_\alpha \mathcal{R}^{\alpha\beta}_{I'}
		 \nabla_\beta \Gamma^{I'}\phi  \cdot \td{q}S \Gamma^I\phi  \ , \qquad
		 T_2 \ = \ \sum_{I'\subsetneq I}
		      \mathcal{S}^{\alpha}_{I'}
		 \partial_\alpha \Gamma^{I'}\phi  \cdot \td{q}S  \Gamma^I\phi     \ , \label{Ii_line}
\end{equation}
and where $\td{q}=uq(u)$ satisfies the assumptions of Lemma \ref{IBP_lem}.
The proof of  \eqref{K0_int_comm_est} is concluded by an application of estimate \eqref{K0_IBP} to handle the
contribution of $T_1$, estimate \eqref{SS_IBP} to handle the contribution of $T_2$, and Lemma \ref{L2L00_lem} followed by
\eqref{Phi_S1} and \eqref{Phi_E1} which together show:
\begin{equation}
		\lp{\sqrt{\tau_x/\tau_+} p \cdot \tau_+(  \Phi^{(k)} + 
		\slash \!\!\!\!\Phi^{(k)})}{L^2[0,T](r>R)}^2 \ \lesssim \ o_R(1)\cdot \big( \lp{\phi}{S^{1,\infty}_k[0,T]}^2  
		+\sup_{0\leq t\leq T}\lp{\phi(t)}{E^1_k }^2\big) \ . \notag
\end{equation}
\end{proof}


\section{$L^\infty$   Estimates}\label{L00_sect}
The purpose of this section is to prove Theorem \ref{glob_sob_thm}.
In fact we will prove the slightly stronger bound:

\begin{prop}[Fixed time global Sobolev inequality] \label{main_L00_prop}
Let $R_1\geq 1$ be large enough so that $\mathcal{K}\subseteq \{r<R_1\}$. Then there exists $R\geq R_1$
sufficiently large such that given any $k\geq 1$ one has the following fixed time estimate uniform in $t\geq 0$:
\begin{multline}
		\sum_{i+|J|\leq k} 
		\lp{ \tau_+^\frac{3}{2} \tau_0^\frac{1}{2}(\tau_-{\partial_u^b})^i
		(\tau_x\partial_x^b)^J\phi(t) }{L^\infty_x}  \lesssim  \lp{\tau_+^\frac{3}{2}\phi^{(k)}(t)}{H^2_{x}(r<R) }
		+ \lp{\phi (t)}{E^1_{k+1}}\\
		  +  \sum_{i+|J|\leq k}\lp{\tau_+^\frac{3}{2} \tau_x^\frac{1}{2}\tau_0  (\tau_-{\partial_u^b})^i
		(\tau_x\partial_x^b)^J\Box_g\phi (t)}{   \ell^1_r L^2_x(r<t)\cap L^2_x} \ . \label{main_L00_est'}
\end{multline}
In the above estimate the implicit constant depends on $R$.
\end{prop}

We first give a quick
demonstrate how Proposition \ref{main_L00_prop}  produces Proposition \ref{glob_sob_thm}.

\begin{proof}[Proof that  \eqref{main_L00_est'} implies \eqref{main_L00_est}]
By an application of   \eqref{w_trace3} and
the identity $\tau_-{\partial_u^b} = S + q\partial$, where $|\partial^J q|\lesssim \tau_x$, we have  uniformly for $0\leq t\leq T$
the bound $ \slp{\tau_+^\frac{3}{2}\phi^{(k)}(t)}{H^2_{x}(r<R) } 
\lesssim_R \sup_{0\leq t\leq T}\slp{\phi(t)}{E^1_{k+1}}+\slp{\phi}{S^{1,\infty}_{k+2}[0,T]}$.

Likewise, by an appropriate combination of \eqref{w_trace1}--\eqref{w_trace3}
we have:
\begin{multline}
		\sup_{0\leq t\leq T}\sum_{i+|J|\leq k}\lp{\tau_+^\frac{3}{2} \tau_x^\frac{1}{2}\tau_0 (\tau_-{\partial_u^b})^i
		(\tau_x\partial_x^b)^J\Box_g\phi (t)}{   \ell^1_r L^2_x(r<t)\cap L^2_x} \ \lesssim \notag\\
		 +\sum_{|J|\leq k}\lp{  \tau_x^2
		(\tau_x \partial)^J\Box_g\phi (0)}{  L^2_x}
		+\sum_{i+|J|\leq k+1}\lp{ (\tau_-{\partial_u^b})^i
		(\tau_x \partial_x^b)^J\Box_g\phi}{  N^{1,1}[0,T]} \ . \notag
\end{multline}
This completes the proof.
\end{proof}


\subsection{Reduction of Proposition \ref{main_L00_prop}}

The proof of \eqref{main_L00_est'}  will rest on previous material
and the following three lemmas. In each of these $t$ is a fixed parameter and 
$\phi=\phi(t)$ only depends on $x$. 
We also assume $R>0$ is chosen as in Proposition \ref{main_L00_prop}.

\begin{lem}[Basic $L^\infty$ estimates]\label{L00_lem}
Let  $\phi$ be a test function supported in $r<\frac{3}{4}\la t \ra$. Then one has:
\begin{equation}
		\lp{ (\tau_x\partial_x\phi,\phi^{(1)})}{L^\infty_x} \ \lesssim \  
		\lp{\tau_x^\frac{1}{2} (\partial_x^2  \phi^{(1)} , \tau_x^{-1}\partial_x
		  \phi^{(1)} , \tau_x^{-2} \phi^{(1)} )}{\ell^\infty_r L^2_x}  \ .  \label{loc_L00}
\end{equation}
On the other hand, without any support conditions imposed on $\phi$ we have:
\begin{equation}
		\lp{\tau_x^\frac{3}{2}\tau_0^\frac{1}{2}\phi }{L^\infty_x} \ \lesssim \ 
		\sum_{l + |J|\leqslant 2}\lp{(\tau_x\tau_0{\partial_u^b})^l(\tau_x\partial_x^b)^J \phi}{L^2_x}
		\ . \label{global_L00}
\end{equation}
\end{lem}

and:

\begin{lem}[Global elliptic estimate]\label{global_ell_lem}
Let $\phi$ be a function supported in $r<\frac{3}{4}\la t \ra$. Then  for $R$ sufficiently large
one has the following fixed time estimate: 
\begin{equation}
		\lp{\tau_x^\frac{1}{2} (\partial_x^2 \phi, \tau_x^{-1}\partial_x\phi, \tau_x^{-2}\phi)}{\ell^\infty_r L^2_x }  \lesssim 
		\lp{ \phi }{H^2_{x}(r<R) } 
		+ \lp{\tau_x^\frac{1}{2} ( \partial{\partial_u^b}\phi, \tau_x^{ -1}{\partial_u^b}\phi , \tau_+^{-1} \partial\phi)}{\ell^1_r L^2_x}
		+ \lp{\tau_x^\frac{1}{2} \Box_g\phi }{\ell^1_r L^2_x}  \ , \label{global_elliptic}
\end{equation}
where the implicit constant depends on $R$.
\end{lem}

and:

\begin{lem}[Fixed time commutator estimate]\label{fixed_comm_lem}
Let $\phi$ be a function which is supported in the region $r<\frac{3}{4}\la t \ra$,
and let $\Gamma^I$ denote a product of vector fields in
$\mathbb{L}=\{S,\Omega_{ij},
{\partial_u^b},\partial_i^b-\omega^i{\partial_u^b}\}$ with  
$|I|=k\geqslant 1$.
Then  for $R$ as in Lemma \ref{global_ell_lem} one has the   following fixed time estimate:
\begin{equation}
                \lp{ \tau_x^\frac{1}{2} [\Box_g, \Gamma^I] \phi}{\ell^1_r L^2_x}
                \ \lesssim \ \lp{   \phi^{(k-1)}}{H^{2}_{x}(r<R)} 
                + \la t\ra^{-\frac{3}{2}}\lp{\phi}{\CE^1_{k}} + \lp{  \tau_x^\frac{1}{2}
               (\Box_g\phi)^{(k-1)}}{\ell^1_r L^2_x} \ , 
                 \label{fixed_time_comm}  
\end{equation}
 where the implicit constant again depends on $R$.
\end{lem}

We postpone the proofs of these in order to first establish Proposition \ref{main_L00_prop}.  

\begin{proof}[Proof of \eqref{main_L00_est'}] 
We estimate the timelike and null/spacelike regions regions separately.

\step{1}{Proof of \eqref{main_L00_est'} in $r>\frac{1}{2}\la t \ra$}
Applying estimate \eqref{global_L00}   to $\chi_{r>\frac{1}{2}\la t\ra}(\tau_-{\partial_u^b})^l
(\tau_x\partial_x^b)^J\phi $ followed
by  \eqref{KS_iden_higher} we have:
\begin{align}
		 \sum_{l+|J|\leqslant k} 
		\slp{ \tau_+^\frac{3}{2}\tau_0^\frac{1}{2}(\tau_-{\partial_u^b})^l
		(\tau_x\partial_x^b)^J\phi}{L^\infty_x(r>\frac{1}{2}\la t\ra)}  &\lesssim  
		\sum_{l + |J|\leqslant k+2} \lp{  (\tau_-{\partial_u^b})^l
		(\tau_x\partial_x^b)^J \phi}{ L^2_x } \ , \notag\\
		&\lesssim  \lp{\phi }{\CE^1_{k+1}}+
		\sum_{l+|J|\leqslant k} 
		\lp{\tau_+^\frac{3}{2}\tau_x^\frac{1}{2}\tau_0(\tau_-{\partial_u^b} )^l(\tau_x\partial_x^b)^J 
		\Box_g \phi}{ L^2_x } 
		 . \notag
\end{align}

\step{2}{Reduction of \eqref{main_L00_est'} in $r<\frac{1}{2}\la t \ra$ to truncated functions}
To prove \eqref{main_L00_est'} in the region $r<\frac{1}{2}\la t \ra$ we claim it suffices to show the fixed time estimate:
\begin{equation}
		\sum_{l+|J|\leqslant k}
		\slp{ (\tau_-{\partial_u^b})^l
		(\tau_x\partial_x^b)^J {\psi}}{L^\infty_x}  \lesssim 
		\slp{   \psi^{(k)} }{H^{2}_x(r<R)} 
                + \la t \ra^{-\frac{3}{2}}\slp{\psi}{\CE^1_{k+1}} +  \!\!\!
                \sum_{l+|J|\leqslant k} \slp{  \tau_x^\frac{1}{2}
                (\tau_-{\partial_u^b})^l
		(\tau_x\partial_x^b)^J  \Box_g {\psi} }{\ell^1_r L^2_x} \ , \label{r<1/2t_L00}
\end{equation}
for functions $\psi$  supported in the region $r<\frac{3}{4}\la t \ra$. Indeed, applying \eqref{r<1/2t_L00} to
$\psi=\chi_{r<\frac{1}{2}\la t \ra}\phi$ and multiplying the result by $\la t\ra^\frac{3}{2}$ we have shown
\eqref{main_L00_est'} in   $r<\frac{1}{2}\la t \ra$ after using
$ \lp{\chi_{r<\frac{1}{2}t}\phi}{\CE^1_{k+1}}\lesssim  \lp{\phi}{\CE^1_{k+1}}$ as well as:
\begin{align}
		 \sum_{l+|J|\leqslant k} \lp{\tau_+^\frac{3}{2} \tau_x^\frac{1}{2}
                (\tau_-{\partial_u^b})^l
		(\tau_x\partial_x^b)^J  [\Box_g,\chi_{r<\frac{1}{2}\la t \ra}] \phi }{\ell^1_r L^2_x}   
		&\lesssim 
		\sum_{l+|J|\leqslant k+1}  \lp{  
                (\tau_-{\partial_u^b})^l
		(\tau_x\partial_x^b)^J   {\phi} }{L^2_x } \ , \notag\\
		&\lesssim \ \lp{\phi}{\CE^1_{k}} + 
		\sum_{l+|J|\leqslant k-1} \lp{ \tau_+^\frac{3}{2}\tau_x^\frac{1}{2}\tau_0
                (\tau_-{\partial_u^b})^l
		(\tau_x\partial_x^b)^J  \Box_g {\phi} }{  L^2_x}
		\ . \notag
\end{align}
On the last line we have again used \eqref{KS_iden_higher}.

\step{3}{Reduction  of \eqref{r<1/2t_L00} to the case $k=1$}
Using \eqref{KS_iden_higher} we have:
\begin{equation}
		\hbox{(LHS)}\eqref{r<1/2t_L00}   \lesssim  
		\sum_{l+|J|\leqslant 1}
		\lp{ (\tau_-{\partial_u^b})^l
		(\tau_x\partial_x^b)^J \psi^{(k-1)}}{L^\infty_x}  + 
		\sum_{l+|J|\leqslant k-2} \lp{  \tau_x^2
                (\tau_-{\partial_u^b})^l
		(\tau_x\partial_x^b)^J  \Box_g \psi}{L^\infty_x} . \notag
\end{equation}
On the other hand for functions $\psi$ supported in $r<\frac{3}{4}\la t\ra$
estimate \eqref{global_L00}  gives:
\begin{equation}
		\sum_{l+|J|\leqslant k-2} \lp{ \tau_x^2
                (\tau_-{\partial_u^b})^l
		(\tau_x\partial_x^b)^J  \Box_g \psi }{L^\infty_x} \ \lesssim \ 
		\sum_{l+|J|\leqslant k} \lp{ \tau_x^\frac{1}{2}
                (\tau_-{\partial_u^b})^l
		(\tau_x\partial_x^b)^J  \Box_g \psi }{L^2_x} \ . \notag
\end{equation}
Therefore, with the help of \eqref{fixed_time_comm} we have  reduce 
\eqref{r<1/2t_L00} to showing:
\begin{equation}
		\sum_{l+|J|\leqslant 1}
		\lp{ (\tau_-{\partial_u^b})^l
		(\tau_x\partial_x^b)^J \psi }{L^\infty_x}  \ \lesssim \ 
		\lp{ \psi^{(1)} }{H^{2}_{x}(r<R) }
		+\la t\ra^{-\frac{3}{2}}\lp{\psi}{\CE^1_{2}} + \lp{ \tau_x^\frac{1}{2} 
		\Box_g  \psi^{(1)} }{\ell^1_r L^2_x}
		\ , \label{r<1/2t_L00'}
\end{equation}
for functions $\psi$ supported in $r<\frac{3}{4}\la t \ra$.

\step{4}{Proof of \eqref{r<1/2t_L00'}}
As a first step we have for $\psi$ supported in  $r<\frac{3}{4}\la t \ra$:
\begin{equation}
		\sum_{l+|J|\leqslant 1}
		\lp{ (\tau_-{\partial_u^b})^l
		(\tau_x\partial_x^b)^J \psi }{L^\infty_x}  \ \lesssim \ 
		\lp{(\tau_x\partial_x\psi,\psi^{(1)})}{L^\infty_x} \ , \notag
\end{equation}
which follows from Remark \ref{coord_rem} and by writing  $S=(u-ru_r){\partial_u^b} + r\partial_r$ 
and $\partial_i^b=\partial_i-u_i{\partial_u^b}$.  

Next, concatenating \eqref{loc_L00}
and \eqref{global_elliptic}   we have:
\begin{equation}
		\slp{ (\tau_x \partial_x\psi ,\psi^{(1)})}{L^\infty_x}
		 \lesssim  \slp{  \psi^{(1)} }{H^2_x(r<R)} 
		+ \la t\ra^{-1}\slp{ 
		\tau_x^\frac{1}{2} ( \partial u {\partial_u^b}\psi^{(1)}  , \tau_x^{ -1}
		u{\partial_u^b} \psi^{(1)}  ,  \partial \psi^{(1)} )}{\ell^1_r L^2_x}
		+ \slp{  \tau_x^\frac{1}{2} \Box_g  \psi^{(1)}  }{\ell^1_r L^2_x}  . \notag
\end{equation}
Here we have used that $| \la t\ra^{-1}  u|\approx 1$ on support of $\psi$, as well as $|[\partial,u]|\lesssim 1$.

Finally, we use the expansion $u{\partial_u^b}=S+q\partial$, where $|\partial^J q|\lesssim \tau_x^{1-|J|}$ on
$r<\frac{3}{4}\la t \ra$, as well as the estimate
$\tau_x^{-\frac{1}{2}}\times$ \eqref{KS_iden} for terms involving $\partial^2$,
which altogether gives:
\begin{equation}
		\lp{ 
		\tau_x^\frac{1}{2} ( \partial u {\partial_u^b}\psi^{(1)}  , \tau_x^{ -1}
		u{\partial_u^b} \psi^{(1)}  ,  \partial \psi^{(1)} )}{\ell^1_r L^2_x} \ \lesssim \ 
		 \la t\ra^{-\frac{1}{2}}\lp{\psi}{\CE^1_2}
		+\la t\ra  \lp{  \tau_x^\frac{1}{2}  \Box_g  \psi^{(1)}  }{L^2_x}  \ . \label{S_trade_bnd}
\end{equation}
This completes the proof of  \eqref{r<1/2t_L00'}, and hence our demonstration of \eqref{main_L00_est'}.
\end{proof}


\subsection{Proof of the supporting lemmas}

We now prove Lemmas \ref{L00_lem}, \ref{global_ell_lem}, and \ref{fixed_comm_lem}.

\begin{proof}[Proof of estimate \eqref{loc_L00}]
First note that a rescaled version of the usual $H^2\to L^\infty$ Sobolev estimates gives:
\begin{equation}
		\lp{ \phi }{L^\infty_x} \ \lesssim \  
		\lp{\tau_x^\frac{1}{2} (\partial_x^2  \phi  , \tau_x^{-1}\partial_x
		  \phi  , \tau_x^{-2} \phi  )}{\ell^\infty_r L^2_x}  \ ,  \notag
\end{equation}
which applied to $\phi^{(1)}$ proves half of \eqref{loc_L00}.

It remains to prove \eqref{loc_L00} for $\tau_x\partial_x \phi$, and this is really only an issue 
where $r>1$. In this case the  result follows from Remark 
\ref{coord_rem} and applying the following global Sobolev estimate to $\partial_x\phi$ for $R\geq 1$:
\begin{equation}
		\lp{    \phi }{L^\infty_x(\frac{1}{2}R< r< 2R)}  \lesssim   \sum_{|J|\leq 1}
		R^{-\frac{1}{2}}\lp{  (\partial_x  \Omega^J \phi  , R^{-1} \Omega^J
		  \phi    )}{  L^2_x (\frac{1}{4}R< r< 4R)}  \ , \quad
		  \hbox{where \ \ } \Omega\in \{ x^i\partial_j-x^j\partial_i \} \ . 
		   \notag
\end{equation}
For $R=1$ this last bound follows by changing over to polar coordinates and using separate (fractional)
Sobolev embeddings in the radial and angular variables. For $R>1$ it follows by rescaling to $R=1$.
\end{proof}

\begin{proof}[Proof of estimate \eqref{global_L00}]
It suffices to consider the region $\frac{1}{2}\la t \ra<r<\frac{3}{2}\la t \ra$ and $r\gg 1$,
as the complementary bound  
follows by rescaling the $H^2\to L^\infty$ Sobolev embedding.
Using dyadic cutoffs we may further
assume $\phi$ is supported where $\tau_-\approx 2^k$ and $\tau_x\approx 2^j$, and by using
angular sector cutoffs in the $x$ variable we may further restrict this support 
to a $\frac{\pi}{4}$ wedge about the $x^1$ axis.  

Next, we introduce the  following variables on $t=const$, $r\gg 1$, and the
$\frac{\pi}{4}$ wedge about the $x^1$ axis:
\begin{equation}
		y^1\ =\ 2^{-k}u \ , \qquad
		y^2\ =\ 2^{-j}x^2 \ , \qquad
		y^3\ =\ 2^{-j}x^3 \ . \notag
\end{equation}
Changing variables we have the formulas on $t=const$:
\begin{equation}
		\partial_{y^1} \ = \ 2^k\big({\partial_u^b} + \frac{1}{u_{x^1}}\partial_{x^1}^b\big) \ , \qquad
		\partial_{y'} \ = \ 2^j\big(\partial^b_{x'}-\frac{u_{x'}}{u_{x^1}}\partial_{x^1}^b\big) \ , \qquad
		2^k2^{2j}dy \ = \ |u_{x^1}|dx \ , \notag
\end{equation}
where the derivatives $u_{x^i}$ are also with respect to $t=const$, and $y'=(y^2,y^3)$.
In particular by condition \eqref{u_sym_bnd} there exists coefficients $c^i_\alpha$ which are uniformly
bounded where $\tau_-\approx 2^k$, $\tau_x\approx 2^j$, and within the
$\frac{\pi}{4}$ wedge about the $x^1$ axis, such that 
$\partial_{y^i}=\sum_\alpha c^i_\alpha e_\alpha$ where
 $e_\alpha\in \{\tau_- {\partial_u^b}, \tau_x\partial_x^b\}$. 
 
 By the change of measures
 formula on the previous line we have:
 \begin{equation}
		2^{\frac{1}{2}k+j}\sum_{|I|\leqslant 2}\lp{e^I \phi }{L^2(dy)} \ \approx \ 
		\sum_{a + |\beta|\leqslant 2}\lp{(\tau_x\tau_0{\partial_u^b})^l(\tau_x\partial_x^b)^J \phi}{L^2(dx)}
		\ . \notag
\end{equation}
To finish the proof it suffices to establish $H^2\to L^\infty$
Sobolev estimates in the $y$ coordinates in terms of  the $e_\alpha$ vector fields.
Since the coefficients $c^i_\alpha$ are possibly very rough with respect to the $y$ variable 
we do this by concatenating $H^1$ Sobolev embeddings in the following way:
\begin{equation}
		\lp{\phi }{L^\infty(dy)} 
		\ \lesssim \ \sum_{|I|\leq 1} \lp{\partial_y^I \phi}{L^6(dy)} 
		\ \lesssim \ \sum_{|I|\leq 1} \lp{e^I \phi}{L^6(dy)} 
		\ \lesssim \ \sum_{|J|\leq 1} \sum_{|I|\leq 1}\lp{\partial_y^J e^I \phi}{L^2(dy)}
		 \ \lesssim \ \sum_{|I|\leq 2} \lp{e^I \phi}{L^2(dy)} \ . \notag
\end{equation}
This completes the proof of \eqref{global_L00}.
\end{proof}

\begin{proof}[Proof of estimate \eqref{global_elliptic}]
Following Remark \ref{coord_rem} we see that
the metric in $(t,x^i)$ coordinates satisfies:
\begin{equation}
		\lp{(\la t\ra\partial_t)^l(\tau_x\partial_x)^J(g^{\alpha\beta}-\eta^{\alpha\beta})}{\ell^1_r L^\infty(r<\frac{3}{4}\la t \ra)} 
		\ \lesssim_{l,J} \ 1 \ , \notag
\end{equation}
where $\eta=diag(-1,1,1,1)$ is the standard Minkowski metric.
This gives  the pointwise estimate:
\begin{equation}
		|\Delta \phi| \ \lesssim\ q \cdot|(\partial_x^2\phi,\tau_x^{-1}\partial_x \phi)|+ |\partial {\partial_u^b} \phi|
		+\tau_x^{-1}|{\partial_u^b}\phi| + \tau_+^{-1}|\partial\phi|+|\Box_g\phi| \ , \qquad
		\hbox{where\ \ \ } q\in\ell^1_r L^\infty
		\ , \notag
\end{equation}
and where $\Delta$ is the standard 3D Laplacian.
Therefore, by choosing $R$ sufficiently large we see that to prove \eqref{global_elliptic}
it suffices to show:
\begin{equation}
		\lp{\tau_x^\frac{1}{2} (\partial_x^2 \phi, \tau_x^{-1}\partial_x\phi, \tau_x^{-2}\phi)}{\ell^\infty_r L^2_x } \ \lesssim \
		\lp{ \phi }{H^2_{x}(r<R_1) } 
		+ \lp{\tau_x^\frac{1}{2} \Delta\phi }{\ell^1_r L^2_x}  \ , \label{full_delta}
\end{equation}
where $\Delta$ is the standard 3D Laplacian and $R_1\geq 1$ is chosen so that $\mathcal{K}\subseteq \{r<R_1\}$.

Next, using the endpoint Hardy estimate
\eqref{end_hardy} and truncating $\phi$ smoothly so it is 
supported away from $\mathcal{K}$, we can reduce 
\eqref{full_delta} to the following global estimate on $\mathbb{R}^3$:
\begin{equation}
		\lp{\tau_x^{-\frac{1}{2}}\partial\Delta^{-1} F}{\ell^1_rL^2_x(\mathbb{R}^3)}+
		\lp{\tau_x^\frac{1}{2}\partial^2\Delta^{-1} F}{\ell^1_rL^2_x(\mathbb{R}^3)} \ 
		\lesssim \ \lp{\tau_x^\frac{1}{2}F}{\ell^1_rL^2_x(\mathbb{R}^3)} \ . \notag
\end{equation}
This last line follows by decomposing $\partial^J\Delta^{-1} F=\sum_{k,j}\chi_j
\partial^J\Delta^{-1} \chi_k F$ where $\chi_j$ is a partition of unity adapted to
dyadic regions $r\approx 2^k,2^j\geq 1$, and  using Young's inequality
to sum over:
\begin{equation}
		\lp{\tau_x^{-\frac{1}{2}} \chi_j\partial\Delta^{-1} \chi_k F}{L^2_x(\mathbb{R}^3)}+
		\lp{\tau_x^\frac{1}{2} \chi_j\partial^2\Delta^{-1} \chi_k F}{L^2_x(\mathbb{R}^3)}\ \lesssim \ 
		2^{j}2^k 2^{-2\max\{j,k\}}\lp{\tau_x^\frac{1}{2}  \chi_k F}{L^2_x(\mathbb{R}^3)} \ ,
		 \notag
\end{equation}
for $j,k\geq 0$.
This final  estimate follows from standard $L^2$ fractional/singular integral bounds.  
\end{proof}

Finally we prove the commutator estimate:

\begin{proof}[Proof of estimate \eqref{fixed_time_comm}]
 Using \eqref{mult_comm_iden1}  in $r<\frac{3}{4}\la t \ra$, we have for 
$\phi$ supported in that region:
\begin{equation}
		 \lp{  \tau_x^\frac{1}{2} [\Box_g, \Gamma^I] \phi}{\ell^1_r L^2_x}
                \ \lesssim \ \lp{  \tau_x^\frac{1}{2}
                (\partial^2\phi^{(k-1)}, 
                \tau_x^{-1}\partial \phi^{(k-1)})}{\ell^\infty_r L^2_x} 
                + \lp{ \tau_x^\frac{1}{2} (\Box_g\phi)^{(k-1)}}{\ell^1_r L^2_x} \ . \notag
\end{equation}
It remains to estimate the first  term on RHS above. 
First, note that by the same reasoning used to establish \eqref{S_trade_bnd} we have:
\begin{equation}
		\lp{ 
		\tau_x^\frac{1}{2} ( \partial  {\partial_u^b}\phi^{(k-1)}  , \tau_x^{ -1}
		{\partial_u^b} \phi^{(k-1)} ,\tau_+\partial\phi^{(k-1)}  )}{\ell^1_r L^2_x} \ \lesssim \ 
		 \la t\ra^{-\frac{3}{2}}\lp{\phi}{\CE^1_k}
		+   \lp{  \tau_x^\frac{1}{2}  \Box_g  \phi^{(k-1)} }{L^2_x}  \ , \notag
\end{equation}
for functions $\phi$ supported in  $r<\frac{3}{4}\la t \ra$. 
In addition we have by applying  \eqref{global_elliptic}  to $\phi^{(k-1)}$ and then using the
previous bound to handle the resulting middle term on RHS\eqref{global_elliptic} the bound:
\begin{equation}
		 \lp{  \tau_x^\frac{1}{2}
                (\partial_x^2\phi^{(k-1)}, 
                r^{-1}\partial_x \phi^{(k-1)})}{\ell^\infty_r L^2_x} \ \lesssim \ 
                \lp{   \phi^{(k-1)} }{H^2_{x}(r<R)} + \la t\ra^{-\frac{3}{2}} \lp{ \phi}{\CE^1_{k}}
		+\lp{  \tau_x^\frac{1}{2}  \Box_g  \phi^{(k-1)} }{\ell^1_r L^2_x} \ . \notag
\end{equation}
Combining the last three lines and using \eqref{u_sym_bnd} to handle $[{\partial_u^b},\partial]$
we see that estimate \eqref{fixed_time_comm} follows via induction on $k$.
\end{proof}


\section{Estimates for Nonlinear Problems}\label{nonlin_sect}

This section is devoted to the proof of Theorem \ref{main_nonlin_thm_est}.


\subsection{Proof of the $N_k\to S_k$ mapping property}

This section is devoted to the first half of Theorem \ref{main_nonlin_thm_est}.

\begin{proof}[Proof of \eqref{NL_NS_bnd}]
We treat each component of the norm separately.

\case{1}{$L^\infty H^s$ and $\WLE$ components}
These are handled via an induction on the index $j$. For $j=0$  use
\eqref{LE_null_est_k}. For  $j\geq 1$ we again use \eqref{LE_null_est_k}
and estimate the middle term on the RHS via:
\begin{equation}
		 \lp{\phi}{H^{16 +3(k-j)}_{j-1}(r<B_0)[0,T]} \ \lesssim \ 
		 \lp{\phi}{\WLE^{13 +3(k-j+1)}_{j-1}[0,T]} \ . \notag
\end{equation}

\case{2}{$S^\frac{1}{2}$ components}
This term is handled by \eqref{Sa_est_k} with $a=\frac{1}{2}$. The middle term on RHS 
of \eqref{Sa_est_k} is bounded by the output of the previous step at level $j=k+4$ as follows:
\begin{equation}
		\lp{\tau_+^{-\frac{1}{2}} \phi }{ H^{1}_{k+4}(r<B_\frac{1}{2})[0,T]} \ \lesssim \ 
		\lp{\phi}{\WLE^{1}_{k+4}[0,T]}  \ . \notag
\end{equation}

\case{3}{$S^{1,\infty}$ and $E^1$ components}
This follows from \eqref{CE_est_k}. The middle term on RHS 
of \eqref{CE_est_k} is bounded by the output of the previous step as follows:
\begin{equation}
		\lp{\  \phi }{ \ell^1_t H^{1}_{k+3}(r<B_\frac{1}{2})[0,T]} \ \lesssim \ 
		\lp{\phi}{S^\frac{1}{2}_{k+3}[0,T]}  \ . \notag
\end{equation}
Note also that expanding vector fields via the basis $\partial^b$ shows:
\begin{equation}
		\lp{F}{N^{1,1}_{k+2}[0,T]} \ \lesssim \ 
		\sum_{i+|J|\leq k+2}
		\lp{ (\tau_-{\partial_u^b})^i
		(\tau_x \partial_x^b)^J F}{  N^{1,1}[0,T]} \ . \notag
\end{equation}

\case{4}{$L^\infty$ components}
The bound for this term follows at once from \eqref{main_L00_est}.

\end{proof}


\subsection{Proof of the null form estimate}
  
Here we prove the estimate \eqref{NL_SN_bnd}. This may be broken up into a number of 
pieces according the the constituent parts of \eqref{big_S_norm} and \eqref{big_N_norm}.

\begin{prop}[Constituent null form estimates]\label{N_space_prop}
Let $\mathcal{N}=\mathcal{N}^{\alpha\beta}(t,x,\phi)
\partial_\alpha\phi\partial_\beta\phi$ denote a quadratic form satisfying the conditions 
of Definition \ref{NF_defn}. Then there exists locally bounded functions $C_k$,
depending on the $c_k$ from line \eqref{NF_cond}, such that for $k\geq 18$ one has:
\begin{align}
		\sum_{j=0}^{k+4} \lp{\mathcal{N} }{(\WLE^{*,13+3(k-j)}_j+L^1_t H^{13+3(k-j)}_{x,j})[0,T]}  
		 &\lesssim     C_k(\LLp{\phi}{S_k[0,T]}) \LLp{\phi}{S_k[0,T]}^2   \ , \label{NFE1}\\
		\lp{\mathcal{N} }{ N^{\frac{1}{2}}_{k+3}[0,T]}  &\lesssim 
		C_k(\LLp{\phi}{S_k[0,T]})  \LLp{\phi}{S_k[0,T]}^2   \ , \label{NFE2}\\
		\sum_{i+|J|\leq k+2}\!\!\!\!\!\!
		\lp{ (\tau_-{\partial_u^b})^i
		(\tau_x \partial_x^b)^J\mathcal{N}}{  N^{1,1}[0,T]} &\lesssim 
		C_k(\LLp{\phi}{S_k[0,T]})  \LLp{\phi}{S_k[0,T]}\big( \LLp{\phi}{S_k[0,T]}
		\!+\!  \LLp{  \Box_g\phi }{  N_k[0,T]}\big) \ . \label{NFE3}\\
		\sum_{|J|\leq k}\!\! \lp{  \tau_x^2
		(\tau_x \partial)^J \mathcal{N} (0)}{  L^2_x}  &\lesssim  
		C_k(\LLp{\phi}{S_k[0,T]})  \LLp{\phi}{S_k[0,T]}\big( \LLp{\phi}{S_k[0,T]}
		\!+\!  \LLp{  \Box_g\phi }{  N_k[0,T]}\big)
		\ . \label{NFE4}
\end{align}
\end{prop}

\begin{rem}
In the sequel we will prove Proposition \ref{N_space_prop}   assuming 
$\mathcal{N} =\mathcal{N}^{\alpha\beta}(t,x)
\partial_\alpha\phi\partial_\beta\phi$. 
In the more general case when $\mathcal{N}^{\alpha\beta}$ depends on $\phi$ as well,
repeated application of the Leibniz rule leads to higher order products  
of the form   $\Pi_{i=1}^m \phi^{(k_i)} \partial\phi^{(l_1)}  \partial\phi^{(l_2)}$. 
Such cubic and higher order expressions  are much easier to handle via the uniform norms employed in 
this paper (e.g.~they do not require a null structure), and treating them explicitly
only serves to clutter notation. 
\end{rem}

We will prove \eqref{NFE1}--\eqref{NFE4} with the help of the following lemma:

\begin{lem}[Leibniz rules]
One has:
\begin{enumerate}[I)]
		\item 
		 Let $\mathcal{N}$ be a  quadratic satisfying the conditions of Definition \ref{NF_defn}
		 but not depending on $\phi$,
		 and use the notation
		$\mathcal{N} (\phi,\psi)=\mathcal{N}^{\alpha\beta}(t,x)\partial_\alpha\phi \partial_\beta\psi$. 
		Then if $X^I$ denotes a product in $\{{\partial_u^b},\partial_i^b-\omega^i{\partial_u^b},
		S,\Omega_{ij},\tau_-{\partial_u^b}, \tau_x\partial_x^b\}$,
		we have the identity:
		\begin{equation}
				X^I \mathcal{N}(\phi,\psi) \ = \ \sum_{I' +I''\subseteq  I} \mathcal{N}_{I',I''}
				(X^{I'}\phi,X^{I''}\psi)  
				\ , \label{NF_leibniz}
		\end{equation}
		where each $\mathcal{N}_{I',I''}$  
		satisfies the conditions of Definition \ref{NF_defn} as well.
		\item Let $f,g$ be smooth functions compactly supported in both time and in the exterior region
		$ \{r>R_0\}$, where $R_0$ is given in Definition \ref{class_LE_defn}.
		Let $w_a$ be a weight satisfying $|\partial^J w_a|\lesssim_{J} \tau_0^a$. Then
		one has the balanced product estimates:
		\begin{align}
				\lp{w_0 f g}{\WLE^{*,s}} \ &\lesssim \ \lp{\tau_x^\frac{1}{2}\tau_-^\frac{1}{2}f}{\ell^1_r\ell^1_u L^\infty}
				\lp{\tau_-^{-\frac{1}{2}}g}{\ell^\infty_r\ell^\infty_u H^s} + \lp{\tau_x^{-\frac{1}{2}}f}{\ell^\infty_r H^s}
				\lp{\tau_x g}{\ell^1_r L^\infty} \ , \label{LE_prod1}\\
				\lp{w_a f g}{\WLE^{*,s}} \ &\lesssim \ \lp{\tau_x\tau_0^a f}{\ell^1_r L^\infty}
				\lp{\tau_x^{-\frac{1}{2}} g}{\ell^\infty_r H^s} 
				+ \lp{\tau_x^{-\frac{1}{2}}f}{\ell^\infty_r H^s}\lp{\tau_x\tau_0^a  g}{\ell^1_r L^\infty} \ . \label{LE_prod2}
		\end{align}
		In addition  there are also holds the following unbalanced versions:
		\begin{align}
				\lp{w_0 f g}{\WLE^{*,s}} \ &\lesssim \ \lp{\tau_x^\frac{1}{2}\tau_-^\frac{1}{2}\partial^s f}{\ell^1_r\ell^1_u L^\infty}
				\lp{\tau_-^{-\frac{1}{2}}g}{\ell^\infty_r\ell^\infty_u H^s}  \ , \label{LE_prod3}\\
				\lp{w_a f g}{\WLE^{*,s}} \ &\lesssim \ \lp{\tau_x\tau_0^a \partial^s f}{\ell^1_r L^\infty}
				\lp{\tau_x^{-\frac{1}{2}} g}{\ell^\infty_r H^s}  \ , \label{LE_prod4}
		\end{align}
		where we are using the shorthand $|\partial^s f|=\sum_{|J|\leq s}|\partial^J f|$.
\end{enumerate}
\end{lem}

\begin{proof}[Proof of \eqref{NF_leibniz}]
By induction it suffices to prove this identity for a single vector field $X$. Using the notation
$\mathcal{N}_X(\phi,\psi)= X \mathcal{N}(\phi,\psi)- \mathcal{N}(X\phi,\psi) - \mathcal{N}(\phi,X\psi)$, we have
 $\mathcal{N}_X^{\alpha\beta} = X( \mathcal{N}^{\alpha\beta}) - 
\partial_\gamma^b (X^\alpha) \mathcal{N}^{\gamma\beta}- \partial_\gamma^b (X^\beta) \mathcal{N}^{\alpha\gamma}$.
The first bound on line  \eqref{NF_cond} for $\mathcal{N}_X$ follows at once from this identity and 
the assumption of  \eqref{NF_cond} for $\mathcal{N}$. To prove the second bound on line  \eqref{NF_cond} 
for $\mathcal{N}_X$, it suffices to study the
contractions $ \partial_\gamma^b (X^u) \mathcal{N}^{u\gamma}$ and  $ \partial_\gamma^b (X^u) \mathcal{N}^{\gamma u}$.
For $\gamma=u$ the bound is again immediate from  assuming \eqref{NF_cond}. On the other hand for 
$\gamma=i$ we need $|(\tau_-{\partial_u^b})^i (\tau_x \partial_x^b)^J \partial_i^b (X^u)|\lesssim_{i,J} \tau_0$
for each $X\in \{{\partial_u^b},\partial_i^b-\omega^i{\partial_u^b},
S,\Omega_{ij},\tau_-{\partial_u^b}, \tau_x\partial_x^b\}$. This follows from inspection.
\end{proof}

\begin{proof}[Proof of estimates \eqref{LE_prod1}--\eqref{LE_prod2}]
First let both $f,g$ be not only (spacetime) compactly supported in the exterior region $\mathbb{R}^4\setminus \mathbb{R}\times
\mathcal{K}$, but  also supported in
dyadic regions $\tau_x\approx 2^j$ and $\tau_-\approx 2^k$. Then one has the Moser
estimate:
\begin{equation}
		\lp{fg}{H^s} \ \lesssim \ \lp{f}{L^\infty}\lp{g}{H^s} + \lp{f}{H^s}\lp{g}{L^\infty} \ . \notag
\end{equation}
Multiplying this  through by the appropriate combination of $2^j$ and $2^k$ and then summing in  $\ell^1_r(\ell^2_u)$,
we have both \eqref{LE_prod1} and \eqref{LE_prod2}.
\end{proof}

\begin{proof}[Proof of estimates \eqref{LE_prod3}--\eqref{LE_prod4}]
This follows  from distributing the derivatives and  H\"older's inequality.
\end{proof}

\begin{proof}[Proof of estimate \eqref{NFE1}]
Let $\chi_{\mathcal{C}}$ be a smooth cutoff which is $\equiv 1$
on the cylinder $\mathcal{C}=[1,T-1]\times \{r>R_0+1\}$ and
which vanishes for $r<R_0$ and $t\in [0,\frac{1}{2}]\cup[T-\frac{1}{2},T]$. Our plan is to show:
\begin{align}
		\sum_{j=0}^{k+4} \slp{ (1-\chi_\mathcal{C})\mathcal{N} 
		}{ L^1_t H^{13+ 3(k-j)}_{x,j}[0,T]} \ &\lesssim \
		\LLp{\phi}{S_k[0,T]}^2 \ , \label{NFE11}\\
		\sum_{j=0}^{k+4} \slp{ \chi_\mathcal{C} \mathcal{N}  }{\WLE^{*,13+3(k-j)}_j}
		\ &\lesssim \  \LLp{\phi}{S_k[0,T]}^2 \ . \label{NFE12}
\end{align}
These bounds are further broken down into a number of cases.

\case{1}{$L^1_tH^s_{x,j}$ bounds in $r<R_0+1$}
Expanding all derivatives into the product and discarding $\mathcal{N}^{\alpha\beta}$ it suffices to prove
for $j\leq k+4$ that:
\begin{equation}
		\slp{\partial^{I}  \partial \phi^{(j_1)}  \partial^{I'}  \partial \phi^{(j_2)}}{L^1_tL^2_x(r<R_0+1)[0,T]}
		\ \lesssim \ \sLLp{\phi}{S_k[0,T]}^2 \ , \quad \hbox{when}\quad
		j_1+j_2= j \ , \quad |I|+|I'|\leq 13+3(k-j) \ .  
		\label{int_L1L2_bnd}
\end{equation}
There are now two subcases: 

\case{1a}{Evenly split derivatives}
If both $|I|+j_1\leq 10+ 3k -2j$ and $|I'|+j_2\leq 10+3k-2j$, then each factor can absorb 
at least two more derivatives and still go in $L^2(r<R_0+1)[0,T]$.
Thus, after an  $L^2_x\to L^\infty_x$ Sobolev embedding on one factor we may use a product of:
\begin{equation}
		\lp{\partial^{I} \partial \phi^{(j_1)}}{L^2_tL^\infty_x(r<R_0+1)[0,T]}\ \lesssim\ \LLp{\phi}{S_k[0,T]}
		\ , \qquad
		\lp{\partial^{I'} \partial \phi^{(j_2)}}{L^2 (r<R_0+1)[0,T]}\ \lesssim\ \LLp{\phi}{S_k[0,T]} \ . \notag
\end{equation}

\case{1b}{Uneven split}	
The alternative to the previous case 
is that one factor, say the first, is such that $|I|+j_1\geq 11+3k-2j$. But this forces
$|I'|+j_2\leq 2$. In particular for $k\geq 3$ we may use a product of:
\begin{align}
		\lp{\partial^{I} \partial \phi^{(j_1)}}{L^\infty_tL^2_x(r<R_0+1)}\ &\lesssim\ \LLp{\phi}{S_k[0,T]} \ , \notag\\
		\lp{ \partial^{I'} \partial \phi^{(j_2)}}{L^1_t L^\infty_x (r<R_0+1)} \ &\lesssim \ 
		\lp{\tau_+^{\frac{3}{2}}\partial^{I'}\partial \phi^{(j_2)}}{L^\infty (r<R_0+1)}\ \lesssim\
		\LLp{\phi}{S_k[0,T]} \ . \notag
\end{align}
This completes the proof of \eqref{int_L1L2_bnd}.
 
\case{2}{$L^1_tH^s_{x,j}$ bounds in $[0,1]\cup[T-1,T]$}
In this case we   show the analog of \eqref{int_L1L2_bnd}
where the integral is restricted to $([0,1]\cup[T-1,T])\times (\mathbb{R}^3\setminus \mathcal{K})$.
This follows by taking a product of two $L^\infty_t H^s_{x,j}$
bounds after an $L^2_x\to L^\infty_x$ embedding for the factor with the least number of derivatives.

\case{3}{$\WLE^{*,s}$ bounds in $\mathcal{C}$}
Using \eqref{NF_leibniz} it suffices to show for $j\leq k+4$ and $\alpha,\beta=u,1,2,3$:
\begin{equation}
		 \lp{\td{\chi}_\mathcal{C}^2 \mathcal{N}^{\alpha\beta}_{j_1,j_2} 
		\partial_\alpha^b \phi^{(j_1)}  \partial_\beta^b 
		\phi^{(j_2)} }{\WLE^{*,13+3(k-j)}} \ \lesssim \ \LLp{\phi}{S_k[0,T]}^2 \ , \qquad
		\hbox{where}\quad  j_1+j_2=j \ ,
		\label{C_LE_est}
\end{equation}
where $\td{\chi}_\mathcal{C}$ is  also supported in the exterior region $(0,T)\times \{r>R_0\}$.
This estimate is further broken down based on the values of $j_1,j_2$ and $\alpha,\beta$.

\case{3a}{$\max\{j_1,j_2\}\leq k-1 $}
In this case we plan to use \eqref{LE_prod1} and \eqref{LE_prod2} to cleanly distribute the $\partial^I$ derivatives.
To facilitate this  freeze the values of $j_1,j_2$ and define:
\begin{equation}
		f_\alpha \ = \ \td{\chi}_\mathcal{C}\partial_\alpha^b \phi^{(j_1)}  \ , \qquad
		g_\beta \ = \ \td{\chi}_\mathcal{C}\partial_\beta^b \phi^{(j_2)}  \ . \label{fg_alpha_not}
\end{equation}
From \eqref{bondi_to_rec} and the definition of $S_k$ from line \eqref{big_S_norm} and we have:
\begin{equation}
		\lp{\tau_x^\frac{1}{2}\tau_-^\frac{1}{2}f_u }{\ell^1_u\ell^1_r L^\infty}+
		\lp{\tau_x\tau_0f_u }{\ell^1_r L^\infty} +
		\lp{\tau_x f_i}{\ell^1_r L^\infty} +\lp{\tau_-^{-\frac{1}{2}}f_i}
		{\ell^\infty_u\ell^\infty_r H^{s}} 
		+ \lp{\tau_x^{-\frac{1}{2}} f_\alpha }{\ell^\infty_r H^{s}} 
		\ \lesssim \ \LLp{\phi}{S_k[0,T]} \ , \label{f_bounds}
\end{equation}
where $s=13+3(k-j)$,
with an identical set of estimates for the components of $g$.

\case{3a.1}{$uu$ components}
Using \eqref{NF_cond}   it suffices to show:
\begin{equation}
		\lp{w_1 f_u g_u}{\WLE^{*,13+3(k-j)}}\ \lesssim \ \LLp{\phi}{S_k[0,T]}^2 \ , \label{fg_uu}
\end{equation}
where $|\partial^I w_1|\lesssim \tau_0$. This follows from \eqref{LE_prod2} with $a=1$ and \eqref{f_bounds}.

\case{3a.2}{$ui$ and $iu$ components}
Here we need the analog of \eqref{fg_uu} with $w_1$ replaced by weight $w_0$ with $|\partial^I w_0|\lesssim 1$.
This follows from \eqref{LE_prod1} and \eqref{f_bounds}.

\case{3a.3}{$ij$ components}
In this case the analog of \eqref{fg_uu} follows from \eqref{LE_prod2} with $a=0$ and \eqref{f_bounds}.

\case{3b}{$\max\{j_1,j_2\}\geq k $}
In this case from the constraint $j_1+j_2=j\leq k+4$
we must have both $\min\{j_1,j_2\}\leq 4$ and $13+3(k-j)\leq 13$. Without loss of 
generality assume $\min\{j_1,j_2\}=j_1$. Then with the   notation from line \eqref{fg_alpha_not}
and setting $s=13+3(k-j)$, we have if $k\geq 18$ the bounds:
\begin{equation}
		\lp{\tau_x^\frac{1}{2}\tau_-^\frac{1}{2}\partial^s f_u}{\ell^1_u\ell^1_r L^\infty}+
		\lp{\tau_x\tau_0\partial^s f_\alpha }{\ell^1_r L^\infty} +
		\lp{\tau_x \partial^s f_i}{\ell^1_r L^\infty} +\lp{\tau_-^{-\frac{1}{2}}g_i}{\ell^\infty_u\ell^\infty_r H^s} 
		+ \lp{\tau_x^{-\frac{1}{2}} g_\alpha }{\ell^\infty_r H^s} \ \lesssim \ \LLp{\phi}{S_k[0,T]} \ . \notag
\end{equation}
In particular \eqref{C_LE_est} for this case follows from the bound on the  line above,
\eqref{LE_prod3}, and \eqref{LE_prod4}. 
\end{proof}

The proof of \eqref{NFE2} and \eqref{NFE3}  largely boils down to \eqref{NF_leibniz}
and the following lemma:

\begin{lem}
Let $\mathcal{N}$ be any quadratic form which satisfies \eqref{NF_cond}, 
and denote $\mathcal{N} (\phi,\psi)=\mathcal{N}^{\alpha\beta}\partial_\alpha\phi \partial_\beta\psi$.
Suppose $X^I$ is a product of vector fields in  $\{{\partial_u^b},\partial_i^b-\omega^i{\partial_u^b},
S,\Omega_{ij},\tau_-{\partial_u^b}, \tau_x\partial_x^b\}$. Then if
$|I|\leq k-1$ we have the pointwise bound on $[0,T]$:
\begin{equation}
		|\tau_0^\frac{1}{2} \mathcal{N}(X^I \phi,\psi)| \ \lesssim 
		\ \LLp{\phi}{S_k[0,T]}\cdot \tau_x^{-1}\tau_+^{-\frac{3}{2}}\tau_0^{-1} \sum_{i+|J|=1} 
		 |(\tau_0{\partial_u^b})^i(\partial_x^b)^J\psi|  
		\ . \label{NF_L00}
\end{equation}
A similar bound holds with the roles of $\phi$ and $\psi$ reversed.
\end{lem}

\begin{proof}
Expanding the $\mathcal{N}$  using condition \eqref{NF_cond} we have:
\begin{equation}
		|\tau_0^\frac{1}{2} \mathcal{N}(X^I \phi,\psi)| \ \lesssim \ \tau_0^\frac{3}{2} |{\partial_u^b}X^I
		\phi|\cdot |{\partial_u^b} \psi|
		+ \tau_0^\frac{1}{2} |\partial X^I \phi|\cdot |\partial_x^b \psi| 
		+ \tau_0^\frac{1}{2} |\partial_x^b X^I \phi|\cdot |\partial \psi| \ . \notag
\end{equation}
On the other hand inspection of the $L^\infty$ term from the $S_k$ norm defined on
line  \eqref{big_S_norm} shows for $|I|\leq k-1$:
\begin{equation}
		\tau_0^\frac{3}{2} |\partial X^I\phi| + \tau_0^{\frac{1}{2} }  |\partial_x^b X^I\phi|
		\ \lesssim \ \tau_x^{-1}\tau_+^{-\frac{3}{2}}\LLp{\phi}{S_k[0,T]} \ . \notag
\end{equation}
Taking the product of the last two lines yields \eqref{NF_L00}.
\end{proof}

\begin{proof}[Proof of estimate \eqref{NFE2}]
Using \eqref{NF_leibniz} to distribute derivatives, and
splitting into interior and exterior bounds, it suffices to show that when $j_1+j_2=k+3$:
\begin{align}
		 \lp{\tau_+ 
		\mathcal{N}_{j_1,j_2}(\phi^{(j_1)},\phi^{(j_2)})}{\ell^1_tH^1(r<R_0)[0,T]} \ &\lesssim \ \LLp{\phi}{S_k[0,T]}^2
		\ , \label{S1/2_int}\\
		 \lp{ 
		\mathcal{N}_{j_1,j_2}(\phi^{(j_1)},\phi^{(j_2)})}{N^\frac{1}{2}(r>R_0)[0,T]} \ &\lesssim \ \LLp{\phi}{S_k[0,T]}^2
		\ . \label{S1/2_ext}
\end{align}
Note that \eqref{S1/2_int} is stronger than what we need here, and for this bound we can even assume all vector fields
are in the collection $\{{\partial_u^b},\partial_i^b-\omega^i{\partial_u^b},
S,\Omega_{ij},\tau_-{\partial_u^b}, \tau_x\partial_x^b\}$.  We will use this greater generality  in a moment.

\case{1}{Interior estimate}
From the conditions $k\geq 18$ and $j_1+j_2=k+3$
we have $\min\{ j_1,j_2\}\leq k-2$. Then \eqref{S1/2_int} follows by taking the product of:
\begin{equation}
		\lp{\tau_+  \phi^{(k)}}{\ell^1_t L^\infty(r<R_0)[0,T]} \ \lesssim \ 
		\LLp{\phi}{S_k[0,T]} \ , \qquad 
		\lp{ \phi^{(k+4)}}{H^1(r<R_0)[0,T]} \ \lesssim \ \LLp{\phi}{S_k[0,T]} \ . \notag
\end{equation}

\case{2}{Exterior estimate}
Using only $\min\{ j_1,j_2\}\leq k-1$ estimate \eqref{S1/2_ext} follows from \eqref{NF_L00} and:
\begin{equation}
		\sum_{i+|J|=1} \lp{\tau_x^{-\frac{1}{2}}  (\tau_0{\partial_u^b})^i(\partial_x^b)^J\phi^{(k+3)} }{\ell^1_r\ell^1_u L^2[0,T]}
		\ \lesssim \   \lp{\tau_x^{-\frac{1}{2}}  \tau_+^\frac{1}{2}
		(\tau_0{\partial_u^b} \phi^{(k+3)} , \partial_x^b \phi^{(k+3)})  }{\ell^\infty_r L^2[0,T]}
		\ \lesssim \ \LLp{\phi}{S_k[0,T]} \ . \notag
\end{equation}
\end{proof}

\begin{proof}[Proof of estimate \eqref{NFE3}]
Combining \eqref{NF_leibniz}, \eqref{NF_L00}, and the Klainerman-Sideris identity \eqref{KS_iden_higher}
we have:
\begin{multline}
		\sum_{i+|J|\leq k+2}
		\tau_0^\frac{1}{2} | (\tau_-{\partial_u^b})^i
		(\tau_x \partial_x^b)^J\mathcal{N} | \ \lesssim \ \LLp{\phi}{S_k[0,T]}\cdot  
		\Big( \sum_{i+|J|=1}  \tau_x^{-\frac{1}{2}} \tau_+^{-\frac{3}{2}}\tau_0^{-\frac{1}{2}}
		|(\tau_0{\partial_u^b})^i(\partial_x^b)^J\phi^{(k+2)}| \\
		 + \sum_{i+|J|\leq k+2}  \tau_0^\frac{1}{2} | (\tau_-{\partial_u^b})^i
		(\tau_x \partial_x^b)^J\Box_g \phi|
		\Big) \ . \label{NF_KS_iden}
\end{multline}
The proof of \eqref{NFE3} then follows by splitting into interior and exterior estimate as in the proof of \eqref{NFE2}
above. Note that \eqref{S1/2_int} already handles the interior contribution. For the exterior contribution 
we use:
\begin{equation}
		\sum_{i+|J|=1} \lp{  (\tau_0{\partial_u^b})^i(\partial_x^b)^J\phi^{(k+2)} }{\ell^1_u\ell^1_t \ell^1_r L^2[0,T]}
		\ \lesssim \   \lp{\tau_x^{-\frac{1}{2}}  \tau_+
		(\tau_0{\partial_u^b} \phi^{(k+2)} , \partial_x^b \phi^{(k+2)})  }{\ell^\infty_t \ell^\infty_r L^2[0,T]} \ \lesssim \ 
		\LLp{\phi}{S_k[0,T]} \ . \notag
\end{equation}
\end{proof}

\begin{proof}[Proof of \eqref{NFE4}]
This follows by applying \eqref{NF_KS_iden} at $t=0$ with index restriction 
$i+|J|\leq k$.
\end{proof}



\section*{Appendix 1: Coordinates}\label{mod_coords_sect}

In this appendix we discuss some basic consequences of Definition \ref{rad_metrics}, as well
as some simple conditions which guarantee the assumptions of Definition \ref{rad_metrics} hold.

\subsection{Bounds between $(t,x)$ and $(u,x)$ coordinates}

\begin{lem}\label{cov_lem}
Let $u(t,x)$ be a function satisfying the condition \ref{opt_cond}) of Definition \ref{rad_metrics}. 
Then for any smooth function
$q$  and integer $N\geq 0$ one has the following equivalence of symbol type bounds:
\begin{equation}
		\sum_{i+|J|\leq N}
		\lp{ (\tau_-{\partial_u^b} )^i (\tau_x \tau_0 \partial_x^b)^J 
		q }{\ell^1_rL^\infty}  \ \lesssim_{N,q} \ 1 \ \ \Longleftrightarrow \ \ 
		\sum_{i+|J|\leq N}
		\lp{ (\tau_-\partial_t )^i (\tau_x \tau_0 \partial_x)^J q }{\ell^1_rL^\infty}  \ \lesssim_{N,q} \ 1
		\ . \label{eq_sym_bnds}
\end{equation} 
In addition   the change of frame bounds \eqref{bondi_to_rec} also hold. 
\end{lem}

\begin{proof}
Thanks to the change of variables formula:
\begin{equation}
		{\partial_u^b} \ = \ \frac{1}{u_t}\partial_t , \qquad \partial_i^b \ = \ \partial_i - \frac{u_i}{u_t}\partial_t
		\ , \label{cov}
\end{equation}
the implication $\partial_t,\partial_x$ bounds $\Rightarrow$ ${\partial_u^b},\partial_i^b$ bounds 
follows easily  from  (assuming $i+|J|\leq N$):
\begin{equation}
	 u_t>c \ ,  \qquad 
	| (\tau_-\partial_t )^i (\tau_x \tau_0 \partial_x)^J (u_t,u_i)|\lesssim_N 1 \ , \notag
\end{equation}
where the line above itself holds thanks to part \ref{opt_cond}) of Definition \ref{rad_metrics}.
Applying this to $q=\partial u$ we have:
\begin{equation}
	| (\tau_-{\partial_u^b} )^i (\tau_x \tau_0 \partial^b_x)^J (u_t,u_i)|\lesssim_N 1 \ . \notag
\end{equation}
Finally, using this last line and formula \eqref{cov} the implication 
${\partial_u^b},\partial_i^b$  bounds $\Rightarrow$ $\partial_t,\partial_x$ bounds becomes clear.

As a last step notice that \eqref{bondi_to_rec} follows from the formulas \eqref{cov} and estimate 
\eqref{eq_sym_bnds} applied to $q=\partial u$.
\end{proof}

\begin{lem}\label{u_bnd_lem}
Let  $u(t,x)$ be a function satisfying the part \ref{opt_cond}) of Definition \ref{rad_metrics}, 
then   $\tau_+^{-1}(u+\tau_x-t)\in \ell^1_rL^\infty$.
\end{lem}

\begin{proof}
We have $\tau_+^{-1}(u+\tau_x-t)=\tau_+^{-1}\int_0^{(t,x)} \partial(u+\tau_x-t)\cdot ds + O(\tau_+^{-1})$
where $ds$ denotes the line integral along a straight ray from the origin to $(t,x)$. 
Bounding the integral in absolute value gives:
\begin{equation}
		\sup_{r\approx 2^j}|\tau_+^{-1}(u+\tau_x-t)| \ \lesssim \ 2^{-j} + \sum_{k\leq j}2^{k-j}\sup_{r\approx 2^k}
		|\partial(u+\tau_x-t)| \ . \notag
\end{equation}
The assumption
$\partial(u+\tau_x-t)\in \ell^1_rL^\infty$ and Young's convolution inequality finishes the proof.
\end{proof}

\begin{lem}
Let  $u(t,x)$ be a function satisfying the condition \ref{opt_cond}) of Definition \ref{rad_metrics}, 
and suppose that $g$ is a metric satisfying condition \ref{sym_bnds}). Then $g$ is weakly asymptotically
flat in $(t,x)$ coordinates in the sense that:
\begin{equation}
		\lp{(\tau_-\partial_t )^i(\tau_x  \tau_0 \partial_x)^J (g-\eta)^{\alpha\beta} }{\ell^1_r L^\infty} \ < \ \infty \ , \qquad
		\hbox{for all\ \ \ }  (i,J)\in\mathbb{N}\times \mathbb{N}^4 \ ,
		\label{weak_asym_flat1}
\end{equation}
where $\eta=diag(-1,1,1,1)$ is the Minkowski metric.
\end{lem}

\begin{proof}
By Lemma \ref{cov_lem} is suffices to prove the  bound:
\begin{equation}
		\lp{(\tau_-{\partial_u^b} )^i(\tau_x  \tau_0 \partial_x^b)^J 
		(g-\eta)^{\alpha\beta} }{\ell^1_r L^\infty} \ < \ \infty \ , \qquad
		\hbox{for all\ \ \ }  (i,J)\in\mathbb{N}\times \mathbb{N}^4 \ ,
		\label{weak_asym_flat'}
\end{equation}
where $(g-\eta)^{\alpha\beta}$ still denotes the components in $(t,x)$ coordinates. Such estimates
for $(g-\eta)^{ij}$ follow at once from the first inclusion on line \eqref{mod_coords} because
$g^{ij}$ is the same in either $(t,x)$ or $(u,x)$ coordinates. 

For remaining components we compute:
\begin{align}
		g^{ti} \ &= \ (g^{u i}+\omega^i) + \omega_j(g^{ij}-\delta^{ij})
		-g^{\alpha i}\partial_\alpha^b(u+\tau_x-t) \ , \notag\\
		g^{tt}+1 \ &= \ g^{uu}+2(g^{ui}+\omega^i)\omega_i
		+\omega_i\omega_j(g^{ij}-\delta^{ij}) 
		- 2(g^{u\alpha}+\omega_ig^{i\alpha})\partial_\alpha^b(u+\tau_x-t) \notag\\
		&\hspace{.45in}+ g^{\alpha\beta}\partial_\alpha^b(u+\tau_x-t)\partial_\beta^b(u+\tau_x-t) 
		+ \tau_x^{-2}
		\ , \notag
\end{align}
where all metric components on the RHS are now computed in $(u,x)$ coordinates, and where we are using the
notation $\omega^i=\omega_i=x^i\tau_x^{-1}$. In addition to these formulas we also have the estimate:
\begin{equation}
		\lp{(\tau_-{\partial_u^b} )^i(\tau_x  \tau_0 \partial_x^b)^J 
		\partial^b (u+\tau_x-t)  }{\ell^1_r L^\infty} \ < \ \infty \ , \notag
\end{equation}
which itself is a consequence of \eqref{u_sym_bnd}, \eqref{cov}, and Lemma \ref{cov_lem}.
The remaining portion of estimate \eqref{weak_asym_flat'}
 follows from the last three lines above combined with  assumption \eqref{mod_coords}.
\end{proof}

\begin{lem}\label{splice_lem}
Fix  $\delta>0$. Let $u_1$ be an approximate optical 
function satisfying the conditions of Definition \ref{rad_metrics} in the region 
$\la t-r\ra\geq  \delta \la t+r\ra$, and let
$u_2$ be an approximate optical 
function satisfying the conditions of Definition \ref{rad_metrics} in the region 
$\la t-r\ra\leq 2 \delta \la t+r\ra$.
Then if $\chi$ is any cutoff function with $\chi\equiv 1$ on $\la t-r\ra\geq 2\delta \la t+r\ra$,  
$\chi\equiv 0$ on $\la t-r\ra\leq  \delta \la t+r\ra$, and $|(\tau_x \partial)^J\chi|\lesssim 1$,
the function $u=\chi u_1+(1-\chi )u_2$
satisfies   the conditions of Definition \ref{rad_metrics} globally. In particular, in Definition  \ref{rad_metrics}
we may always assume $u=t-\tau_x$ away from the region $\la t-r \ra\ll \la t+r\ra$.
\end{lem}

\begin{proof}
Using Lemma \ref{u_bnd_lem} we have both:
\begin{align}
	\lp{(\la t\ra \partial_t)^i(\tau_x \partial_x)^J\tau_+^{-1}(u_1+\tau_x -t)}
	{\ell^1_r L^\infty(\delta\leq \la t+r\ra^{-1}\la t-r\ra \leq 2\delta)} \ &< \ \infty \ ,  \notag\\
	\lp{(\la t\ra\partial_t)^i(\tau_x \partial_x)^J\tau_+^{-1}(u_2+\tau_x -t)}
	{\ell^1_r L^\infty(\delta\leq \la t+r\ra^{-1}\la t-r\ra \leq 2\delta)} \ &< \ \infty  \ . \notag
\end{align}
Thus
$\tau_+^{-1}(u_1-u_2)$ satisfies the same bound in $\delta\leq \la t+r\ra^{-1}\la t-r\ra \leq 2\delta$, and so
$u=\chi u_1+(1-\chi)u_2$ satisfies \eqref{u_sym_bnd} globally. Note that $\tau_+^{-1} \la u_1 \ra \approx 
\tau_+^{-1} \la u \ra \approx 1$ in $\la t-r \ra \geq \delta \la t+r\ra$ and 
$\tau_+^{-1} \la u_2 \ra \approx 
\tau_+^{-1} \la u \ra \approx 1$ in $\la t-r \ra \leq  \delta \la t+r\ra$ thanks to Lemma \ref{u_bnd_lem}, so the
definition of $\tau_0$ is not affected by splicing $u_1,u_2$.  

It remains to show the bounds \eqref{mod_coords} hold in $(u,x)$ coordinates. Because
$u=u_2$ when $\la t-r \ra\leq \delta \la t+r\ra$, we concentrate on the complementary region.
Here it suffices to show that if $u$ is any function
satisfying \eqref{u_sym_bnd}, and $g$ is any metric satisfying \eqref{weak_asym_flat1},
then one has automatically has the first inclusion on line 
\eqref{mod_coords} restricted to the region $\la t-r \ra\geq  \delta \la t+r\ra$.
Notice that by combining \eqref{weak_asym_flat1} and \eqref{u_sym_bnd},
we see that \eqref{weak_asym_flat1} also holds for all Bondi coordinate components
of $(g^{\alpha\beta}-\eta^{\alpha\beta})$. Therefore, adding and subtracting the tensor $h^{\alpha\beta}$
defined on line \eqref{h_tensor}, our task boils down to showing: 
\begin{equation}
		\lp{(\la t\ra \partial_t )^i(\tau_x    \partial)^J 
		(\eta^{uu},\eta^{ui}+\omega^i)  }
		{\ell^1_r L^\infty(\la t+r\ra^{-1} \la t-r \ra\geq \delta)} \ < \ \infty \ . \notag
\end{equation}
This last line follows from a few simple calculations and the estimate:
\begin{equation}
		\lp{(\la t\ra \partial_t )^i(\tau_x    \partial)^J 
		\partial (u+\tau_x-t)  }
		{\ell^1_r L^\infty(\tau_+^{-1} \la t-r \ra\geq \delta)} \ < \ \infty \ , \notag
\end{equation}
which is an immediate consequence of  \eqref{u_sym_bnd}.
\end{proof}

\subsection{Constructions  for nearly stationary/spherically-symmetric  metrics}

In this section we discuss a simple situation where one can construct an approximate ``optical function''
$u(t,x)$ satisfying conditions \eqref{mod_coords}. This is given by the following definitions.

\begin{defn}
Let $g_{\alpha\beta}$ be a Lorentzian metric on $[0,\infty)\times (\mathbb{R}^{3}\setminus \mathcal{K})$, 
where $\mathcal{K}$ is a compact set. Then:
\begin{enumerate}[i)]
\item 
$g$  is called ``weakly asymptotically flat and quasi-stationary'' if:
\begin{equation}
		  \lp{ \ln^2(1+\tau_x)   (t\partial_t)^i (\tau_x \partial)^J (g_{\alpha\beta}
		-\eta_{\alpha\beta})}{ \ell^1_r L^\infty } \ < \ \infty \ ,
		\qquad \hbox{all \ \ } (i,J) \in \mathbb{N}\times \mathbb{N}^4 \ . \label{q_st}
\end{equation}
Here $\eta=diag(-1,1,1,1)$ is the Minkowski metric in $(t,x)\in \mathbb{R}\times\mathbb{R}^3$ coordinates.

\item
$g$ is called ``quasi-spherical'' if one can write
 $g=g_0+g_1$, where $g_0$ is a spherically symmetric  in $(t,x)$ coordinates, and the remainder $g_1$ satisfies:
\begin{equation}
		  \lp{\tau_x    (\tau_x \partial)^J (g_1)_{\alpha\beta}
		}{ \ell^1_r L^\infty } \ < \ \infty \ ,
		\qquad \hbox{all \ \ } J \in  \mathbb{N}^4 \ . \label{q_sp}
\end{equation}
\end{enumerate}
\end{defn}


\begin{prop}
Let $g_{\alpha\beta}$ be a Lorentzian metric on $[0,\infty)\times (\mathbb{R}^{3}\setminus \mathcal{K})$, 
where $\mathcal{K}$ is a compact set.  Suppose that $g$ is weakly
asymptotically flat and quasi-stationary/spherical
in the sense that \eqref{q_st} and \eqref{q_sp} both hold. Then $g$ satisfies the assumptions of Definition
\ref{rad_metrics} (after a possible redefinition of the $x^i$ coordinates).
\end{prop}



\begin{proof}
We'll prove this in a series of steps.

\step{1}{Preliminary reduction} 
By Lemma \ref{splice_lem} above it suffices to construct an approximate 
optical function $u(t,x)$ satisfying conditions \eqref{u_sym_bnd} and \eqref{mod_coords}
in the region $\la t-r \ra\ll \la t+r\ra$. 
Using a partition of unity we may   extend
$g$ to be the
Minkowski metric in the exterior  $\la t-r \ra\gtrsim \la t+r\ra$. This extension will still satisfy 
\eqref{q_st} and \eqref{q_sp}.

Next, after a possible radial change of variables which preserves both \eqref{q_st} and \eqref{q_sp}, 
we may assume that the area of $t=const$ and $r=const$ with respect to the restriction of $g_0$ is $4\pi r^2$.
In other words we may assume the spherically symmetric part $g_0$
can be written  in polar coordinates as:
\begin{equation}
		g_0 \ = \ (g_0)_{tt}dt^2 + 2(g_0)_{tr}dtdr + (g_0)_{rr}dr^2 + r^2 d\sigma^2 \ , \label{g_0_form}
\end{equation}
where $d\sigma^2$ is the standard round metric on $\mathbb{S}^2$. 

The goal now is to  
construct $u$ in two pieces $u=u_0+u_1$, where $u_0=u_0(t,r)$ is radially symmetric and
corresponds to $g_0$, while the remainder $u_1$ takes into account $g_1=g-g_0$. The requirements
for these two functions will be:
\begin{equation}
		\lp{  \ln(r)  (  r    \partial)^J 
		\partial (u_0+r-t)  }
		{\ell^1_r L^\infty( r\geq R)}  +
		\lp{   r (  r    \partial)^J \partial
		u_1}
		{\ell^1_r L^\infty( r\geq R)}
		\ \lesssim_J \ 1   \ , \label{u_01_est}
\end{equation}
for sufficiently large $R$, and in  addition:
\begin{equation}
		g_0(du_0,du_0) \ = \ 2g_0( du_0,du_1)+g_1(du_0,du_0) \ = \ 0 \ . \label{u_def_eqns}
\end{equation}
Notice that line \eqref{u_01_est} and formulas \eqref{cov}
allows us to freely change $(r\partial)^J$ with $(r\partial^b)^J$
in any estimate we consider.

First suppose that we have achieved both \eqref{u_01_est} and \eqref{u_def_eqns}.
By the assumption \eqref{q_st} and \eqref{u_01_est}, we have:
\begin{equation}
		\slp{ \ln(r)  (  r   \partial )^J 
		( g^{\alpha\beta}-\eta^{\alpha\beta}
		)  }{\ell^1_r L^\infty (r\geq R)}  +
		\slp{ \ln(r)  (  r   \partial )^J 
		( \eta^{uu},\eta^{ui}+\omega^{i}
		)  }{\ell^1_r L^\infty (r\geq R)} 
		\ <  \ \infty \ ,  \notag
\end{equation}
where all components of $(g^{\alpha\beta}-\eta^{\alpha\beta})$ are computed in $(u,x)$ coordinates.
This suffices to give the first inclusion on line \eqref{mod_coords} (note we only need this for
$\la t-r \ra\ll \la t+r\ra$). We remark that the convergence factor
$\ln(r)$ is sufficient to sum in $\ell^1_u$ when $r\approx t\approx 2^j$.

Next, from the explicit form \eqref{g_0_form} and
the identities on line \eqref{u_def_eqns}, we have both:
\begin{equation}
		 \sqrt{|g_0|}g_0^{u_0 i}+r^{-1}x^i  \ = \ 0 \ , \qquad
		 g^{uu}  \ = \ g(du_1,du_1) +2g_1(du_0,du_1)\ , \notag
\end{equation}
where $ \sqrt{|g_0|}$ is computed in $(u_0,x)$ coordinates. Using \eqref{q_sp} and \eqref{u_01_est}
we have:
\begin{equation}
	\lp{ r  (  r    \partial)^J ( \sqrt{|g_0|} - \sqrt{|g|})}
	{\ell^1_r L^\infty( r\geq R)}  \ < \ \infty \ , \notag
\end{equation}
where $\sqrt{|g|}$ is computed in $(u,x)$ coordinates.
Combining the last two lines 
and  \eqref{u_01_est} again gives:
\begin{equation}
		\lp{  r^2 (  r    \partial)^J 
		  g^{uu}  }
		{\ell^1_r L^\infty( r\geq R)}  +
		\lp{ r  (  r    \partial)^J ( \sqrt{|g|}g^{u i}+\omega^i, g^{ui}-\omega^i\omega_j  g^{uj})
		 }
		{\ell^1_r L^\infty( r\geq R)} \ < \ \infty \ , \notag
\end{equation}
which are sufficient to produce the remaining bounds on line \eqref{mod_coords}
(again for $\la t-r \ra\ll \la t+r\ra$).

It remains to construct $u_0$ and $u_1$ such that \eqref{u_01_est} and \eqref{u_def_eqns} hold.


\step{2}{Construction of $u_0$} 
For $g_0$ we have the expression \eqref{g_0_form} where:
\begin{equation}
		\lp{ \ln^2(r)  ( r  \partial )^J 
	\big(g_0^{tt}+1,g_0^{rr}-1,g_0^{tr}\big)  }
		{\ell^1_r L^\infty (r\geq 1)}  
		\ <  \ \infty \ , 
		\qquad \hbox{all \ \ } J\in\mathbb{N}^2
		\ . \notag 
\end{equation}
For the remainder of the construction we only need to work in the $(t,r)$ coordinates.
 
Let $v=v(t,r)$ be any function which solves the radial eikonal equation $g_0^{\alpha\beta}\partial_\alpha v\partial_\beta v=0$. 
From this we define the quantity $\zeta = \partial v-(1,-1)$, where $\partial$ denotes $(t,r)$ derivatives.
As long as $|\zeta| < 1$ the coordinate change $(t,r)\mapsto (u,r)$
is well defined and we have:
\begin{equation}
		\partial_r^b \zeta \ = \ G(t,r,\zeta) \ , \qquad \partial^b = q(\zeta)\partial  , \label{zeta_sys}
\end{equation}
where both $G$ and $q$ are smooth universally defined functions depending only on the $(t,r)$ components of $g_0$
and not on $v$. Moreover  $q=q_0+q_1(\zeta)$ with 
$q_0$ a constant invertible matrix and $q_1=O(\zeta)$ when   $|\zeta|\leq \frac{1}{2}$. 
Finally we have the uniform symbol  bounds:
\begin{equation}
		\lp{    r \ln^2(r)    ( r  \partial )^J (\partial_\zeta)^k G}
		{\ell^1_r L^\infty (r\geq R)}  
		\ \lesssim_{J,k}\  o_R(1) \ , \qquad \hbox{for\ \ } |\zeta|\leq \frac{1}{2} \ . \label{G_sym_bnds}
\end{equation}

Now  let $\td{u}=v$ in the previous construction with initial normalization
$\partial_r\td{u}<0$ and $\td{u}|_{r=R}=t-R$ for sufficiently large $R$, and set
$\td{\zeta}=\partial \td{u}-(1,-1)$. Then $\td{u}$
is globally defined and smooth in $r>R$ because $(1+1)$ Lorentzian metrics have no caustics. We also have
$|\td{\zeta}|\ll 1$  at least  initially close to $r=R$. 
Commuting the equation \eqref{zeta_sys} with vector fields $\partial^b = q(\zeta)\partial$, and applying a
straightforward bootstrapping argument, we may extend this to   uniform bounds 
$|\td\zeta|\ll1$ and $|\partial^J \td\zeta|\lesssim_J 1$  for all $r>R$.

Next, define the outgoing limit:
\begin{equation}
		f(\td{u}) \ = \ \lim_{r\to \infty}  \partial_t \td u \ = \ 1+ \int_R^\infty  G_t(t(\td u, r),r,\td \zeta) dr \ , \notag
\end{equation}
where $G_t$ denotes the $t$ component of $G$.
By the previous paragraph  we have both $|f-1|\ll 1$ and $|\partial_{\td{u}}^j f|\lesssim_j  1$.
Let $F$ solve $F'=\frac{1}{f}$, and  finally set $u_0=F(\td{u})$. Again we denote by $\zeta=\partial u_0-(1,-1)$,
which we remind the reader  solves equation \eqref{zeta_sys} with conditions \eqref{G_sym_bnds}.

By construction we immediately have $g_0^{\alpha\beta}\partial_\alpha u_0 \partial_\beta u_0=0$, $|\zeta|\ll 1$, and 
$|\partial^J\zeta|\lesssim_J 1$.  In addition
to this and the fact that $\partial_t u_0=\frac{1}{f}\td u_t$, we have $\lim_{r\to \infty}\partial_t u_0=1$.
Combining this last piece of information 
with the $r\to \infty$ limits $g_0^{tt}\to -1, g_0^{rr}\to 1, g_0^{tr}\to 0$, as well as  $\partial_r u_0<0$,
we  have the normalization $\zeta\to 0$ as $r\to \infty$. In particular we may write:
\begin{equation}
		\zeta(u_0,r) \ = \   -\int_r^\infty  G (t(  u_0, s),s,\zeta) ds \ . \notag
\end{equation}
Differentiating this identity any number of time with respect to  
$\partial^b=q(\zeta)\partial$, and using  the weighted estimate \eqref{G_sym_bnds} and a straight forward
bootstrapping argument, 
 gives the first bound on line \eqref{u_01_est}.

\step{3}{Construction of $u_1$}
Using the second identity from line \eqref{u_def_eqns},
we have that  the correction $u_1$ solves the linear equation:
\begin{equation}
		g_0^{\alpha\beta} \partial_\alpha u_0 \partial_\beta u_1 \ = \ -\frac{1}{2} g_1^{\alpha\beta}
		\partial_\alpha u_0 \partial_\beta u_0 \ , \qquad
		u_1 \ \to \ 0 \hbox{\ \ \ as \ \ \ } r\to \infty \ . \notag
\end{equation}
Integrating this  in $(u_0,x)$ coordinates we find that:
\begin{equation}
		u_1(u_0,x) \ = \ \frac{1}{2}\int_{|x|}^\infty  (g_1^{u_0 u_0}/g_0^{ru_0}) (u_0,sx/|x|)ds \ . \notag
\end{equation}
By the first estimate on line \eqref{u_01_est} and assumptions \eqref{q_st}--\eqref{q_sp} we have:
\begin{equation}
		\lp{r (r\partial^b )^J (g_1^{u_0 u_0}/g_0^{ru_0})}{\ell^1_r L^\infty (r\geq R)}\ <\ \infty \ , \notag
\end{equation}
where $\partial^b$ are computed in $(u_0,x)$ coordinates.
An application of Young's convolution inequality to the above integral 
gives  line \eqref{u_01_est}
for $u_1$.
\end{proof}


\section*{Appendix 2: Local Energy Decay}

In this appendix we discuss how assumption \eqref{basic_stat_LE} relates to
assumption \eqref{basic_LE} when the metric $g$ enjoys structural properties
similar to the Kerr family of metrics with angular momentum in a moderate range.

Let $\mathcal{T}\subseteq \mathbb{R}^3$ be a compact region contained in the exterior
$\mathbb{R}^3\setminus \mathcal{K}$. We first redefine the norms \eqref{class_LE_norm2}
and \eqref{class_LE_norm4} so the regularity loss occurs only on $\mathcal{T}$:
\begin{align}
		\slp{\phi}{\WLE^s_{class}[0,T]} \ &= \
		\sum_{|J| \leq s} \big( 
		\slp{\tau_x^{-1} \partial^J \phi}{\LE[0,T]} + 
		\slp{ \partial \partial^{J} \phi}{\LE(\mathcal{T}^c)[0,T]} 
		\big) \ , \label{class_LE_norm2'}\\
		\lp{F}{\WLE^{*,s}[0,T]} \  &= \  \sum_{|J| \leq s} \big(
		\lp{ \partial^J F }{\LE^*[0,T]} 
		+ \lp{\partial  \partial^{J}F}{L^2(\mathcal{T})[0,T]} 
		\big)\ , \label{class_LE_norm4'}
\end{align}
With respect to these modified norms we have:

\begin{prop}\label{LE_assum_prop}
Let the norms $\WLE^s_{class}$ and $\WLE^{*,s}_{class}$ be defined as on lines \eqref{class_LE_norm2'} 
and \eqref{class_LE_norm4'}. Suppose  in addition 
that  $\partial_t$ is uniformly timelike on $[0,\infty)\times \mathcal{T}$.
Then estimate:
\begin{equation}
		\sup_{0\leq t\leq T}\
		\lp{\partial\phi(t)}{H^{s}_x} + \lp{\phi}{\WLE^s_{class}[0,T]}  \ \lesssim\  \lp{\partial \phi(0)}{H^{s}_x} 
		+ \lp{\Box_g \phi }{\WLE^{*,s}[0,T]} \ , \label{basic_LE'}\\
\end{equation}
implies estimate:
\begin{align}
		\sup_{0\leq t\leq T}\!\!
		\slp{\partial\phi(t)}{H^{s}_x}\!+\! \slp{\phi}{\LE^s_{class}[0,T]}   \ \lesssim\   \slp{\partial \phi(0)}{H^{s}_x} 
		+ \slp{ \phi}{H^s(\td{\mathcal{T}})[0,T]}		
		+ \slp{\partial_t\phi}{H^{s}(\td{\mathcal{T}})[0,T]}
		+ \slp{\Box_g \phi }{\LE^{*,s} [0,T]} \ , \label{basic_stat_LE'}
\end{align}
where $\td{\mathcal{T}}$ is any compact neighborhood of $\mathcal{T}$ (here the implicit constant depends 
on $\td{\mathcal{T}}$).
\end{prop}

\begin{rem}
We do not need to make  other assumptions on the metric $g$ besides 
$\partial_t$ being uniformly timelike on $[0,\infty)\times \mathcal{T}$. In particular the time variation of
$g$ plays no role in establishing \eqref{basic_stat_LE'} from \eqref{basic_LE'}. 

Note that \eqref{basic_stat_LE'} implies
\eqref{basic_stat_LE} when $s=0$. For $s>0$ a simple induction allows us to reduce the second term in RHS\eqref{basic_stat_LE'}
to $\slp{ \phi}{L^2(\td{\mathcal{T}})[0,T]}$.
\end{rem}

\begin{proof}
Let $\mathcal{T}\subset\!\subset \mathcal{T}'\subset\!\subset \td{\mathcal{T}}$ where $\mathcal{T}'$
is an intermediate compact neighborgood. Without loss of generality we may assume
that $\td{\mathcal{T}}$ is a small enough   that $\partial_t$ is still uniformly timelike 
on it. In particular we have that $P(x,D)=\frac{1}{\sqrt{|g|}}\partial_i \sqrt{|g|} g^{ij}\partial_j $ 
is uniformly elliptic on $\td{\mathcal{T}}$, where the $ij$ contraction is only on the spatial variables.

Estimate \eqref{basic_stat_LE'} follows from adding together the  bounds:
\begin{equation}
		\sup_{0\leq t\leq T}\lp{\partial \phi(t)}{H^s_x((\mathcal{T}')^c)} +
		 \lp{\phi}{\LE^s_{class}((\mathcal{T}')^c)[0,T]}  \ \lesssim\  \lp{\partial \phi(0)}{H^{s}_x} +
		 \lp{\phi}{H^{s+1}(\mathcal{T}')[0,T]}
		+ \lp{\Box_g \phi }{\LE^{*,s}[0,T]} \ , \label{trunc_LE}
\end{equation}
\begin{equation}
		\sup_{0\leq t\leq T}\!\!\! \slp{\partial \phi(t)}{H^s_x(\mathcal{T'})} +
		\slp{  \phi}{H^{s+1}(\mathcal{T'})[0,T]} \lesssim 
		\slp{\partial \phi(0)}{H^{s}_x} + 
		\slp{\phi}{H^s(\td{\mathcal{T}})[0,T]} 
		+ \slp{\partial_t \phi}{H^s(\td{\mathcal{T}})[0,T]} + \slp{\Box_g \phi }{\LE^{*,s}[0,T]} \ . \label{interior_LE}
\end{equation}

The first estimate \eqref{trunc_LE} follows directly by applying \eqref{basic_LE'} to $(1-\chi_\mathcal{T})\phi$, where
$\chi_\mathcal{T}=1$ on $\mathcal{T}$
and $\chi_\mathcal{T}=0$ on $(\mathcal{T}')^c$, and then using the Hardy estimate \eqref{hardy0} with $a=0$
for the boundary term where all derivatives fall on the cutoff.

The second estimate  \eqref{interior_LE} is essentially an elliptic bound. 
Fixing a pair of intermediate regions 
$\mathcal{T}'\subset\!\subset \td{\mathcal{T}}'\subset\!\subset \td{\mathcal{T}}''\subset\!\subset \td{\mathcal{T}}$
one begins with (for any $s\geq 0$):
\begin{equation}
		\lp{  \phi}{H^{s+1}(\td{\mathcal{T}}'')[\epsilon,T-\epsilon]} \ \lesssim \ \lp{\phi}{H^s(\td{\mathcal{T}})[0,T]} 
		+ \lp{\partial_t \phi}{H^{s}(\td{\mathcal{T}})[0,T]} + \lp{\Box_g \phi }{H^{s-1}(\td{\mathcal{T}})[0,T]} \ , \notag
\end{equation}
where $\epsilon>0$ small enough that the domain of dependence of $\mathcal{T}'$ in $[0,\epsilon]\cup [T-\epsilon,T]$
is contained in $\td{\mathcal{T}}'$. One then fills in the slabs $[0,\epsilon]\times \mathcal{T}'$ and 
$[T-\epsilon,T]\times \mathcal{T}'$ with local in time energy estimates, where the second slab estimate propagates
from a localized energy estimate at time $T-\epsilon$ based on the
multiplier $X=\chi_{\td{\mathcal{T}}'}\partial_t$, where $\chi_{\td{\mathcal{T}}'}=1$ on $\td{\mathcal{T}}'$
and $\chi_{\td{\mathcal{T}}'}=0$ on $(\td{\mathcal{T}}'')^c$. Note that all of the error terms generated by doing
this are bounded by the LHS of the expression on the previous line (we may also assume $\chi_{\td{\mathcal{T}}'}=0$
when $t<\epsilon$).
\end{proof}


\section*{Appendix 3: Hardy and trace inequalities}
   
 \begin{lem}[Hardy  inequalities]\label{hardy_lem}
Let $\mathcal{K}\subseteq \mathbb{R}^3$ be compact with
connected complement, and for any other $Q\subseteq \mathbb{R}^3$ define 
$L^2_x(Q)$ where the domain of integration is $Q\setminus \mathcal{K}$, and $L^2(Q)[0,T]$
where the domain of integration is  $[0,T] \times(Q\setminus \mathcal{K}) \subseteq \mathbb{R}^4$.
Then for  test functions $\phi$  one has the following:  
\begin{enumerate}[I)]
	
\item 
For all $R\geq 0$ there holds uniformly:
\begin{align}
		\lp{ \tau_x^{a-1} \phi}{L^2_x (r>R)}  
		\ &\lesssim_a \    \lp{\tau_x^a \partial_x \phi }{L^2_x (r>R)}
		\ , &\hbox{when \ \ } -\frac{1}{2}< &a< \infty \ , \label{hardy0}\\
		\lp{\tau_x^{-\frac{3}{2}}\phi }{\ell^\infty_r L^2_x(r>R)} \ &\lesssim\   
		\lp{\tau_x^{-\frac{3}{2}}\phi }{ L^2_x(\frac{1}{2}R<r<R)}+
		\lp{\tau_x^{-\frac{1}{2}} \partial_x\phi}{\ell^1_r L^2_x(r>\frac{1}{2}R)} \ . \label{end_hardy}
\end{align}

\item 
When $R\geq 1$ is large enough that $\mathcal{K}\subseteq \{r<\frac{1}{2}R\}$ there 
is the fixed time estimate:
\begin{equation}
		\lp{ \phi}{L^2_x(r>R) } 
		\ \lesssim\  \lp{\tau_- \tau_x^{-1}\partial_x(\tau_x\phi)}{L^2_x(r> R)}
		+ R^{-\frac{1}{2}}\lp{  \tau_-^\frac{1}{2}  \phi}{L^2_x(\frac{1}{2}R<r<R) }
		+ R^{\frac{1}{2}}\lp{  \tau_-^\frac{1}{2}  \partial \phi}{L^2_x(\frac{1}{2}R<r<R) }
		\ . \label{hardy2}
\end{equation}

\item
Again for $R\geq 1$  large enough that $\mathcal{K}\subseteq \{r<\frac{1}{2}R\}$ there 
is the spacetime estimate:
\begin{multline}
		\lp{\tau_x^{a-1}\phi(T)}{L^2_x(r>R) }+
		\lp{\tau_x^{a-\frac{3}{2}}\phi }{L^2(r>R)[0,T]}\
		\lesssim_a  \
		\lp{\tau_x^{a-\frac{3}{2}}\partial_r^b(\tau_x \phi)}{L^2(r>\frac{1}{2}R)[0,T]}\\ 
		+ \lp{\tau_x^{a-\frac{3}{2}} \phi}{L^2(\frac{1}{2}R<r<R)[0,T]}  +  \lp{\tau_x^a\partial \phi(0)}{L^2_x(r>\frac{1}{2}R)}
		\ , \qquad \hbox{when\ \ } a<1 \ .  \label{hardy3}
\end{multline}
\end{enumerate}
\end{lem}

\begin{lem}[Weighted trace inequalities]
For $T\geq 0$ and any $R\geq 1$ 
one has the uniform bound:
\begin{subequations}\label{w_trace}
\begin{align}
		\lp{\tau_-^\frac{1}{2}\phi(T)}{L^2_x( \frac{1}{2}R<r<R)} \ &\lesssim \ 
		\lp{(\tau_-{\partial_u^b}\phi, \phi)}{  L^2(\frac{1}{2}R< r<R)[\frac{1}{2}T,T]}
		\ , &R\lesssim T \ , \label{w_trace1}\\
		\lp{\tau_-^\frac{1}{2}\phi(T)}{L^2_x(\frac{1}{2}R<r<R)} \ &\lesssim \ 
		\lp{(\tau_-{\partial_u^b}\phi, \phi)}{  L^2(\frac{1}{2}R<r<R)[0,T]}
		+  \lp{\tau_x^\frac{1}{2} \phi(0)}{  L^2_x(\frac{1}{2}R<r<R) } \ , 
		&R \gg T
		\ , \label{w_trace2}\\
		\lp{\tau_-^\frac{1}{2}\phi(T)}{L^2_x( r<R)} \ &\lesssim \ 
		\lp{(\tau_-{\partial_u^b}\phi, \phi)}{  L^2( r<R)[0,T]}
		+  \lp{\tau_x^\frac{1}{2} \phi(0)}{  L^2_x( r<R) } \ . \label{w_trace3}
\end{align}
\end{subequations}
We note that estimate \eqref{w_trace1} also holds with the restriction $\frac{1}{2}R<r<R$ replaced by
$r<R$.
\end{lem}

\begin{proof}[Proof of \eqref{hardy0}]
Let $R_1\geq 1$ be chosen so that $\mathcal{K}\subseteq \{r<R_1\}$. First we prove the bound
assuming $R\geq R_1$. We have 
$I=\int_{R}^\infty\!\int_{\mathbb{S}^2}\partial_r (r^{2a+1} \phi^2) d\omega dr\leq 0$, and also:
\begin{equation}
		(2a+1)\lp{r^{a-1}\phi}{L^2_x(r>R)}^2 \ \leq \ I + 2\lp{r^{a-1}\phi}{L^2_x(r>R)}
		\lp{r^a \partial \phi}{L^2_x(r>R)} \ , \notag 
\end{equation}
which concludes the proof in this case because $2a+1>0$ and $r^\alpha\approx \tau_x^\alpha$ in $r>1$.

Now suppose $0\leq R<R_1$, where $R_1$ is as above. A standard compactness argument shows:
\begin{equation}
		\lp{\phi}{L^2_x(R<r<R_1)} \ \lesssim  \ \lp{\phi}{L^2_x(R_1<r<2R_1)} 
		+ \lp{\partial\phi}{L^2(R<r<2R_1)} \ , \label{local_hardy}
\end{equation}
where the implicit constant is uniform in $R$. Combining this bound with estimate \eqref{hardy0}
in $r>R_1$ completes the proof.
\end{proof}

\begin{proof}[Proof of \eqref{end_hardy}]
Using   estimates of the form \eqref{local_hardy} it suffices to show  for $R\geq R_1$,
where $R_1$ is as above, that:
\begin{equation}
		\lp{\tau_x^{-\frac{3}{2}}\phi }{\ell^\infty L^2_x(r>2R)}^2 \ \lesssim \ 
		\lp{\tau_x^{-\frac{3}{2}}\phi }{\ell^\infty L^2_x(r>R)}
		\lp{\tau_x^{-\frac{1}{2}}\partial \phi }{\ell^1 L^2_x(r>R)} \ . \notag
\end{equation}
To prove it choose $h_k(r)$ so that $h_k(R)=0$ and $h'_k=2^{-k}\chi_k$, where
 $\chi_k=1$ when $R2^k\leq r\leq R2^{k+1}$ for $k\geq 1$, and  $\chi_k=0$
 when either $r\leq R2^{k-1}$ or $r\geq R2^{k+2}$. Then 
computing the integral 
$I=\int_{R}^\infty\!\int_{\mathbb{S}^2}\partial_r (h_k \phi^2) d\omega dr = 0$ 
and taking $\sup_{k\geq 1}$ of the result yields
the estimate on the line above.
\end{proof}

\begin{proof}[Proof of \eqref{hardy2}]
Thanks to Remark \ref{optical_remark} and the conditions \eqref{u_sym_bnd}
we can assume the coordinate $u$ is chosen so that $-C\leq \partial_r u\leq -\frac{1}{C}$ for some 
fixed $C>0$. Then  \eqref{hardy2} follows from integration of
$\partial_r (\chi_{>R}   u (\tau_x \phi)^2)$ with respect to $drd\omega$. 
\end{proof}

\begin{proof}[Proof of \eqref{hardy3}]
This boils down to  integration of the quantity $\partial_r^b(\chi_{r>R}r^{2a-2} 
(\tau_x\phi)^2)$ with respect to the measure $dudrd\omega$
over the slab $0\leq t\leq T$. By Remark \ref{optical_remark} and the conditions \eqref{u_sym_bnd} we have
the pair of bounds $\frac{1}{C}\leq \partial_r^b t , \partial_t u \leq C$. In particular:
\begin{equation}
		\int_0^T\!\!\!\int_{0}^\infty\!\!\!\int_{\mathbb{S}^2} \partial_r^b(\chi_{r>R}r^{2a-2} (\tau_x\phi)^2) \alpha d\omega dr dt
		= \int_{0}^\infty\!\!\!\int_{\mathbb{S}^2}  \chi_{r>R}r^{2a-2} (\tau_x\phi)^2 \beta d\omega dr \Big|_{0}^T \ , 
		\label{coor_div_iden}
\end{equation}
for some pair of smooth functions $\alpha,\beta\approx 1$. This yields \eqref{hardy3} using the condition $a<1$.
\end{proof}

\begin{proof}[Proof of \eqref{w_trace}]
To prove \eqref{w_trace1}
let $\chi_T(t)=\chi(t/T)$ where $\chi(s)\in C_0^\infty(s> \frac{1}{2})$ with $\chi(1)=1$.
Then \eqref{w_trace1} follows from  integration of ${\partial_u^b}( \tau_- \chi_T \phi^2)$ with respect to 
$dudx$ on the cylinder $0\leq t\leq T$
and   $\frac{1}{2}R<r<R$. In the case of \eqref{w_trace2} and  \eqref{w_trace3}
we use a similar procedure but simply drop the cutoff $\chi_T$.
\end{proof}
  


\end{document}